\documentclass[12pt,draft]{amsart}
\usepackage[mathcal]{eucal}
\usepackage{amssymb,amsmath,amsthm,mathrsfs}
\usepackage[all]{xy}
\usepackage{hyperref}

\usepackage[usenames,dvipsnames]{color}

\theoremstyle{plain}
\newtheorem{thm}{Theorem}[section]
\newtheorem{prop}[thm]{Proposition}
\newtheorem{lem}[thm]{Lemma}
\newtheorem{cor}[thm]{Corollary}

\theoremstyle{definition}

\newtheorem{defn}[thm]{Definition}
\newtheorem{rem}[thm]{Remark}
\newtheorem{obs}[thm]{Observation}
\newtheorem{exa}[thm]{Example}
\newtheorem*{case}{Special Case}

\newtheorem*{definition*}{Definition}

\numberwithin{equation}{section}

\newlength{\mylistleftmargin}
\newlength{\myenumilistleftmargin}
\newcommand{\C}{\mathbb{C}}
\newcommand{\F}{\mathbb{F}}
\newcommand{\N}{\mathbb{N}}
\newcommand{\Z}{\mathbb{Z}}
\newcommand{\Q}{\mathbb{Q}}
\newcommand{\R}{\mathbb{R}}

\newcommand{\Th}{\boldsymbol{\mathcal{T}}}

\newcommand{\fp}{\mathfrak{p}}
\newcommand{\fq}{\mathfrak{q}}

\newcommand{\scX}{\mathscr{X}}
\newcommand{\scY}{\mathscr{Y}}
\newcommand{\scZ}{\mathscr{Z}}

\renewcommand{\epsilon}{\varepsilon}
\renewcommand{\theta}{\vartheta}
\renewcommand{\rho}{\varrho}
\renewcommand{\phi}{\varphi}

\newcommand{\fields}{\mathrm{fields}}
\newcommand{\alg}{\mathrm{alg}}
\newcommand{\unr}{\mathrm{unr}}

\newcommand{\Grass}{\ensuremath{\mathrm{Gr}_\mathrm{Lie}}}

\DeclareMathOperator{\Rad}{Rad}
\DeclareMathOperator{\pr}{pr}
\DeclareMathOperator{\SL}{SL}
\DeclareMathOperator{\GL}{GL}

\DeclareMathOperator{\val}{val}
\DeclareMathOperator{\ac}{ac}

\DeclareMathOperator{\Aut}{Aut}
\DeclareMathOperator{\Hom}{Hom}

\DeclareMathOperator{\Ad}{Ad}
\DeclareMathOperator{\ad}{ad}
\DeclareMathOperator{\gr}{gr}

\DeclareMathOperator{\res}{res}
\DeclareMathOperator{\Res}{Res}
\DeclareMathOperator{\Lie}{Lie}
\DeclareMathOperator{\rk}{rk}

\DeclareMathOperator{\Ind}{Ind}
\DeclareMathOperator{\red}{red}
\DeclareMathOperator{\Spec}{Spec}

\DeclareMathOperator{\Frob}{Frob}
\DeclareMathOperator{\Id}{Id}
\DeclareMathOperator{\Gal}{Gal}
\DeclareMathOperator{\Witt}{Witt}
\DeclareMathOperator{\FWitt}{FWitt}
\DeclareMathOperator{\Gr}{Gr}
\DeclareMathOperator{\Irr}{Irr}
\DeclareMathOperator{\Stab}{Stab}
\DeclareMathOperator{\charac}{char}
\DeclareMathOperator{\supp}{supp}

\newcommand{\nir}[1]{{{#1}}}
\newcommand{\uri}[1]{{{#1}}}
\newcommand{\ben}[1]{{{#1}}}
\newcommand{\cv}[1]{{{#1}}}

\title[Arithmetic Groups, Base Change, and Representation
Growth]{Arithmetic Groups, Base Change,\\ and Representation Growth}

\author{Nir Avni} \address{Department of Mathematics, Northwestern
  University, 2033 Sheridan Rd., Evanston IL 60201, USA}
\email{avni.nir@gmail.com}

\author{Benjamin Klopsch} \address{Mathematisches Institut,
  Heinrich-Heine-Universit\"at, D-40225 D\"usseldorf, Germany}
\email{klopsch@math.uni-duesseldorf.de}

\author{Uri Onn} \address{Department of Mathematics, Ben Gurion
  University of the Negev, Beer-Sheva 84105 Israel}
\email{urionn@math.bgu.ac.il}

\author{Christopher Voll} \address{Fakult\"at f\"ur Mathematik,
  Universit\"at Bielefeld, D-33501 Bielefeld, Germany} 
\email{C.Voll.98@cantab.net}

\begin{document}

\begin{abstract} Consider an arithmetic group~$\mathbf{G}(O_S)$, where
  $\mathbf{G}$ is an affine group scheme with connected, simply
  connected absolutely almost simple generic fiber, defined over the
  ring of $S$-integers $O_S$ of a number field $K$ with respect to a
  finite set of places~$S$.  For each $n \in \N$, let
  $R_n(\mathbf{G}(O_S))$ denote the number of irreducible complex
  representations of $\mathbf{G}(O_S)$ of dimension at most~$n$.  The
  degree of representation growth $\alpha(\mathbf{G}(O_S)) = \lim_{n
    \rightarrow \infty} \log R_n(\mathbf{G}(O_S)) / \log n$ is finite
  if and only if $\mathbf{G}(O_S)$ has the weak Congruence Subgroup
  Property.

  We establish that for every $\mathbf{G}(O_S)$ with the weak
  Congruence Subgroup Property the invariant $\alpha(\mathbf{G}(O_S))$
  is already determined by the absolute root system of~$\mathbf{G}$.
  To show this we demonstrate that the abscissae of convergence of the
  representation zeta functions of such groups are invariant under
  base extensions $K \subset L$.  We deduce from our result a variant
  of a conjecture of Larsen and Lubotzky regarding the representation
  growth of irreducible lattices in higher rank semi-simple groups.
  In particular, this reduces Larsen and Lubotzky's conjecture to
  Serre's conjecture on the weak Congruence Subgroup Property, which
  it refines.
\end{abstract}

\maketitle
\thispagestyle{empty}

% set counter to {1} to avoid listing of subsections:
% set to (3) for full details
\setcounter{tocdepth}{2}

\tableofcontents

%%%%%

\section{Introduction and Main Results}
\subsection{Background and Motivation} \label{sec:intro} One of the
aims of this paper is to prove a variant of a conjecture of Larsen and
Lubotzky on the representation growth of irreducible lattices in
higher rank semi-simple groups.  We recall that the representation
growth of an arbitrary group $G$ is given by the asymptotic behavior
of the sequence $R_n(G)$, $n \in \N$, where $R_n(G)$ denotes the
number of equivalence classes of irreducible complex representations
of $G$ of dimension at most~$n$.  Whenever $G$ is a topological
\ben{(resp.\ algebraic)} group, we restrict the investigation without
further comment to continuous \ben{(resp.\ rational)} representations.
According to Margulis' Arithmeticity Theorem, the lattices in question
arise in the following way.  Consider an arithmetic group
$\mathbf{G}(O_S)$, where $O_S$ is the ring of $S$-integers of a number
field $K$ with respect to a finite set of places $S$ of $K$ and
$\mathbf{G}$ is an affine group scheme over $O_S$ whose generic fiber
is connected, simply connected absolutely almost simple.  Suppose
further that the $S$-rank of $\mathbf{G}$, i.e., $\rk_S \mathbf{G} =
\sum_{v\in S} \rk_{K_v} \mathbf{G}$, is at least~$2$ and that an
infinite place $v$ is included in $S$ if $\rk_{K_v} \mathbf{G} \geq
1$.  A theorem of Borel and Harish--Chandra shows that the image of
$\mathbf{G}(O_S)$ under the diagonal embedding is indeed an
irreducible lattice in the higher rank semi-simple group $H =
\prod_{v\in S} \mathbf{G}(K_v)$.  Moreover, Margulis proved that this
construction produces, up to commensurability, essentially all
irreducible lattices in higher rank semi-simple groups.  Precise
notions and a more complete description can be found
in~\cite{Margulis}.  In this paper, we call a group arithmetic if it
is commensurable to a group of the form $\mathbf{G}(O_S)$ as above.
In particular, all arithmetic groups that we consider are defined in
characteristic~$0$.

The study of representation growth for arithmetic groups was initiated
by Lubotzky and Martin in~\cite{LM}.  They showed that, whenever
$\Gamma$ is commensurable to $\mathbf{G}(O_S)$ \ben{as above and
  $\rk_{K_\fp} \mathbf{G} \geq 1$ for every finite place $\fp \in S$},
then the growth of the sequence $R_n(\Gamma)$, $n \in \N$, is bounded
polynomially in $n$ if and only if $\mathbf{G}(O_S)$ has the weak
Congruence Subgroup Property.  To discuss the latter, let
$\widehat{\mathbf{G}(O_S)}$ and $\widehat{O_S}$ denote the profinite
completions of the group $\mathbf{G}(O_S)$ and the ring~$O_S$.
Furthermore, we write $O_\fp$ for the completion of the ring of
integers $O$ of $K$ at a prime~$\fp$.  The group $\mathbf{G}(O_S)$ has
the weak Congruence Subgroup Property (wCSP) if the kernel of the
natural map
\[
\widehat{\mathbf{G}(O_S)} \longrightarrow \mathbf{G}(\widehat{O_S})
\cong \prod_{\fp \in \Spec(O) \smallsetminus S}
\mathbf{G}(O_\fp),
\]
is finite.  A long-standing conjecture of Serre asserts, in
particular, that $\mathbf{G}(O_S)$ has the wCSP whenever
$\rk_S \mathbf{G} \geq 2$ and $\rk_{K_\fp} \mathbf{G}
\geq 1$ for every finite place $\fp \in S$.  This part of
Serre's conjecture is known to be true in many cases; e.g., it holds
for groups yielding non-uniform irreducible lattices in higher rank
semi-simple groups.  For more information see~\cite[Chapter 9.5]{PR},
\cite{Ra} or~\cite{PrRa10}, and the references therein.

Next we recall the definition of the representation zeta function of a
group $G$ and its abscissa of convergence, which captures the degree
of representation growth of~$G$.

\begin{definition*} Let $G$ be a group such that $R_n(G)$ is finite
  for all~$n\in\N$.  The representation zeta function of $G$ is the
  Dirichlet generating series
  \[
  \zeta_G(s) = \sum_{\rho \in \Irr(G)} (\dim \rho)^{-s},
  \]
  where $\Irr(G)$ is the set of equivalence classes of
  finite-dimensional irreducible complex representations $\rho$ of $G$
  and $s\in \C$ is a complex variable.

  The abscissa of convergence of $\zeta_G(s)$ is the infimum of all
  $\sigma \in \R$ such that the series $\zeta_G(s)$ converges
  absolutely for all $s \in \C$ with $\mathrm{Re}(s) >
  \sigma$; we denote this invariant by~$\alpha(G)$.  In particular,
  $\alpha(G) = \infty$ if $\zeta_G(s)$ diverges for all $s \in
  \C$.
\end{definition*}

Whenever a group $G$, as in the definition above, possesses infinitely
many finite-dimensional irreducible complex representations, the
abscissa of convergence $\alpha(G)$ is related to the asymptotic
behavior of the sequence $R_n(G)$ by the equation
\begin{equation} \label{eq:abscissa.limsup} \alpha(G) = \limsup_{n
    \rightarrow \infty}\frac{\log R_n(G)}{\log n}.
\end{equation}
In the case of an arithmetic group $\Gamma = \mathbf{G}(O_S)$, as
described above, Lubotzky and Martin's result in~\cite{LM} can
therefore be stated as follows: $\Gamma$ has the wCSP if and only if
$\alpha(\Gamma)<\infty$.  In this sense the invariant $\alpha(\Gamma)$
provides a means to study the wCSP in a quantitative way.  We also
remark that if $\Gamma$ has the wCSP, then $\alpha(\Gamma) = \lim_{n
  \rightarrow \infty}\frac{\log R_n(\Gamma)}{\log n}$ is actually a
limit, hence $R_n(\Gamma) = n^{\alpha(\Gamma)+o(1)}$; this is shown
implicitly in~\cite{A}.

\subsection{Discussion of Main Results}
In this paper we establish new quantitative results regarding the
representation growth of arithmetic groups with the wCSP.  Our first
main theorem is the following.

\begin{thm} \label{thm:strong.LL} Let $\Phi$ be an irreducible root
  system.  Then there exists a constant $\alpha_\Phi$ such that, for
  every arithmetic group $\mathbf{G}(O_S)$, where $O_S$ is the ring of
  $S$-integers of a number field $K$ with respect to a finite set of
  places $S$ and $\mathbf{G}$ is an affine group scheme over $O_S$
  whose generic fiber is connected, simply connected absolutely almost
  simple with absolute root system $\Phi$, the following holds: if
  $\mathbf{G}(O_S)$ has the wCSP, then
  $\alpha(\mathbf{G}(O_S))=\alpha_\Phi$.
\end{thm}

The theorem highlights two challenging open problems, namely to
determine the constants~$\alpha_\Phi$ and to establish finer
asymptotics for the representation growth of arithmetic groups with
the wCSP.  Even at the conjectural level we are presently very far
from solving these problems.  The main theorem in~\cite{A} shows that
$\alpha_\Phi \in \mathbb{Q}$ for all~$\Phi$.  Furthermore,
$\alpha_\Phi \geq \frac{1}{15}$ (see~\cite[Theorem~8.1]{LL}) and
$\alpha_{A_\ell} \leq 22$ for all $\ell \in \N$, with similar bounds on other root systems (see~\cite{AA}).  The
only precisely known values are $\alpha_{A_1} = 2$
(see~\cite[Theorem~10.1]{LL}) and $\alpha_{A_2} = 1$ (see
\cite[Theorem~C]{AKOV}).  In fact, for arithmetic groups $\Gamma =
\mathbf{G}(O_S)$ with the wCSP which arise from affine group schemes
$\mathbf{G}$ of type $A_1$ or~$A_2$, even finer asymptotics of the
representation growth of $\Gamma$ have been established.  If
$\mathbf{G}$ has absolute root system $A_1$, then $\zeta_\Gamma(s)$
admits a meromorphic continuation beyond its abscissa of convergence
and has a simple pole at $s=2$ (compare~\cite{LL} and~\cite{AKOV2});
consequently, $R_n(\Gamma) = (c_\Gamma + o(1)) n^2$ for a constant
$c_\Gamma \in \R$.  Similarly, if $\mathbf{G}$ has absolute root
system $A_2$, then $\zeta_\Gamma(s)$ has a meromorphic continuation
beyond its abscissa of convergence and a double pole at $s=1$
(see~\cite{AKOV2}); consequently, $R_n(\Gamma) = (c_\Gamma + o(1)) n
\log n$ for a constant $c_\Gamma \in \R$.  For general~$\Gamma$, it
remains open whether and how far $\zeta_\Gamma(s)$ can be extended
meromorphically and, if so, whether the order of the resulting pole at
$s =\alpha(\Gamma)$ depends only on the absolute root system $\Phi$.

Besides its intrinsic group theoretic importance, the invariant
$\alpha(\mathbf{G}(O_S))$ of an arithmetic group $\mathbf{G}(O_S)$ is
also related to the singularities of deformation varieties of surface
groups inside $\mathbf{G}(\C)$; see~\cite{AA}.  Furthermore it
is significant for the volumes of the moduli space of $U$-local
systems on algebraic curves, where $U$ is a compact group.  We refer
to Witten's paper~\cite{Wi} for the case where $U$ is a compact Lie
group and to~\cite{AA} for the case where $U$ is a maximal compact
subgroup in the adelic group $\mathbf{G}(\mathbb{A}_K)$.

Finally, we remark that Theorem~\ref{thm:strong.LL} is similar in
spirit to the main result of~\cite{LuNi}, which pins down the subgroup
growth of irreducible lattices in higher rank semi-simple groups
(modulo Serre's conjecture and the generalized Riemann Hypothesis).
We emphasize that the subgroup growth of such lattices is always
faster than polynomial and, in fact, our proofs and methods are
completely different from those used in~\cite{LuNi}.

\medskip

Let us now focus on the representation growth of irreducible lattices
in semi-simple locally compact groups.  Let $H$ be a semi-simple group
in characteristic~$0$ of the form $H = \prod_{j=1}^r
\mathbf{H}_j(F_j)$, where each $F_j$ is a local field of
characteristic~$0$ and each $\mathbf{H}_j$ is a connected, almost
simple $F_j$-group.  As indicated at the end of
Section~\ref{sec:intro}, for every arithmetic irreducible lattice
$\Gamma$ in $H$ we may regard $\alpha(\Gamma)$ as a quantitative
measure for the wCSP.  Moreover, Serre's conjecture on the Congruence
Subgroup Problem asserts that the question whether such a lattice
$\Gamma$ has the wCSP, equivalently whether $\alpha(\Gamma)<\infty$,
does not depend on the particular choice of $\Gamma$, but is
controlled by the ambient group~$H$.  Larsen and Lubotzky conjectured
that, if $H$ has higher rank, i.e., $\sum_{j=1}^r \rk_{F_j}
\mathbf{H}_j \geq 2$, then the abscissa of convergence
$\alpha(\Gamma)$ is the same for all irreducible lattices $\Gamma$ in
$H$; see~\cite[Conjecture 1.5]{LL}.  We establish the following
variant.

\begin{thm} \label{thm:LL} Let $H$ be a semi-simple group in
  characteristic~$0$, and let $\Gamma_1, \Gamma_2$ be two arithmetic
  irreducible lattices in $H$, both having the wCSP or, equivalently,
  satisfying $\alpha(\Gamma_1), \alpha(\Gamma_2) < \infty$. Then
  $\alpha(\Gamma_1) = \alpha(\Gamma_2)$.
\end{thm}

This unconditional result and Margulis' Arithmeticity Theorem
immediately reduce Larsen and Lubotzky's conjecture to the original
conjecture of Serre.

\begin{thm} \label{thm:LL.conditional} Let $H$ be a higher-rank
  semi-simple group in characteristic~$0$.  Assuming Serre's
  conjecture, for any two irreducible lattices $\Gamma_1,\Gamma_2$
  in~$H$ we have $\alpha(\Gamma_1) = \alpha(\Gamma_2)$.
\end{thm}

We emphasize that our results are at the same time weaker and stronger
than~\cite[Conjecture 1.5]{LL}.  They are weaker, since we do not
prove Serre's conjecture, and stronger, since
Theorem~\ref{thm:strong.LL} shows that the abscissa of convergence
depends only on the absolute root system associated to the ambient
semi-simple group~$H$.

\medskip

Central to this paper are new insights into the behavior of the
abscissa of convergence under base change.  More precisely, we
consider the relation between the abscissae of convergence for groups
$\mathbf{G}(O_1)$ and $\mathbf{G}(O_2)$, where $O_1 \subset O_2$ is a
ring extension.  We initiated this study in~\cite{AKOV} in the local
case, where each $O_i$ is a compact discrete valuation ring of
characteristic~$0$.  In the present paper we consider the global case,
where each $O_i$ is the ring of $S_i$-integers in a number
field~$K_i$.

It is convenient to organize the results on base change of arithmetic
groups into two theorems.  Theorem~\ref{thm:discrete.to.adeles} links
the abscissa of convergence for an arithmetic group $\mathbf{G}(O_S)$
with the wCSP to the abscissa of convergence for the group
$\mathbf{G}(\widehat{O})$ over the integral adeles $\widehat{O} =
\prod_{\fp \in \Spec(O)} O_\fp$.  Key to this are results of Larsen
and Lubotzky, such as an Euler product factorization for
representation zeta functions of arithmetic groups and results on the
representation growth of Lie groups, and~\cite[Theorem~B]{AKOV}.  The
latter is a base change result in the local case; we record a relevant
corollary as Theorem~\ref{thm:BC.local} below.

\begin{thm} \label{thm:discrete.to.adeles} Let $K$ be a number
  field with ring of integers $O$ and let $S$ be a finite set of
  places of~$K$.  Let $\mathbf{G}$ be an affine group scheme defined
  over $O_S$ whose generic fiber is connected and simply connected
  semi-simple.  Suppose that $\mathbf{G}(O_S)$ has the wCSP.  Then
  $\alpha(\mathbf{G}(O_S)) = \alpha \big(\mathbf{G}(\widehat{O})
  \big)$.
\end{thm}

Theorem~\ref{thm:BC.global} incorporates the main thrust of the
current paper.  It relates the abscissae of convergence for the adelic
groups $\mathbf{G}(\widehat{O_K})$ and $\mathbf{G}(\widehat{O_L})$,
where $K \subset L$ is an extension of number fields.

\begin{thm} \label{thm:BC.global} Let $K \subset L$ be number fields
  with rings of integers $O_K \subset O_L$, and let $\mathbf{G}$ be an
  affine group scheme defined over $O_K$ whose generic fiber is
  connected and simply connected semi-simple. Then
  $\alpha \big( \mathbf{G}(\widehat{O_K}) \big) =
  \alpha \big( \mathbf{G}(\widehat{O_L}) \big)$.
\end{thm}

It is noteworthy that the global situation is more rigid than the
local one: in the local case, the abscissa of convergence is monotone
non-decreasing with respect to base change, but it can be strictly
increasing; see Remark~\ref{rem:local.global.comparison}.

At a formal level we can summarize the base change theorems as
follows.  Theorem~\ref{thm:discrete.to.adeles} means that
$\alpha(\mathbf{G}(O_1)) = \alpha(\mathbf{G}(O_2))$ when $\Spec(O_2)
\rightarrow \Spec(O_1)$ is an open embedding.
Theorem~\ref{thm:BC.global} means that $\alpha(\mathbf{G}(O_1)) =
\alpha(\mathbf{G}(O_2))$ when $\Spec(O_2) \rightarrow \Spec(O_1)$ is
finite.  Taken together, they mean that $\alpha(\mathbf{G}(O_1)) =
\alpha(\mathbf{G}(O_2))$ when $\Spec(O_2) \rightarrow \Spec(O_1)$ has
finite fibers.

Theorem~\ref{thm:BC.global} is more difficult to prove than its local
counterpart Theorem~\ref{thm:BC.local}.  Our approach is based on
close approximations of the representation zeta functions associated
to groups of the form $\mathbf{G}(O_\fp)$.  A central feature of these
approximations is that they are uniform, as~$O$ ranges over the set of
all finite extensions of a fixed global ring and $\fp$ ranges over a
cofinite set of primes of~$O$.  We give such approximations in
Theorem~\ref{thm:zeta.approx}.  Similar approximations for zeta
functions of the finite groups of Lie type $\mathbf{G}(O/\fp)$ are
derived in Theorem~\ref{thm:BC.finite.new}, using Deligne--Lusztig
theory.  This allows us to control representations of
$\mathbf{G}(\widehat{O})$ that factor through $\prod_\fp
\mathbf{G}(O/\fp)$, but does not account for other
representations.  In contrast to the situation for finite
groups of Lie type, the representation theory of groups over local
rings is at present poorly understood.  Instead of enumerating
representations directly, we follow in the proof of
Theorem~\ref{thm:zeta.approx} Weil's idea to express local factors of
zeta functions as $p$-adic integrals.  We show that in our case the
local factors can be approximated by a class of $p$-adic integrals
involving quantifier-free definable functions
(Theorem~\ref{thm:int.approx}) and that such integrals admit uniform
formulae (Theorem~\ref{thm:mot.int}).  In this context we employ tools
from the model theory of valued fields, such as partial elimination of
quantifiers in the theory of Henselian valued fields of residue
characteristic~$0$.  Working with quantifier-free definable functions,
we strike a balance between being able to approximate the relevant
local factors in the first place and being able to derive uniform
formulae for the resulting integrals.  At present it is unknown
whether one may use more elementary classes of functions, such as
polynomials, to carry out an equally effective approximation.

\medskip

\noindent \textbf{Organization.}  In Section~\ref{sec:overview} we
prove Theorem~\ref{thm:discrete.to.adeles}, state
Theorem~\ref{thm:zeta.approx}, and prove Theorems~\ref{thm:BC.global},
\ref{thm:strong.LL}, and~\ref{thm:LL}, using
Theorem~\ref{thm:zeta.approx}.  The rest of the paper is dedicated to
proving Theorem~\ref{thm:zeta.approx}.  In Section~\ref{sec:BC.finite}
we prove Theorem~\ref{thm:BC.finite.new}, which is a variant of
Theorem~\ref{thm:zeta.approx} for zeta functions of semi-simple
algebraic groups over finite fields, and apply it to finite quotients
of arithmetic groups.  In Section~\ref{sec:preliminaries} we collect
some results about relative representation zeta functions, the
Kirillov orbit method, and model theory. In
Section~\ref{sec:parameterize} we prove that the local factors of
representation zeta functions of arithmetic groups are approximated by
integrals of quantifier-free definable functions. In
Section~\ref{sec:q.f.integrals} we prove that integrals of
quantifier-free definable functions have a uniform formula, and finish
the proof of Theorem~\ref{thm:zeta.approx}.

\medskip

\noindent \textbf{Notation.} All affine group schemes appearing in
this paper are algebraic.  For reference purposes we summarize some of
the notation used frequently.
\begin{list}{$\circ$}{\setlength{\leftmargin}{\mylistleftmargin}
    \setlength{\labelwidth}{10pt} \setlength{\itemsep}{0pt}
    \setlength{\parsep}{1pt}}
\item $\mathbf{G}$, $\mathbf{H}$ denote affine group schemes; in this
  context $\mathfrak{g}$ refers to the Lie algebra of~$\mathbf{G}$,
  but sometimes $\mathfrak{g}$ denotes a more general Lie lattice.
\item $\Phi$ denotes a root system; $\rk \Phi$ its rank \ben{and}
  $\Phi^+$ a choice of positive roots\ben{; }%
  %, and $\comp(\Phi)$ the set of irreducible components; 
  $\mathcal{C}(\mathbf{G}, \mathbf{T}, \mathbb{E})$ and
  $\mathcal{C}(\Phi)$ are defined in~\eqref{equ:C}, \eqref{equ:C_Phi}.
\item $\Gamma, \Delta$ denote arithmetic groups.
\item $K, L$ denote number fields, with rings of integers $O_K, O_L$.
\item $\fp$ denotes a prime of $O_K$ and $\fq$ a
  prime of $O_L$.
\item $K_\fp$ and $O_{K,\fp}=O_{K_{\fp}}$ denote the completions of
  $K$ and $O_K$ at $\fp$; similarly, $L_{\fq}$ and
  $O_{L,\fq}=O_{L_{\fq}}$ are the completions of $L$ and $O_L$ at
  $\fq$.
\item $\mathcal{A}$ is a semi-ring with an ideal $\mathcal{A}^+$ and
  $\xi_{a,q}$ a Dirichlet polynomial; see Definition~\ref{defn:Trop}.
\item $\Grass(\mathfrak{g})$ and $\Grass^\mathrm{nilp}(\mathfrak{g})$
  are the Grassmannians of Lie subalgebras and of nilpotent Lie
  subalgebras of a Lie algebra~$\mathfrak{g}$; see
  Definition~\ref{def:grassmannian}.
\item $\mathsf{F}, \mathsf{k}, \mathsf{\Gamma}$ are the sorts of the
  Denef--Pas language of valued fields; see
  Section~\ref{sec:valued.fields}.
\item $\Th_\text{fields}$, $\Th_{\fields,R}$,
  $\Th_{\text{perf.-fields},p,R}$, $\Th_{\mathrm{Hen},0}$, and
  $\Th_{\mathrm{Hen},K,0}$ denote the first-order theories of certain
  types of fields; see Sections~\ref{subsec:def.fun}
  and~\ref{sec:valued.fields}.
\item $\Pi,\Xi$ and decorations thereof refer to relative orbit method
  functions; see Section~\ref{subsec:rel.orbit.method}.
\item $\scX,\scY$ are quantifier-free definable sets introduced in
  Definitions~\ref{defn:X} and~\ref{defn:Y}.
\item $\mathcal{R}, \mathcal{L}, \mathcal{S}$ and decorations thereof
  refer to definable functions/families of Lie algebras/groups over
  $\scX$ and $\scY$; see Section~\ref{subsec:rel.orbit.method} and
  also Theorem~\ref{thm:partial.int}.
\end{list}
Additional summaries of more specialised notation can be found in
Sections~\ref{subsec:rel.orbit.method}
and~\ref{subsec:the.stabilizer}.

\medskip

\noindent \textbf{Acknowledgments.}  Avni was supported by NSF Grant
DMS-0901638. Onn was supported by ISF grant 382/11. We acknowledge
support from the EPSRC. Many thanks to Udi Hrushovski, Michael Larsen,
and Alex Lubotzky for lots of helpful conversations \nir{and the referees  for a number of comments that improved the exposition of the paper}.

%%%%%

\section{Reduction of the Main Results to
  Theorem~\ref{thm:zeta.approx}} \label{sec:overview}

In this section we
prove Theorems~\ref{thm:strong.LL}, \ref{thm:LL},
and~\ref{thm:BC.global}, modulo Theorem~\ref{thm:zeta.approx} stated
below.  For this purpose we fix the following notation.
\begin{list}{$\circ$}{\setlength{\leftmargin}{\mylistleftmargin}
    \setlength{\labelwidth}{10pt} \setlength{\itemsep}{0pt}
    \setlength{\parsep}{1pt}}
\item $K$ is a number field, $O=O_K$ its ring of integers, and $O_S$
  the ring of $S$-integers for a finite set of places $S$.
\item $\mathbf{G}$ is an affine group scheme defined over $O_S$ whose
  generic fiber is connected and simply connected semi-simple.
\item $\Gamma=\mathbf{G}(O_S)$ has the wCSP.
\end{list}
The groups that we consider in this article have the property that
their categories of finite-dimensional complex representations are
semi-simple.  For example, $\Gamma$ satisfies this condition since,
for every finite-dimensional complex representation $\rho$ of
$\Gamma$, the Zariski closure of $\rho(\Gamma)$ is a reductive
algebraic group in characteristic~$0$.

Our starting point is an Euler product factorization for the
representation zeta function of a suitable subgroup~$\Delta \subset
\Gamma$.  The following elementary lemma will be used repeatedly.

\begin{lem} \label{lem:absc.direct.sum} Let $G_1$ and $G_2$ be groups
  such that their categories of finite-dimensional complex
  representations are semi-simple.  If $R_n(G_1)$ and $R_n(G_2)$ are
  finite for all $n \in \N$, then $\zeta_{G_1 \times G_2} =
  \zeta_{G_1} \zeta_{G_2}$. In particular, $\alpha(G_1 \times G_2) =
  \max\{\alpha(G_1), \alpha(G_2)\}$.
\end{lem}

\begin{proof} This follows from the fact that the irreducible complex
  representations of $G_1 \times G_2$ are the tensor products of
  irreducible complex representations of $G_1$ and~$G_2$.
\end{proof}

Considering pro-algebraic completions, one finds that there is a
finite-index subgroup $\Delta \subset \Gamma$ such that
\begin{equation} \label{eq:Euler} \zeta_\Delta =
  \zeta_{\mathbf{G}(\C)}^{\, [K:\mathbb{Q}]} \cdot \prod_{\fp \notin
    S}\zeta_{\Delta_\fp},
\end{equation}
where the product ranges over the primes $\fp \in \Spec(O)
\smallsetminus S = \Spec(O_S)$ and $\Delta_\fp$ is the closure,
in the $\fp$-adic topology, of the image of $\Delta$ under
the embedding $\Gamma \rightarrow \mathbf{G}(O_\fp)$;
see~\cite[Theorem~3.3 and Proposition~4.6]{LL}.  Furthermore, the
generating series $\zeta_{\mathbf{G}(\C)}$ counts only
rational representations, and the generating series
$\zeta_{\Delta_\fp}$ count only continuous
representations. Since $\Delta$ has finite index in $\Gamma$, the
Strong Approximation Theorem implies that $\Delta_\fp$ is
open in $\mathbf{G}(O_\fp)$, for every $\fp$, and
$\Delta_\fp=\mathbf{G}(O_\fp)$, for all but
finitely many primes.

It is well-known that the abscissa of convergence for groups is a
commensurability invariant (see~\cite[Lemma 2.2]{LM}; we prove a more
general version in Lemma~\ref{lem:fin.index.BK}).  In particular,
$\alpha(\Gamma)$ is equal to
$\alpha(\Delta)$. By~\cite[Theorem~5.1]{LL}
and~\cite[Proposition~6.6]{LL}, we have $\alpha(\mathbf{G}(\C)) \leq
\alpha(\mathbf{G}(O_\fp))$, for every $\fp\in \Spec(O_S)$. Therefore,
\eqref{eq:Euler} shows that $\alpha(\Gamma)$ is equal to the abscissa
of convergence of the product $\prod_{\fp \notin S}
\zeta_{\Delta_\fp}$. By another application of the commensurability
invariance, $\alpha(\Gamma)$ is equal to the abscissa of convergence
of $\prod_{\fp\notin S}\zeta_{\mathbf{G}(O_\fp)}$.

To deduce Theorem~\ref{thm:discrete.to.adeles} we need to justify that
the abscissa of convergence of the product $\prod_\fp
\zeta_{\mathbf{G}(O_\fp)}$, ranging over all primes $\fp
\in \Spec(O)$, is unchanged by omitting finitely many factors. This is
a consequence of the following more general result.

\begin{thm} \label{thm:BC.local} Let $K$ be a number field with ring
  of integers $O_K$, and let $\mathbf{G}$ be an affine group scheme
  over $O_K$ whose generic fiber is connected and simply connected
  semi-simple.  Then, for every finite extension $L$ of $K$ with ring
  of integers $O_L$ and every prime $\fq$ of~$O_L$, there are
  infinitely many primes $\fp$ of $O_K$ such that
  $\alpha(\mathbf{G}(O_{L,\fq})) \leq \alpha(\mathbf{G}(O_{K,\fp}))$.
\end{thm}

Theorem~\ref{thm:BC.local} is a corollary of~\cite[Theorem~B]{AKOV},
whose proof uses $p$-adic integrals to analyze the representation zeta
functions of $\fq$-adic groups such as $\mathbf{G}(O_{L,\fq})$.  The
connection to $p$-adic integrals, and more generally definable
integrals in the sense of model theory, is \cite[Corollary~3.7]{AKOV}
(see also~\cite[Lemma~4.1]{Jai}).  The notion of a quantifier-free
definable function is explained in Section~\ref{sec:preliminaries}. It
follows from~\cite[Corollary~3.7]{AKOV} that there are $d \in \N$ and
quantifier-free definable functions $\phi_1,\phi_2$ such that, for
every finite extension $L$ of $K$, every prime $\fq$ of $O_L$, and
every sufficiently large integer $r$, the representation zeta function
of the $r$th principal congruence subgroup
$\mathbf{G}^{(r)}(O_{L,\fq}) = \ker (\mathbf{G}(O_{L,\fq}) \rightarrow
\mathbf{G}(O_L/\fq^r))$ can be expressed as follows:
\begin{equation} \label{eq:AKOV}
  \zeta_{\mathbf{G}^{(r)}(O_{L,\fq})}(s) = \lvert O_L/ \fq \rvert^{r
    \cdot \dim \mathbf{G}}\int_{O_{L,\fq}^d} \lvert \phi_1(x)
  \rvert_\fq \; \lvert \phi_2(x) \rvert_\fq^{-s} d\lambda(x),
\end{equation}
where the absolute value in the integrand is the $\fq$-adic one, and
$\lambda$ is the additive Haar measure on $L_{\fq}^d$ normalized so
that $\lambda(O_{L,\fq}^d) = 1$. Since the abscissae of convergence
for $\mathbf{G}(O_{L,\fq})$ and its $r$th principal congruence
subgroup are equal, the claim in Theorem~\ref{thm:BC.local} may thus
be reduced to a similar claim for integrals of the
form~\eqref{eq:AKOV}.  The main point is that the functions $\phi_1,
\phi_2$ are independent of $L$ and $\fq$, allowing us to compare the
integrals for varying fields and primes.

\begin{rem} \label{rem:local.global.comparison}
  Theorem~\ref{thm:BC.local} can be regarded as a local analog of
  Theorem~\ref{thm:BC.global}.  We explain why a more naive analog is
  false.  Suppose that $K\subset L$ is a finite extension of number
  fields and that $\fq$ is a prime of $O_L$ lying over a prime $\fp$
  of~$O_K$.  Then~\cite[Theorem~B]{AKOV} implies that
  $\alpha(\mathbf{G}(O_{K,\fp})) \leq \alpha(\mathbf{G}(O_{L,\fq}))$.
  However, in contrast to the global case considered in
  Theorem~\ref{thm:BC.global}, the inequality can be strict. For
  example, let $D$ be a central division algebra of degree $d$ over a
  non-archimedean local field~$F$, and let $\mathbf{G}$ be an affine
  group scheme over $O_F$ such that $\mathbf{G}(F)$ is the group of
  norm $1$ elements in~$D$.  Then the abscissa of convergence for
  $\mathbf{G}(O_F)$ is $2/d$; see~\cite[Theorem~7.1]{LL}.  But if $F
  \subset E$ is an extension such that $D$ splits over~$E$, then
  $\alpha(\mathbf{G}(O_E)) \geq 1/15$ by~\cite[Theorem~8.1]{LL}, and
  hence strictly greater than $2/d$ for~$d>30$.
\end{rem}

We move on to the proof of Theorem~\ref{thm:BC.global}.  The
description in~\eqref{eq:AKOV} of representation zeta functions of
principal congruence subgroups $\mathbf{G}^{(r)}(O_\fp)$ as $p$-adic
integrals is of limited use for
determining~$\alpha(\mathbf{G}(\widehat{O}))$.  This is due to the
fact that the infinite product of such congruence subgroups is not of
finite index in $\mathbf{G}(\widehat{O}) = \prod_\fp
\mathbf{G}(O_\fp)$.  Facing the challenge to deal with the groups
$\mathbf{G}(O_\fp)$, and not merely congruence subgroups of
sufficiently large index, we approximate their representation zeta
functions in the following sense.

\begin{defn} \label{defn:lesssim} Let $f(s) = \sum_{n=1}^\infty
  a_n n^{-s}$ and $g(s) = \sum_{n=1}^\infty b_n n^{-s}$ be Dirichlet
  generating series, i.e., Dirichlet series with integer coefficients
  $a_n,b_n \geq 0$.  Let $C \in \R$.  Suppose that $\sigma_0 \in
  \R_{\geq 0}$ is greater than or equal to the abscissae of
  convergence of $f$ and $g$.  We write
  \[
  f \lesssim_C g \quad \text{for $\sigma > \sigma_0$}
  \]
  if $f(\sigma) \leq C^{1+\sigma} g(\sigma)$ for every $\sigma \in
  \R$ with $\sigma > \sigma_0$. We write $f \lesssim_C g$, without
  specifying the domain, if $f$ and $g$ have the same abscissa of
  convergence $\alpha$ and if $f \lesssim_C g$ for $\sigma > \max \{
  0,\alpha \}$.  Finally, we write $f \sim_C g$ if $f\lesssim_C g$ and
  $g\lesssim_C f$.
\end{defn}

We routinely use the fact that $f \lesssim_{C_1} g$ and $g
\lesssim_{C_2} h$ imply $f \lesssim_{C_1C_2} h$.

\begin{lem} \label{lem:lesssim} Let $f, g$ be Dirichlet generating
  series with abscissae of convergence $\alpha_f, \alpha_g$.  Suppose
  that $f = \prod_{m=1}^\infty (1+f_m)$ and $g = \prod_{m=1}^\infty
  (1+g_m)$, where $f_m, g_m$ are Dirichlet generating series with
  vanishing constant terms, and, for every $m \in \N$, let $\beta_m$
  denote the abscissa of convergence of $g_m$.  Suppose further that,
  for each $\epsilon >0$, there is $C(\epsilon) \in \R_{>0}$ such
  that, for all $m$, $f_m \lesssim_{C(\epsilon)} g_m$ for $\sigma >
  \beta_m +\epsilon$. Then $\alpha_f \leq \alpha_g$.
\end{lem}

\begin{proof} The assumptions imply that the abscissa of convergence
  of $f_m$ is less than or equal to $\beta_m$ for every~$m$. The
  abscissa of convergence of a Dirichlet generating series is
  determined by its behavior on the real axis.  Fix $\epsilon >0$, and
  let $\sigma \in \R$ with $\sigma > \max\{0,\alpha_g\} + \epsilon$,
  so that $\sigma > \beta_m +\epsilon$ for all~$m$. Then $g(\sigma) =
  \prod_m(1+g_m(\sigma))$ and hence $\sum_m g_m(\sigma)$ converge.  As
  $f_m(\sigma) \leq C(\epsilon)^{1+\sigma} g_m(\sigma)$ for all $m$,
  this implies that $\sum_m f_m(\sigma)$ and hence $f(\sigma) =
  \prod_m(1+f_m(\sigma))$ converge.  Thus $\sigma > \alpha_f$. Letting
  $\epsilon$ tend to $0$, we deduce that $\alpha_g \geq \alpha_f$.
\end{proof}

We now introduce a semi-ring $\mathcal{A}$ whose elements $a$ are used
to index Dirichlet polynomials $\xi_{a,q}$ approximating certain
Dirichlet generating series.

\begin{defn} \label{defn:Trop} Let $\mathcal{A}$ be the collection of
  all finite subsets $a \subset \mathbb{Z}_{\geq 0} \times
  \mathbb{Z}_{>0} \cup \left\{ (0,0) \right\}$. We turn $\mathcal{A}$
  into a commutative unital semi-ring by defining the sum of
  $a,b\in\mathcal{A}$ as $a+b = a \cup b$ and their product as $a\cdot
  b = \{ u+v \mid u \in a, v \in b\}$.  Note that the neutral elements
  in $\mathcal{A}$ are $0 = \varnothing$ and $1 = \{(0,0)\}$.  For $a
  \in \mathcal{A}$ and $n \in \N$ we write $a^{(n)} = \{ (n u_1, n
  u_2) \mid (u_1,u_2) \in a \}$.

  The set $\mathcal{A}^+ = \{ a \in \mathcal{A} \mid (0,0) \notin a
  \}$ forms an ideal of the semi-ring~$\mathcal{A}$.  For $a,b \in
  \mathcal{A}^+$ we define $a*b \in \mathcal{A}^+$ by the formula
  $(1+a)(1+b) = 1+a*b$.  For $a \in \mathcal{A}^+$ and $n \in \N$ we
  write $a^{*n} = a* \cdots *a$ for the $n$-fold power with respect
  to~$*$.

  For $a\in\mathcal{A}$ and $q \in \N_{\geq 2}$ we define the
  Dirichlet polynomial
  \[
  \xi_{a,q}(s) = \sum_{(m,n)\in a}q^{m-ns}.
  \]
\end{defn}

\begin{rem}\label{rem:xi}
  (1) Let $a,b\in \mathcal{A}^+$ and $n\in \N$.  The following
  properties are immediate from the definitions: $a^{(n)} \subset
  a^{*n}$, $\xi_{a,q^n} = \xi_{a^{(n)},q}$, and there exists $C=C(a,b)
  \in \R$ such that
  \[
  (1+\xi_{a,q})(1+\xi_{b,q})-1 \sim_C \xi_{a*b,q} \quad \text{for all
    $q \in \N_{\geq 2}$.}
  \]
  For instance, $C(a,b) = 1 + \min \{ \lvert a \rvert, \lvert b \rvert
  \}$ works.  Furthermore, if $a \subset b$, then $\xi_{a,q}
  \lesssim_1 \xi_{b,q}$ for all $q \in \N_{\geq 2}$.

  (2) Let $a,b\in \mathcal{A}$.  Let $\mathcal{N}(a)$ denote the
  ``north-west''-Newton polytope associated to~$a$, i.e., the convex
  hull of $\bigcup \{ u + (\R_{\leq 0} \times \R_{\geq
    0}) \mid u \in a \}$ in $\R^2$.  Then $\mathcal{N}(a)
  \subset \mathcal{N}(b)$ if and only if there exists $C \in
  \R$ such that $\xi_{a,q}(s) \lesssim_C \xi_{b,q}(s)$ for all
  $q\in\N_{\geq 2}$.  In fact, if $\mathcal{N}(a) \subset
  \mathcal{N}(b)$ one can take $C = \lvert a \rvert$.
\end{rem}

Recall that the Dedekind zeta function $\zeta_K(s)=\prod_{\fp \in
  \Spec(O)}\left( 1 - \lvert O/\fp \rvert^{-s}\right)^{-1}$ of a
number field $K$ has abscissa of convergence $1$, and that\nir{, for
  any subset $T\subset \Spec(O)$ with positive analytic density, the
  abscissa of convergence of $\prod_{\fp \in T}\left(1 - \lvert O/\fp
    \rvert^{-s}\right)^{-1}$ is equal to~$1$.}  Of course, every
co-finite subset of $\Spec(O)$ has positive analytic density.

\begin{thm} \label{thm:zeta.approx}  Let $K$
  be a number field with ring of integers $O_K$, and let $\mathbf{G}$
  be an affine group scheme over $O_K$ whose generic fiber is
  connected, simply connected semi-simple.  Then there exist
  \begin{list}{$\circ$}{\setlength{\leftmargin}{\mylistleftmargin}
      \setlength{\labelwidth}{10pt} \setlength{\itemsep}{0pt}
      \setlength{\parsep}{1pt}}
  \item $c(\mathbf{G}) \in \mathcal{A}^+$,
  \item for every finite extension $L$ of $K$ with ring of integers
    $O_L$, subsets $R(L) \subset T(L) \subset \Spec(O_L)$ with $T(L)$
    co-finite and $R(L)$ of positive analytic density in $\Spec(O_L)$,
  \item for every $\epsilon \in \R_{>0}$, a constant
    $C(\epsilon) \in \R_{>0}$
  \end{list}
  such that for every finite extension $L$ of $K$ with ring of
  integers $O_L$ the following hold.
  \begin{list}{}{\setlength{\leftmargin}{\myenumilistleftmargin}
      \setlength{\labelwidth}{20pt} \setlength{\itemsep}{0pt}
      \setlength{\parsep}{1pt}}
  \item[\textup{(1)}] For every $\fq \in T(L)$, there is a
    subset $c \subset c(\mathbf{G})$ such that, for every $\epsilon
    \in \R_{>0}$,
    \[
    \zeta_{\mathbf{G}(O_{L,\fq})}-1 \sim_{C(\epsilon)} \xi_{c,
      \lvert O_L/\fq \rvert} \quad \text{for $\sigma >
      \alpha(\mathbf{G}(O_{L,\fq}))+\epsilon$.}
    \]
  \item[\textup{(2)}] For every $\fq \in R(L)$ and every
    $\epsilon \in \R_{>0}$,
    \[
    \zeta_{\mathbf{G}(O_{L,\fq})}-1 \sim_{C(\epsilon)}
    \xi_{c(\mathbf{G}), \lvert O_L/\fq \rvert} \quad
    \text{for $\sigma > \alpha(\mathbf{G}(O_{L,\fq})) +
      \epsilon$.}
    \]
  \end{list}
\end{thm}

We remark that the element $c(\mathbf{G})$ in
Theorem~\ref{thm:zeta.approx} depends on $\mathbf{G}$, but is not
canonically determined by it: in the course of the proof we implicitly
make a number of choices, influenced by triangulations of polytopes
and resolutions of singularities, and different choices result in
different elements~$c(\mathbf{G})$. In addition, the proof actually
shows that the set $R(L)$ is a Chebotarev set in the sense of
Definition \ref{defn:Chebotarev}. The proof of
Theorem~\ref{thm:zeta.approx} occupies a large part of the paper and
is completed at the end of Section~\ref{sec:q.f.integrals}.  We now
show how Theorem~\ref{thm:zeta.approx} implies
Theorem~\ref{thm:BC.global}.
%, \ref{thm:strong.LL}, and~\ref{thm:LL}.

\begin{lem} \label{lem:Dedekind-arg} Let $a \in
  \mathcal{A}^+\smallsetminus\{\varnothing\}$, let $L$ be a number
  field with ring of integers~$O_L$ \nir{and $R\subset \Spec(O_L)$ of
    positive analytic density}.  Then the abscissa of convergence of
  the Dirichlet series $\xi = \prod_{\nir{\fq \in R}}(1+\xi_{a, \lvert
    O_L/ \fq\rvert})$ is equal to $\max \{ (m+1)/n \mid (m,n) \in a
  \}$.
\end{lem}

\begin{proof}
  For any two sequences $(x_n)_{n \in \N}$ and $(y_n)_{n \in \N}$ of
  positive real numbers, the product $\prod_{n \in \N} (1+x_n+y_n)$
  converges if and only if $\prod_{n \in \N}(1+x_n)$ and $\prod_{n \in
    \N}(1+y_n)$ converge individually.  Thus, if $a=b+c$ in
  $\mathcal{A}$, then the abscissa of convergence of $\xi$ is the
  maximum of the abscissae of convergence of $\prod_{\nir{\fq \in
      R}}(1+\xi_{b, \lvert O_L/ \fq\rvert})$ and $\prod_{\nir{\fq \in
      R}} (1+\xi_{c, \lvert O_L/ \fq\rvert})$.  Thus we may assume
  that $a = \left\{ (m,n) \right\}$ is a singleton.

  Let \nir{$\zeta_{L,\Spec(O_L)\setminus R}(s) = \prod_{\fq \in
      R}(1-|O_L/\fq|^{-s})^{-1}$ denote the ``restriction'' of the
    Dedekind zeta function of~$L$ to the Euler product over factors in
    $R$.}  We observe that \nir{\[ \xi(s) =
    \frac{\zeta_{L,\Spec(O_L)\setminus R}
      (ns-m)}{\zeta_{L,\Spec(O_L)\setminus R} (2(ns-m))}.
    \]} Since the abscissa of convergence of
  \nir{$\zeta_{L,\Spec(O_L)\setminus R}(s)$} is equal to $1$ and
  \nir{$\zeta_{L,\Spec(O_L)\setminus R}(s) \neq 0$} for $\mathrm{Re}(s)>1$,
  we deduce that the abscissa of convergence of $\xi$ is equal to
  $(m+1)/n$.
\end{proof}

\begin{proof}[Proof of Theorem~\ref{thm:BC.global}]
  Let $c(\mathbf{G}) \in \mathcal{A}^+$, $R(L) \subset T(L) \subset
  \Spec(O_L)$ and $C(\epsilon)$, for $\epsilon >0$, be as in
  Theorem~\ref{thm:zeta.approx}.  As noted before,
  Theorem~\ref{thm:BC.local} implies that $\alpha \big(
  \mathbf{G}(\widehat{O_{L}}) \big) = \alpha \big( \prod_{\fq \in
    \Spec(O_L)} \mathbf{G}(O_{L,\fq}) \big) = \alpha \big( \prod_{\fq
    \in T(L)} \mathbf{G}(O_{L,\fq}) \big)$.  Moreover, $R(L) \subset
  T(L)$ implies that
  \begin{equation} \label{eq:global} \alpha \Big(
    \prod\nolimits_{\fq \in R(L)}
    \mathbf{G}(O_{L,\fq}) \Big) \leq \alpha \Big(
    \prod\nolimits_{\fq \in T(L)}
    \mathbf{G}(O_{L,\fq}) \Big).
  \end{equation}

  For $\fq \in R(L)$ and $\epsilon > 0$, Theorem~\ref{thm:zeta.approx}
  yields that $\zeta_{\mathbf{G}(O_{L,\fq})}-1 \sim_{C(\epsilon)}
  \xi_{c(\mathbf{G}),\lvert O_L/\fq\rvert}$ for $\sigma >
  \alpha(\mathbf{G}(O_{L,\fq})) +\epsilon$. Lemma~\ref{lem:lesssim}
  implies that the left-hand side of \eqref{eq:global} is equal to the
  abscissa of convergence of $\prod_{\fq \in R(L)}(1+
  \xi_{c(\mathbf{G}),\lvert O_L/ \fq\rvert})$.  Similarly, since
  \[
  \zeta_{\mathbf{G}(O_{L,\fq})}-1 \lesssim_{C(\epsilon)} \xi_{c,\lvert
    O_L/\fq\rvert} \lesssim_1 \xi_{c(\mathbf{G}),\lvert O_L/\fq\rvert}
  \quad \text{for $\sigma > \alpha(\mathbf{G}(O_{L,\fq})) +\epsilon$},
  \]
  for every $\fq \in~T(L)$ and suitable $c \subset c(\mathbf{G})$, the
  right-hand side of \eqref{eq:global} is less than or equal to the
  abscissa of convergence of $\prod_{\fq \in T(L)}(1+
  \xi_{c(\mathbf{G}), \lvert O_L/ \fq\rvert})$. Since $R(L)$ has
  positive analytic density, the two abscissae are equal\nir{. By
    Lemma~\ref{lem:Dedekind-arg} their common value, and hence $\alpha
    \big( \mathbf{G}(\widehat{O_L}) \big)$, is completely determined
    by~$c(\mathbf{G})$.} % consequently, $\alpha
%  \big( \mathbf{G}(\widehat{O_L}) \big)$ is equal to the abscissa of
%  convergence of $\prod_{\fq \in \Spec(O_L)}(1+
%  \xi_{c(\mathbf{G}),\lvert O_L/ \fq \rvert})$.  By
%  Lemma~\ref{lem:Dedekind-arg}, the latter is completely determined
%  by~$c(\mathbf{G})$.
\end{proof}

\begin{proof}[Proof of Theorem~\ref{thm:strong.LL}]
  Let $\mathbf{S}$ be a Chevalley group with root system $\Phi$, i.e.,
  an affine group scheme over $\Z$ whose generic fiber is split,
  connected, simply connected absolutely almost simple with absolute
  root system~$\Phi$.  Consider an arithmetic group $\mathbf{G}(O_S)$
  with the wCSP, as in the statement of the theorem.  There is a
  finite extension $L$ of $K$ such that $\mathbf{G}$ and $\mathbf{S}$
  are isomorphic over~$L$.  Denoting the ring of integers of $L$ by
  $O_L$, Theorems~\ref{thm:discrete.to.adeles} and~\ref{thm:BC.global}
  imply that
  \[
  \alpha(\mathbf{G}(O_S)) = \alpha \big(\mathbf{G}(\widehat{O})\big) =
  \alpha \big(\mathbf{G}(\widehat{O_L})\big) =
  \alpha \big(\mathbf{S}(\widehat{O_L})\big) =
  \alpha \big(\mathbf{S}(\widehat{\Z})\big).
  \]
  Thus $\alpha(\mathbf{G}(O_S))$ depends only on $\Phi$.
\end{proof}

\begin{proof}[Proof of Theorem~\ref{thm:LL}]
  Let $H$ be a semi-simple group in characteristic~$0$. Recall that
  this means that $H=\prod_{j=1}^r \mathbf{H}_j(F_j)$, where each
  $F_j$ is a local field of characteristic~$0$, and each
  $\mathbf{H}_j$ is a connected, almost simple group defined over
  $F_j$. Let $\Gamma$ be an arithmetic irreducible lattice in~$H$.

  Then there are a number field $K$ with ring of integers~$O$, a
  finite set $S$ of places of $K$, an affine group scheme $\mathbf{G}$
  defined over $O_S$ whose generic fiber is connected, simply
  connected absolutely almost simple, and a continuous homomorphism
  $\psi \colon \prod_{v\in S} \mathbf{G}(K_v) \rightarrow H$ whose
  kernel and cokernel are compact such that $\psi(\mathbf{G}(O_S))$ is
  commensurable to $\Gamma$. Since $\ker(\psi)\cap \mathbf{G}(O_S)$ is
  finite and $\mathbf{G}(O_S)$ is residually finite, the latter
  contains a finite-index subgroup that is isomorphic to a
  finite-index subgroup of~$\Gamma$. Hence
  $\alpha(\Gamma)=\alpha(\mathbf{G}(O_S))$.

  Fix $j \in \{1,\ldots,r\}$, and denote by $\Lie(\mathbf{G})$ and
  $\Lie(\mathbf{H}_j)$ the Lie algebras associated to $\mathbf{G}$ and
  $\mathbf{H}_j$.  The homomorphism $\psi$ induces a surjection
  $\bigoplus_{v\in S} \Lie(\mathbf{G})(K_v) \rightarrow
  \Lie(\mathbf{H}_j)(F_j)$ which is additive and preserves Lie
  brackets; if $F_j$ is non-archimedean of residue characteristic~$p$,
  the construction uses the Lie correspondence for
  $p$\nobreakdash-adic groups.  Taking the tensor product with
  $\C$ -- over $\R$ if $F_j$ is archimedean and over $\Q_p$
  embedded into $\C$ if $F_j$ is non-archimedean of residue
  characteristic~$p$ -- we obtain a Lie algebra epimorphism
  $\Lie(\mathbf{G})(\C)^m \rightarrow
  \Lie(\mathbf{H}_j)(\C)^n$, for some $m,n \in \N$. In
  particular, since $\mathbf{G}(\C)$ is almost simple, the
  group $\mathbf{G}(\C)$ and the simple factors of the groups
  $\mathbf{H}_j(\C)$ are all isogenous to one another, and
  therefore have the same absolute root system~$\Phi$.  The claim now
  follows from Theorem~\ref{thm:strong.LL}.
\end{proof}

As mentioned above, the abscissa of convergence for groups is a
commensurability invariant; see~\cite[Lemma~2.2]{LM}.  We conclude the
section with a more general lemma about relative zeta functions.

\begin{defn}\label{def:rel.zeta.func} Let $G$ be a group with a
  normal subgroup $N \subset G$ and let $\theta$ be an irreducible,
  finite-dimensional complex representation of~$N$.  Denote by $\Irr(G
  \vert \theta)$ the set of (equivalence classes of)
  finite-dimensional irreducible complex representations $\rho$ of $G$
  such that $\theta$ is a constituent of the restriction~$\Res^G_N
  \rho$, which is completely reducible by Clifford's
  Theorem.

  For $n \in \N$ we denote by $R_n(G \vert \theta)$ the number of
  representations $\rho \in \Irr(G \vert \theta)$ such that $\dim \rho
  \leq n \dim \theta$.  Suppose that $R_n(G \vert \theta)$ is finite
  for all~$n \in \N$.  Then the relative zeta function of $G$ over
  $\theta$ is defined as the Dirichlet generating series
  \[
  \zeta_{G \vert \theta}(s) = \sum_{\rho \in \Irr(G \vert \theta)}
  \left( \frac{\dim \rho}{\dim \theta} \right)^{-s}.
  \]
  We also set
  \[
  \zeta_{G \vert \theta}'(s) = \sum_{\substack{\rho \in \Irr(G \vert
      \theta) \\
    \dim \rho > \dim \theta}} \left( \frac{\dim \rho}{\dim \theta}
  \right)^{-s}.
  \]
\end{defn}

\begin{lem} \label{lem:fin.index.BK} Let $N \subset H \subset G$ be
  groups such that $\lvert G : H \rvert$ is finite and $N$ is normal
  in~$G$.  Let $\theta \in \Irr(N)$, and suppose that $R_n(G \vert
  \theta)$, $R_n(H \vert \theta)$ are finite for all~$n \in \N$.
  \begin{list}{}{\setlength{\leftmargin}{\myenumilistleftmargin}
      \setlength{\labelwidth}{20pt} \setlength{\itemsep}{0pt}
      \setlength{\parsep}{1pt}}
  \item[\textup{(1)}] Let $m \in \N$.  Then
    \[
    \lvert G:H \rvert^{-1} R_{\lfloor m / \lvert G:H \rvert \rfloor}(H
    \vert \theta) \leq R_m(G \vert \theta) \leq \lvert G:H \rvert
    R_m(H \vert \theta).
    \]
  \item[\textup{(2)}] Let $\sigma \in \R_{\geq 0}$.  If one of
    $\zeta_{G \vert \theta}(\sigma)$ or $\zeta_{H \vert
      \theta}(\sigma)$ converges, then so does the other, and
    \[
    \lvert G:H \rvert^{-1-\sigma} \zeta_{H \vert \theta}(\sigma) \leq
    \zeta_{G \vert \theta}(\sigma) \leq \lvert G:H \rvert \zeta_{H
      \vert \theta}(\sigma).
    \]
  \item[\textup{(3)}] Suppose that $\min \{ \dim \rho/\dim \theta \mid
    \rho \in \Irr(G \vert \theta) \text{ with } \dim \rho > \dim
    \theta \} > \lvert G : H \rvert$.  Let $\sigma \in \R_{\geq 0}$.
    If one of $\zeta_{G \vert \theta}'(\sigma)$ or $\zeta_{H \vert
      \theta}'(\sigma)$ converges, then so does the other, and
    \[
    \lvert G:H \rvert^{-1-\sigma} \zeta_{H \vert \theta}'(\sigma) \leq
    \zeta_{G \vert \theta}'(\sigma) \leq \lvert G:H \rvert \zeta_{H
      \vert \theta}' (\sigma).
    \]
  \end{list}
\end{lem}

\begin{proof} Consider the bipartite graph $B$ whose vertex set is
  $\Irr(G \vert \theta) \sqcup \Irr(H \vert \theta)$ and which has the
  property that there is an edge between $\rho_1 \in \Irr(G \vert
  \theta)$ and $\rho_2 \in \Irr(H \vert \theta)$ if and only if
  $\rho_2$ is a constituent of $\Res^G_H \rho_1$.  By Nakayama's
  generalisation of Frobenius reciprocity, the latter condition is
  equivalent to $\rho_1$ being a constituent of $\Ind^G_H \rho_2$;
  see~\cite[Chapter~VII, Section~4]{HuBl}. This implies that
  \begin{list}{$\circ$}{\setlength{\leftmargin}{\mylistleftmargin}
      \setlength{\labelwidth}{10pt}\setlength{\itemsep}{0pt}
      \setlength{\parsep}{1pt}}
  \item the degree of every vertex of $B$ is positive and
    bounded by $\lvert G:H \rvert$ and that
  \item if $\rho_1 \in \Irr(G \vert \theta)$ and $\rho_2 \in \Irr(H
    \vert \theta)$ are connected by an edge, then $\dim \rho_2\leq
    \dim\rho_1 \leq \lvert G:H \rvert \dim\rho_2$.
  \end{list}
  Let $m\in\N$.  Write $\Irr(G \vert \theta)_m \subset \Irr(G \vert
  \theta)$ for the subset consisting of representations of dimension
  at most $m \dim\theta$, and likewise define $\Irr(H \vert
  \theta)_m$.  Then $\Irr(G \vert \theta)_m$ is contained in the set
  of neighbors of $\Irr(H \vert \theta)_m$, hence
  \[
  R_m(G \vert \theta) = \lvert \Irr(G \vert \theta)_m \rvert \leq
  \lvert G:H \rvert \, \lvert \Irr(H \vert \theta)_m \rvert = \lvert G
  : H \rvert R_m(H \vert \theta).
  \]
  Similarly, $\Irr(H \vert \theta)_{\lfloor m/\lvert G:H \rvert
    \rfloor}$ is contained in the set of neighbors of $\Irr(G \vert
  \theta)_m$, hence
  \[
  R_{\lfloor m / \lvert G:H \rvert \rfloor}(H \vert \theta) = \lvert
  \Irr(H \vert \theta)_{\lfloor m/\lvert G:H \rvert \rfloor} \rvert
  \leq \lvert G:H \rvert \, \lvert \Irr(G \vert \theta)_m \rvert =
  \lvert G:H \rvert R_m(G \vert \theta).
  \]
  This proves~(1).  A similar argument yields~(2) and (3).  For (3)
  one observes that, if $\min \{ \dim \rho/\dim \theta \mid \rho \in
  \Irr(G \vert \theta) \text{ with } \dim \rho > \dim \theta \} >
  \lvert G : H \rvert$ and if $\rho_1 \in \Irr(G \vert \theta)$ and
  $\rho_2 \in \Irr(H \vert \theta)$ are connected by an edge, then
  $\dim \rho_1 > \dim \theta$ if and only if $\dim \rho_2 > \dim
  \theta$.
\end{proof}

\section{Base Change for Finite Groups of Lie Type}
\label{sec:BC.finite}

In this section we prove Theorem~\ref{thm:BC.finite.new}, a variant of
Theorem~\ref{thm:zeta.approx} for zeta functions of semi-simple
algebraic groups over finite fields. In Section~\ref{sec:applications}
we apply our results to finite quotients of arithmetic groups.

\subsection{Finite Groups of Lie Type}

Given a root system $\Phi$, we denote by $\rk \Phi$ its rank and by
$\Phi^+$ a choice of positive roots. By a Lie type $\mathcal{L}$ we
mean a pair $(\Phi,\tau)$, where $\Phi$ is a root system with an
automorphism~$\tau$ stabilising $\Phi^+$.  We say that a reductive
algebraic group $\mathbf{G}$ defined over a finite field $\F_q$ has
Lie type $\mathcal{L} = (\Phi,\tau)$ if the absolute root system of
$\mathbf{G}$ associated to an $\F_q$-rational and maximally $\mathbb{F}_q$-split maximal torus of
$\mathbf{G}$ is $\Phi$ and the action of the Frobenius automorphism
$\Frob_q$ on $\Phi$ is given by~$\tau$; compare~\cite[Chapter~3]{DM}.
For each finite field~$\F_q$, the connected semi-simple algebraic
groups over $\F_q$ are parametrized up to isogeny by their Lie types.
Given a Lie type $\mathcal{L} = (\Phi,\tau)$ we write
$\mathcal{L}^\mathrm{sp} = (\Phi,\Id)$ for the Lie type of a split
group with underlying root system~$\Phi$.  Recall the notation
$\mathcal{A}^+$ and $\xi_{a,q}$ from Definition~\ref{defn:Trop}.  In
this section we prove the following result.

\begin{thm} \label{thm:BC.finite.new} Let $\Phi$ be a non-trivial root
  system, and let $\boldsymbol{\mathcal{L}}_\Phi$ denote the
  collection of Lie types with underlying root system~$\Phi$.  Let
  $\boldsymbol{\mathcal{Q}}$ denote the set of all prime powers.  Then
  there exist $C \in \R$, $m \in \N$, $a(\Phi) \in \mathcal{A}^+$, and
  $a(\mathcal{L},q) \in \mathcal{A}^+$ for $(\mathcal{L},q) \in
  \boldsymbol{\mathcal{L}}_\Phi \times \boldsymbol{\mathcal{Q}}$ such
  that the following hold:
  \begin{list}{$\circ$}{\setlength{\leftmargin}{\myenumilistleftmargin}
      \setlength{\labelwidth}{20pt} \setlength{\itemsep}{0pt}
      \setlength{\parsep}{1pt}}
  \item[$\textup{(1)}$] $a(\mathcal{L},q) \subset a(\Phi)$ for all
    $(\mathcal{L},q) \in \boldsymbol{\mathcal{L}}_\Phi \times
    \boldsymbol{\mathcal{Q}}$,
  \item[$\textup{(2)}$] %$a(L^\mathrm{sp},q) = a(\Phi)$ for all $(\mathcal{L},q)
    % \in \boldsymbol{\mathcal{L}}_\Phi \times
    % \boldsymbol{\mathcal{Q}}$ with $q \equiv 1$
    % modulo~$m$,
    $a((\Phi,\Id),q) = a(\Phi)$ for all $q \equiv_m 1$,
  \item[$\textup{(3)}$] for every connected semi-simple algebraic
    group $\mathbf{G}$ defined over a finite field~$\F_q$
    which has Lie type $L \in \boldsymbol{\mathcal{L}}_\Phi$,
    \[
    \zeta_{\mathbf{G}(\F_q)} - \lvert \mathbf{G}(\F_q)
    / [\mathbf{G}(\F_q),\mathbf{G}(\F_q)] \rvert
    \sim_C \xi_{a(\mathcal{L},q),q}.
    \]
  \end{list}
  Moreover, $(\rk \Phi,\lvert \Phi^+ \rvert) \in a(\Phi)$, and every
  connected semi-simple algebraic group $\mathbf{G}$ defined over a
  finite field $\F_q$ with absolute root system $\Phi$
  satisfies
  \begin{equation} \label{eq:thm3.2-generous-estimate}
    \zeta_{\mathbf{G}(\F_q)}(s) \sim_C 1 + q^{\rk \Phi -
      \lvert \Phi^+ \rvert s}.
  \end{equation}
\end{thm}

\begin{rem} \label{rem:BC.finite} \uri{We remark that if $\mathbf{G}$
    is a connected, simply connected semi-simple algebraic group over
    $\F_q$ and $q>3$, then $\mathbf{G}(\F_q)$ is perfect and therefore
    $\lvert \mathbf{G}(\F_q) / [\mathbf{G}(\F_q),\mathbf{G}(\F_q)]
    \rvert = 1$. Indeed, for simple groups this is a result of Tits
    \cite{Tits} (see also \cite[Theorem~24.17]{MT}), and the
    semi-simple case follows by taking products.}
  
  % This follows, for instance, from the proof of
  % Theorem~\ref{thm:BC.finite.new}: we show that $\lvert
  % \mathbf{G}(\F_q) / [\mathbf{G}(\F_q),\mathbf{G}(\F_q)] \rvert =
  % \lvert Z(\mathbf{G}^*(\F_q)) \rvert$ for large~$q$, where
  % $\mathbf{G}^*$ is the dual group of $\mathbf{G}$.  As $\mathbf{G}$
  % is simply connected, $\mathbf{G}^*$ is adjoint and $\lvert
  % Z(\mathbf{G}^*(\F_q)) \rvert = 1$.
  % Compare~\cite[Remark~13.31]{DM}
  % for simply connected groups, or~\cite{Tits}.
\end{rem}

% Original remark
% \begin{rem} \label{rem:BC.finite} If $\mathbf{G}$ is a connected,
%   simply connected semi-simple algebraic group over $\F_q$ and $q$
%   is
%   sufficiently large, then $\lvert \mathbf{G}(\F_q) /
%   [\mathbf{G}(\F_q),\mathbf{G}(\F_q)] \rvert = 1.$ This follows, for
%   instance, from the proof of Theorem~\ref{thm:BC.finite.new}: we
%   show that $\lvert \mathbf{G}(\F_q) /
%   [\mathbf{G}(\F_q),\mathbf{G}(\F_q)] \rvert = \lvert
%   Z(\mathbf{G}^*(\F_q)) \rvert$ for large~$q$, where
%   $\mathbf{G}^*$ is the dual group of $\mathbf{G}$.  As $\mathbf{G}$
%   is simply connected, $\mathbf{G}^*$ is adjoint and $\lvert
%   Z(\mathbf{G}^*(\F_q)) \rvert = 1$.
%   Compare~\cite[Remark~13.31]{DM}
%   for simply connected groups, or~\cite{Tits}.
%\end{rem}

\begin{exa} The representations of the special linear group
  $\SL_2(\F_q)$ over a finite field $\F_q$ are well known; e.g.\
  see~\cite[Chapter~15]{DM}.  In particular, we have
  $\zeta_{\SL_2(\F_q)}(s)-1 \sim_2 q^{1-s}$ for sufficiently
  large~$q$.  Consider the following semi-simple algebraic groups
  defined over $\F_q$ whose absolute root systems are both $A_1 \times
  A_1$: $\mathbf{G}_1 = \SL_2 \times \SL_2$ and $\mathbf{G}_2 =
  \mathcal{R}_{\F_{q^2} \vert \F_q} \SL_2$, the restriction of scalars
  of $\SL_2$ defined over a quadratic extension.  Then
  $\mathbf{G}_1(\F_q) = \SL_2(\F_q) \times \SL_2(\F_q)$ and
  $\zeta_{\mathbf{G}_1(\F_q)}(s) - 1 \sim_4 q^{1-s} + q^{2-2s}$,
  whereas $\mathbf{G}_2(\F_q) = \SL_2(\F_{q^2})$ and
  $\zeta_{\mathbf{G}_2(\F_q)}(s) - 1 \sim_2 q^{2-2s}$.  Since
  $q^{1-s}+q^{2-2s} \not \sim_C q^{2-2s}$ for any fixed $C$ but
  unbounded~$q$, we see that the inclusion in part~(1) of
  Theorem~\ref{thm:BC.finite.new} can be strict.
\end{exa}

% \begin{exa}
%   Include the example of groups of type $G_2$ to motivate part (2) of
%   the theorem.\footnote{BK: perhaps add details after everything else
%     is sorted out.}
% \end{exa}

Whilst Theorem~\ref{thm:BC.finite.new} is a result on semi-simple
groups, it leads to the following corollary on reductive groups.

\begin{cor} \label{cor:finite.reductive.zeta.approx} Let
  $\mathbf{G}$ be a connected reductive algebraic group defined over a
  finite field~$\F_q$ with absolute root system~$\Phi$.  There
  is a constant~$D \in \R$, depending only on $\Phi$ and the dimension
  of $\mathbf{G}$, such that $\zeta_{\mathbf{G}(\F_q)}(s)
  \sim_D q^{\dim Z(\mathbf{G})} (1 + q^{\rk \Phi - \lvert \Phi^+
    \rvert s})$.
\end{cor}

The proof of Theorem~\ref{thm:BC.finite.new} and its
Corollary~\ref{cor:finite.reductive.zeta.approx} is given in
Section~\ref{subsubsec:BC.finite.new}, and prepared in the preceding
sections. It is based on Lusztig's classification of irreducible
representations of finite groups of Lie type; e.g.,
see~\cite[Chapter~13]{DM}.  We write $\mathbf{G}^*$ for the dual group
of a connected reductive group $\mathbf{G}$ and recall that, while
$\mathbf{G}$ and $\mathbf{G}^*$ have isomorphic Weyl groups $W \cong
W^*$, the absolute root system $\Phi^{\vee}$ of $\mathbf{G}^*$ is dual
to $\Phi$ in the sense that the roles of long and short roots are
interchanged.  We also note that $\mathbf{G}$ is simply connected
respectively\ adjoint if and only if $\mathbf{G}^*$ is adjoint
respectively\ simply connected.  Since equivalence classes of
irreducible representations of $\mathbf{G}(\F_q)$ are parametrized by
the corresponding characters, we use the notation
$\Irr(\mathbf{G}(\F_q))$ in a flexible way to denote also the set of
irreducible complex characters of~$\mathbf{G}(\F_q)$.  The set
$\Irr(\mathbf{G}(\F_q))$ is the disjoint union of certain Lusztig
series $\mathcal{E}(\mathbf{G}(\F_q),(g))$ so that
\begin{equation} \label{equ:Lusztig-zeta}
\zeta_{\mathbf{G}(\F_q)}(s) = \sum_{(g) \subset \mathbf{G}
  ^*(\F_q)} \; \sum_{\chi \in
  \mathcal{E}(\mathbf{G}(\F_q),(g))} \chi(1)^{-s},
\end{equation}
where the outer sum ranges over $\mathbf{G}^*$-conjugacy classes of
semi-simple elements in~$\mathbf{G} ^*(\F_q)$.

\subsubsection{Unipotent Zeta Functions}\label{subsubsec:c}
 The elements
of $\mathcal{E}(\mathbf{G}(\F_q),(1))$ are known as the
unipotent characters of~$\mathbf{G}(\F_q)$. We set%, and we set
\[
\zeta_{\mathbf{G}(\F_q)}^\mathrm{unip}(s) = \sum_{\chi \in
  \mathcal{E}(\mathbf{G}(\F_q),(1))} \chi(1)^{-s}.
\]

\begin{prop} \label{lem:BC.unipotent.representations} Let $\Phi$ be a
  non-trivial root system and let $\boldsymbol{\mathcal{L}}_\Phi$
  denote the collection of Lie types with underlying root
  system~$\Phi$.  Then there are $C \in \R$ and
  $b^\mathrm{c}(\mathcal{L}) \in \mathcal{A}^+$ for $\mathcal{L} \in
  \boldsymbol{\mathcal{L}}_\Phi$ such that the following hold:
  \begin{list}{}{\setlength{\leftmargin}{\myenumilistleftmargin}
      \setlength{\labelwidth}{20pt} \setlength{\itemsep}{0pt}
      \setlength{\parsep}{1pt}}
  \item[\textup{(1)}] $b^\mathrm{c}(\mathcal{L}) \subset
    b^\mathrm{c}(\mathcal{L}^\mathrm{sp})$ for all $L \in
    \boldsymbol{\mathcal{L}}_\Phi$,
  \item[\textup{(2)}] for every connected reductive algebraic group
    $\mathbf{G}$ defined over a finite field $\F_q$ which has Lie
    type~$\mathcal{L} \in \boldsymbol{\mathcal{L}}_\Phi$,
    \[
    \zeta_{\mathbf{G}(\F_q)}^\mathrm{unip} - 1 \sim_C
    \xi_{b^\mathrm{c}(\mathcal{L}),q} \qquad \text{and} \qquad
    \zeta_{\mathbf{G}(\F_q)}^\mathrm{unip} \sim_C 1.
    \]
  \end{list}
\end{prop}

\begin{proof} Throughout, all constants -- be they real or elements of
  $\mathcal{A}^+$ -- depend only on~$\Phi$ or a given Lie type
  $\mathcal{L} \in \boldsymbol{\mathcal{L}}_\Phi$.  Let $\mathbf{G}$
  be a connected reductive algebraic group of Lie type $\mathcal{L} =
  (\Phi,\tau) \in \boldsymbol{\mathcal{L}}_\Phi$.  Then
  $\mathbf{G}/Z(\mathbf{G})$ is a semi-simple group over $\F_q$ of
  adjoint type and, by~\cite[Proposition~13.20]{DM}, every unipotent
  character of $\mathbf{G}(\F_q)$ factors through
  $(\mathbf{G}/Z(\mathbf{G}))(\F_q)$.

  Therefore we may assume that $\mathbf{G}$ is semi-simple of adjoint
  type and thus a direct product of $\F_q$-simple groups
  $\mathbf{S}_i$ of Lie type $\mathcal{L}_i = (\Phi_i,\tau_i)$, say,
  where $i \in \{1, \ldots, m\}$.  The unipotent characters of
  $\mathbf{G}(\F_q) = \mathbf{S}_1(\F_q) \times \cdots \times
  \mathbf{S}_m(\F_q)$ are the irreducible characters that appear in
  the Deligne--Lusztig character $R_\mathbf{T}^\mathbf{G}(1)$, where
  $\mathbf{T}$ is a maximal torus defined over~$\F_q$.  The
  generalised character $R_\mathbf{T}^\mathbf{G}(1)$ arises from the
  action of $\mathbf{G}(\F_q) \times \mathbf{T}(\F_q)$ on the
  cohomologies of the Deligne--Lusztig variety $\mathbf{X} =
  \boldsymbol{\ell}^{-1}(\mathbf{U})$, where $\boldsymbol{\ell}$
  denotes the Lang map and $\mathbf{U}$ is the unipotent radical of a
  Borel subgroup containing~$\mathbf{T}$; see~\cite[Section~11]{DM}.
  This construction is compatible with taking products, hence
  $\mathbf{X}$ can be taken to be the product of the Deligne--Lusztig
  varieties of the groups $\mathbf{S}_i(\F_q)$.  The K\"unneth formula
  now implies that
  \begin{equation} \label{equ:unip-product}
    \zeta^\mathrm{unip}_{\mathbf{G}(\F_q)} = \prod_{i=1}^m
    \zeta^\mathrm{unip}_{\mathbf{S}_i(\F_q)}.
  \end{equation}

  Fix $i \in \{1,\ldots,m\}$.  Since $\mathbf{S}_i$ is $\F_q$-simple,
  $\tau_i$ permutes the irreducible components $\Phi_{i,1}$,
  $\Phi_{i,2}$, \ldots, $\Phi_{i,n(i)}$ of $\Phi_i$ transitively.
  Thus $\tau_i^{n(i)}$ restricts to an automorphism of~$\Psi_i =
  \Phi_{i,1}$.  Writing $\widetilde{\mathbf{S}}_i$ for the adjoint
  absolutely simple algebraic group over $\F_{q^{n(i)}}$ of Lie type
  $\widetilde{\mathcal{L}}_i = (\Psi_i,\tau_i^{n(i)})$, we have
  $\mathbf{S}_i(\F_q) = \widetilde{\mathbf{S}}_i(\F_{q^{n(i)}})$.  By
  \cite[Sections~13.8 and~13.9]{Ca}, the number $k(i)$ of unipotent
  characters of $\widetilde{\mathbf{S}}_i(\F_{q^{n(i)}})$ depends only
  on~$\widetilde{\mathcal{L}}_i$.  This shows that
  $\zeta_{\mathbf{S}_i(\F_q)}^\mathrm{unip} \lesssim_{C_i} 1$ for a
  suitable constant~$C_i \in \R$.  Since the trivial character is
  unipotent, we also get $1 \lesssim_1
  \zeta_{\mathbf{S}_i(\F_q)}^\mathrm{unip}$.  This yields
  $\zeta_{\mathbf{S}_i(\F_q)}^\mathrm{unip} \sim_{C_i} 1$ and thus
  $\zeta^\mathrm{unip}_{\mathbf{G}(\F_q)} \sim_{C_1\cdots C_m} 1$,
  using~\eqref{equ:unip-product}.

  Furthermore, looking through the tables in~\cite[Sections~13.8
  and~13.9]{Ca}, one sees that there are polynomials
  $f_{i,1},\ldots,f_{i,k(i)} \in \Q[X]$ depending only on
  $\widetilde{\mathcal{L}}_i$ such that the degrees of the unipotent
  characters of $\widetilde{\mathbf{S}}_i(\F_{q^{n(i)}})$ are
  precisely $f_{i,1}(q^{n(i)}),\ldots,f_{i,k(i)}(q^{n(i)})$.  Moreover, only one of
  these polynomials has degree zero, namely the constant
  polynomial~$1$ giving the degree of the trivial character. We may
  assume that $f_{i,1}=1$ and define
  \[
  b^\mathrm{c}(\widetilde{\mathcal{L}}_i) = \{ (0, \deg f_{i,j}) \mid
  2 \leq j \leq k(i) \}\in\mathcal{A}^+.
  \]

  We show below that, while the degrees of the $f_{i,j}$ generally
  depend on $\widetilde{\mathcal{L}}_i$ and not only on~$\Psi_i$, the
  set of degrees in the split case
  $\widetilde{\mathcal{L}}_i^{\mathrm{sp}} = (\Psi_i,\Id)$ is a
  superset of the set of degrees in the twisted cases.  Granted that
  this is so, we continue and define $b^\mathrm{c}(L_i) =
  b^\mathrm{c}(\widetilde{\mathcal{L}}_i)^{(n(i))}$ so that for a
  suitable constant $C_i' \in \R$ we obtain
  \begin{equation} \label{equ:each-i-strong}
    \zeta_{\mathbf{S}_i(\F_q)}^\mathrm{unip} - 1 =
    \zeta_{\widetilde{\mathbf{S}}_i(\F_{q^{n(i)}})}^\mathrm{unip} - 1
    \sim_{C_i'} \xi_{b^\mathrm{c}(\widetilde{\mathcal{L}}_i),q^{n(i)}} =
    \xi_{b^\mathrm{c}(L_i),q}.
  \end{equation}
  We set $b^\mathrm{c}(\mathcal{L}) = b^\mathrm{c}(L_1) * \dots *
  b^\mathrm{c}(L_m)$.  From our claim regarding the degrees of the
  polynomials $f_{i,j}$ and Remark~\ref{rem:xi} we deduce that, for
  each $i \in \{1,\ldots,m\}$,
  \begin{equation} \label{equ:each-i-split} b^\mathrm{c}(L_i) =
    b^\mathrm{c}(\widetilde{\mathcal{L}}_i)^{(n(i))} \subset
    b^\mathrm{c}(\Psi_i,\Id)^{(n(i))} \subset
    b^\mathrm{c}(\Psi_i,\Id)^{*n(i)} =
    b^\mathrm{c}(\mathcal{L}_i^\mathrm{sp}).
  \end{equation}
  In the special case $m=1$ this establishes assertion (1) of the
  proposition, and \eqref{equ:each-i-strong} directly yields the
  remaining first assertion in~(2).  In the general case we conclude
  the proof based on~\eqref{equ:unip-product} and Remark~\ref{rem:xi}.

  \medskip

  It remains to justify the claim about the degrees of unipotent
  characters for split and twisted finite groups of Lie type with the
  same underlying absolute root system.  Simplifying the notation used
  above, let $\mathbf{S}$ (rather than $\widetilde{\mathbf{S}}_i$) be
  a connected, adjoint absolutely simple algebraic group defined over
  $\F_q$ (rather than $\F_{q^{n(i)}}$), and let
  $\mathbf{S}^\mathrm{sp}$ denote a split form of~$\mathbf{S}$
  over~$\F_q$.  We consider the possible groups case by case.  By the
  remarks following~\cite[Theorem~3.17]{DM}, the relevant twisted Lie
  types are $^2A_n$, $^2D_n$, $^3D_4$, and~$^2E_6$, where the left
  index represents as customary the order of the automorphism of the
  root system.  (In our setup, the Suzuki and Ree groups do not occur.
  In fact, the analogous statement for groups of Lie type $F_4$ and
  $^2F_4$ is incorrect; see Remark~\ref{rem:unipotent.exceptional}.)

  \medskip

  \noindent \textit{Case $1$}: $\mathbf{S}$ has Lie type $^2A_n$ and
  $\mathbf{S}^\mathrm{sp}$ has Lie type $A_n$.  The unipotent
  characters of each group, $\mathbf{S}(\F_q)$ and
  $\mathbf{S}^\mathrm{sp}(\F_q)$, are parametrized by partitions
  $\alpha$ of~$n+1$.  The degree $\chi_\alpha(1)$ of the unipotent
  character $\chi_\alpha$ associated to $\alpha$ in either case is
  given by a polynomial in $q$ whose degree only depends on~$\alpha$;
  see~\cite[p.~465]{Ca}. This proves that $b^\mathrm{c}(^2A_n) =
  b^\mathrm{c}(A_n)$.

  \medskip

  \noindent \textit{Case $2$}: $\mathbf{S}$ has Lie type $^2D_n$ and
  $\mathbf{S}^\mathrm{sp}$ has Lie type $D_n$. The unipotent
  characters of $\mathbf{S}^\mathrm{sp}(\F_q)$ are
  parametrized by so-called Lustzig symbols (or just ``symbols'' in
  the terminology of~\cite{Ca}). These are certain pairs $(S,T)$ of
  finite subsets of $\mathbb{Z}_{\geq 0}$ such that $\lvert S \rvert -
  \lvert T \rvert \equiv 0 \pmod{4}$, where pairs of the form $(S,S)$
  correspond to pairs of characters rather than single characters, but
  this subtlety is irrelevant for our purposes.  Similarly, the
  unipotent characters of $\mathbf{S}(\F_q)$ are parametrized
  by pairs $(S,T)$ such that $\lvert S \rvert - \lvert T \rvert \equiv
  2 \pmod{4}$.  Let $(S,T)$ be a pair of the second kind and suppose
  without loss of generality that there exists an element $a \in S
  \smallsetminus T$.  Then $(S',T') = (S\smallsetminus \left\{ a
  \right\},T \cup \left\{ a \right\})$ is a pair of the first kind.
  Inspection of the formulae giving the degrees $\chi_{(S,T)}(1)$ and
  $\chi_{(S',T')}(1)$ of the unipotent characters $\chi_{(S,T)}$ of
  $\mathbf{S}(\F_q)$ associated with $(S,T)$ and $\chi_{(S',T')}$ of
  $\mathbf{S}^\mathrm{sp}(\F_q)$ associated with $(S',T')$ yields:
  the polynomials in $q$ giving $\chi_{(S,T)}(1)$ and
  $\chi_{(S',T')}(1)$ have equal degrees.  For details
  see~\cite[p.~471 and 475f.]{Ca}.  This proves that
  $b^\mathrm{c}(^2D_n) \subset b^\mathrm{c}(D_n)$.

 \medskip

 \noindent \textit{Case $3$}: $\mathbf{S}$ has Lie type $^3D_4$ and
 $\mathbf{S}^\mathrm{sp}$ has Lie type $D_4$. The set of degrees of
 the polynomials expressing the degrees of unipotent characters of
 $\mathbf{S}(\F_q)$ is $\left\{ 0,5,9,11,12 \right\}$
 (see~\cite[p.~478]{Ca}), whereas the corresponding set for
 $\mathbf{S}^\mathrm{sp}(\F_q)$ is $\left\{0,5,6,9,10,11,12
 \right\}$.  This shows that $b^\mathrm{c}(^3D_4)\subset
 b^\mathrm{c}(D_4)$.

 \medskip

 \noindent \textit{Case $4$}: $\mathbf{S}$ has Lie type $^2E_6$ and
 $\mathbf{S}^\mathrm{sp}$ has Lie type $E_6$.  The 30 unipotent
 characters of these groups are in degree-preserving bijective
 correspondence; cf.~\cite[p.~480f.]{Ca}. This shows that
 $b^\mathrm{c}(E_6) = b^\mathrm{c}(^2E_6)$.
\end{proof}

\begin{rem}\label{rem:unipotent.exceptional}
  %We remark that
  Proposition~\ref{lem:BC.unipotent.representations}
  extends only partly to the remaining groups of Lie type, the Suzuki
  and Ree groups.  For the groups $B_2(q)$ and $^2B_2(q^2)$ the
  degrees of the polynomials giving the degrees of the unipotent
  characters coincide.  The same is true for the groups $G_2(q)$ and
  $^2G_2(q^2)$.  But the degrees of the polynomials giving the
  unipotent character degrees of the Ree groups $^2F_4(q^2)$ are
  $\{0,9,14,20,21,22,24\}$ (see~\cite[p.~489]{Ca}), whereas the
  corresponding degrees for the split groups $F_4(q)$ are
  $\{0,11,15,20,21,22,23,24\}$ (see~\cite[p.~479]{Ca}).
\end{rem}

\begin{rem}
  In fact, the sets $b^\mathrm{c}(\mathcal{L})$ in
  Proposition~\ref{lem:BC.unipotent.representations} are all contained
  in $\{0\}\times\mathbb{Z}_{>0}$, so
  $\xi_{b^\mathrm{c}(\mathcal{L}),q}$ could be replaced by
  $q^{-m(\mathcal{L})s}$, where $m(\mathcal{L}) =
  \min\{m\in\mathbb{Z}_{>0}\mid (0,m)\in
  b^{\mathrm{c}}(\mathcal{L})\}$. Moreover, it is conceivable that
  $m(\mathcal{L})$ actually only depends on the root system $\Phi$. As
  the proof of the proposition shows, this reduces to a question
  regarding the degrees of unipotent characters of groups of type
  $D_n$. More precisely, it would follow if the degrees of the
  polynomials giving the degrees of the unipotent characters of the
  (untwisted) groups of type $D_n$ indexed by special Lusztig symbols
  of the form $(S,S)$ were already among the corresponding degrees
  obtained from Lusztig symbols of the form $(S,T)$ with $S\neq T$.
\end{rem}

\subsubsection{Connected Centralizers of Semi-Simple Elements}
In preparation for the proof of Theorem~\ref{thm:BC.finite.new} we
record some auxiliary facts about the connected centralisers of
semi-simple elements in algebraic groups over finite fields.

Let $\Phi$ be a root system.  We consider a connected split reductive
algebraic group $\mathbf{G}$ over a finite field $\F_q$ whose
absolute root system, associated to some $\F_q$-rational split
maximal torus~$\mathbf{T}$, is isomorphic to and identified
with~$\Phi$. We denote by $\F_q^\alg$ the algebraic
closure of~$\F_q$.  The connected centraliser
$C_\mathbf{G}(g)^\circ$ of any semi-simple element $g \in
\mathbf{G}(\F_q^\alg)$ is a reductive subgroup of maximal rank
in~$\mathbf{G}$.  Indeed, every semi-simple element of $\mathbf{G}$ is
conjugate to an element of~$\mathbf{T}$.  Furthermore the connected
centraliser of $g \in \mathbf{T}(\F_q^\alg)$ is the reductive
group $\langle \mathbf{T}, \mathbf{U}_\alpha \mid \alpha \in \Psi_g
\rangle$ with root system $\Psi_g = \{ \alpha \in \Phi \mid
\alpha(g)=1 \}$, where each $\mathbf{U}_\alpha$ denotes the root
subgroup associated to~$\alpha$; see~\cite[Proposition~2.3]{DM}.  For
every extension $\mathbb{E}$ of $\F_q$ we put
\begin{equation}\label{equ:C}
  \mathcal{C}(\mathbf{G},\mathbf{T},\mathbb{E}) = \{ \Psi_g \mid g \in
  \mathbf{T}(\mathbb{E}) \}.
\end{equation}
Furthermore, we set
\begin{equation}\label{equ:C_Phi}
  \mathcal{C}(\Phi) = \bigcup_{\mathbf{G},\mathbf{T},\F_q}
  \mathcal{C}(\mathbf{G},\mathbf{T},\F_q), 
\end{equation}
where $\mathbf{G}$, $\mathbf{T}$ and $\F_q$ range over all possible
choices for the fixed root system~$\Phi$.  In
Proposition~\ref{pro:centralizers.split} we show that, in fact,
$\mathcal{C}(\mathbf{G},\mathbf{T},\F_q) = \mathcal{C}(\Phi)$ as long
as $\F_q$ contains a sufficient supply of roots of unity.

\begin{lem} \label{lem:centr-isogeny} Let $\phi \colon \mathbf{G}
  \rightarrow \mathbf{H}$ be an isogeny, i.e., a surjective morphism
  with finite central kernel, between algebraic groups defined over a
  finite field~$\F_q$.  If $g \in \mathbf{G}(\F_q)$, then $\phi$
  maps $C_\mathbf{G}(g)^\circ$ onto $C_\mathbf{H}(\phi(g))^\circ$.
\end{lem}

\begin{proof}
  Write $h = \phi(g)$ and observe that
  \begin{equation}\label{equ:inclusions}
    \bigcap_{\tilde g \in \mathbf{G}(\F_q^\alg) \text{, } \phi(\tilde g) = h}
    C_\mathbf{G}(\tilde g) = C_\mathbf{G}(g) \subset \phi^{-1}
    (C_\mathbf{H}(h)).
  \end{equation}
  Furthermore, $\phi^{-1} (C_\mathbf{H}(h)) / \bigcap_{\phi(\tilde g)
    = h} C_\mathbf{G}(\tilde g)$ acts faithfully by conjugation on the
  finite set $g \ker(\phi)$.  This shows that the inclusion
  in~\eqref{equ:inclusions} is of finite index, and hence
  $C_\mathbf{G}(g)^\circ$ is a finite index subgroup
  of~$\phi^{-1}(C_\mathbf{H}(h))$.  Since
  $\phi(C_\mathbf{G}(g)^\circ)$ is connected and of finite index in
  $C_\mathbf{H}(h)$, we conclude that $\phi(C_\mathbf{G}(g)^\circ) =
  C_\mathbf{H}(h)^\circ$.
\end{proof}

\begin{prop} \label{pro:centralizers.split} Let $\Phi$ be a root
  system.  There exists $m \in \N$ such that for every connected
  split reductive algebraic group $\mathbf{G}$ over a finite field
  $\F_q$ whose absolute root system, associated to an
  $\F_q$-rational split maximal torus~$\mathbf{T}$, is isomorphic to
  and identified with~$\Phi$ the following holds: if $\F_q$
  contains primitive $m$th roots of unity, equivalently $q \equiv_m
  1$, then $\mathcal{C}(\mathbf{G},\mathbf{T},\F_q) =
  \mathcal{C}(\Phi)$.

  In particular, if $\mathbf{G}$ is as above and $q \equiv_m 1$, then
  for every semi-simple element $g \in
  \mathbf{G}(\F_q^\alg)$, there is a semi-simple element
  $g_0 \in \mathbf{G}(\F_q)$ such that $C_\mathbf{G}(g)^\circ$
  and $C_\mathbf{G}(g_0)^\circ$ are $\mathbf{G}$-conjugate.
\end{prop}

\begin{proof}
  Fix $\Psi \in \mathcal{C}(\Phi)$.  Clearly, it suffices to show that
  there exists $m(\Psi) \in \N$ so that $\Psi$ arises as the absolute
  root system of a connected centraliser in each instance of
  $(\mathbf{G},\mathbf{T},\F_q)$ with $q \equiv_{m(\Psi)} 1$.

  Let $\mathbf{G}$ be a connected split reductive algebraic group over
  a finite field $\F_q$, with absolute root system~$\Phi$.
  The group $\mathbf{G}$ is the image of an isogeny from the direct
  product of the connected split semi-simple group
  $[\mathbf{G},\mathbf{G}]$ and $Z(\mathbf{G})^\circ$, a torus.  Using
  Lemma~\ref{lem:centr-isogeny}, we may concentrate on the case that
  $\mathbf{G}$ is semi-simple and hence an isogenous image of a direct
  product of split simply-connected almost simple algebraic groups.
  Applying once more Lemma~\ref{lem:centr-isogeny}, we may restrict to
  the case that $\mathbf{G}$ is simply-connected almost simple, with
  root datum $(X,\Phi,\mathbb{Z}\Phi^{\vee},\Phi^{\vee})$, say.  The
  root lattice $\mathbb{Z}\Phi$ is a sublattice of finite index in the
  free $\mathbb{Z}$-lattice~$X$. We fix free generators
  $\chi_1,\dots,\chi_{\ell}$ of $X$.

  Let $\mathbf{T} \subset \mathbf{G}$ be an $\F_q$-rational split
  maximal torus.  We identify $\Phi$ with the root system and $X$ with
  the character group associated to~$\mathbf{T}$. Let $g \in
  \mathbf{T}(\F_q^\alg)$.  As noted before, the connected centraliser
  of $g$ is the reductive subgroup
  \[
  C_\mathbf{G}(g)^\circ = \langle \mathbf{T},
  \mathbf{U}_\alpha \mid \alpha \in \Psi_g \rangle
  \]
  with root system $\Psi_g = \{ \alpha \in \Phi \mid \alpha(g)=1 \}$.
  Furthermore, the multiplicative group $(\F_q^\alg)^*$ is
  isomorphic to the additive group $\Q_{p'} / \Z$, where $p$ denotes
  the characteristic of $\F_q$ and $\Q_{p'}$ the ring of all
  rational numbers whose denominator is not divisible by~$p$.  Under
  this isomorphism $\F_q^*$ corresponds to the additive group
  $\Z[1/(q-1)]/\Z$.

  The conditions $\alpha(g)=1$ for $\alpha \in \Psi_g$ and $\alpha(g)
  \neq 1$ for $\alpha \in \Phi \smallsetminus \Psi_g$, which
  determine~$C_\mathbf{G}(g)^\circ$, translate into a finite system of
  linear equations and inequations with integer coefficients in $\ell$
  variables $x_1, \ldots, x_\ell$, corresponding to the generators
  $\chi_1,\dots,\chi_{\ell}$ of~$X$.  The element $g$ corresponds to a
  solution $(a_1/b_1 + \Z,\ldots,a_\ell/b_\ell + \Z) \in
  (\Q_{p'}/\Z)^\ell$ of the system modulo~$\Z$.  Without loss of
  generality, we may assume that $\gcd(a_i,b_i) = 1$ for each $i \in
  \{1,\ldots,\ell\}$, and we set $m(\Psi_g) =
  \mathrm{lcm}(b_1,\ldots,b_\ell)$.  Then $g \in \mathbf{T}(\F_q)$ if
  and only if $\F_q$ contains primitive roots of unity of
  degree~$m(\Psi_g)$, equivalently $q \equiv_{m(\Psi_g)} 1$.

  As $\Psi \in \mathcal{C}(\Phi)$, there is an instance of
  $(\mathbf{G},\mathbf{T},\F_q)$ and $g \in \F_q$ as described above
  such that~$\Psi = \Psi_g$.  This single solution shows that $\Psi
  \in \mathcal{C}(\mathbf{G},\mathbf{T},\F_q)$ for all instances
  $(\mathbf{G},\mathbf{T},\F_q)$ with $q \equiv_{m(\Psi)} 1$.
\end{proof}

\begin{cor}\label{cor:C_Phi}
 There exists $c\in\N$, depending only on $\Phi$, such that,
 for all finite fields $\F_q$ of characteristic $p>c$,
 \[
 \mathcal{C}(\mathbf{G},\mathbf{T},\F_q^\alg) =
 \mathcal{C}(\Phi).
 \]
\end{cor}

\begin{rem} \label{rem:Deriziotis}
  Proposition~\ref{pro:centralizers.split} is related to more detailed
  investigations into semi-simple conjugacy classes and their
  centralizers in finite groups of Lie type, which were initiated by
  Carter and Deriziotis.  In particular, their results imply that
  $\mathcal{C}(\mathbf{G},\mathbf{T},\F_q^\alg) = \mathcal{C}(\Phi)$
  if the characteristic of $\F_q$ is greater than $5$;
  see~\cite[Proposition~2.3 and remarks]{De81}.  For an overview,
  including also subsequent developments we refer to \cite[Chapters~2
  and~8]{Hu} and the references given therein.
\end{rem}

\begin{cor} \label{cor:centralizers.uniform}
 Let $\Phi$ be a non-trivial root system, and let $\mathcal{C}(\Phi)$
 be as defined in~\eqref{equ:C_Phi}.  Let $W(\Phi)$ denote the Weyl
 group of $\Phi$.  Let $\mathbf{G}$ be a connected semi-simple
 algebraic group defined over a finite field $\F_q$ with
 absolute root system~$\Phi$.  Identify $\Phi$ with the root system
 associated to a maximal torus $\mathbf{T}$ of~$\mathbf{G}$, and let
 $\mathcal{C}(\mathbf{G},\mathbf{T},\F_q^\alg)$ be as defined
 in~\eqref{equ:C}.  Then there are connected subgroups
 $\mathbf{H}_\Psi$ of $\mathbf{G}$, not necessarily defined over
 $\F_q$ and indexed by $\Psi \in
 \mathcal{C}(\mathbf{G},\mathbf{T},\F_q^\alg)$, such that the
 following hold.
  \begin{list}{}{\setlength{\leftmargin}{\myenumilistleftmargin}
      \setlength{\labelwidth}{20pt} \setlength{\itemsep}{0pt}
      \setlength{\parsep}{1pt}}
  \item[\textup{(1)}] The groups $\mathbf{H}_\Psi$ form a -- typically
    redundant -- set of representatives for the $\mathbf{G}$-conjugacy
    classes of connected centralisers $C_\mathbf{G}(g)^\circ$ of
    semi-simple elements $g\in\mathbf{G}(\F_q^\alg)$.
  \item[\textup{(2)}] The data $(\dim Z(\mathbf{H}_\Psi), \dim
    \mathbf{H}_\Psi, \lvert \Psi^+ \rvert)$ depend only on~\ben{$\Phi$
      and} $\Psi$; \ben{more precisely, the first two entries satisfy
      $\dim Z(\mathbf{H}_\Psi) = \rk \Phi - \rk \Psi$} and $\dim
    \mathbf{H}_\Psi = \rk \Psi + 2 \lvert \Psi^+ \rvert$.
%   \item[\textup{(2)}] The data $(\dim Z(\mathbf{H}_\Psi), \dim
%     \mathbf{H}_\Psi, \lvert \Psi^+ \rvert)$ depend only
%     on~$\Psi$. Indeed, $\dim Z(\mathbf{H}_\Psi) = \lvert \comp(\Psi)
%     \rvert - 1$, where $\lvert \comp(\Psi) \rvert$ denotes the number
%     of irreducible components of~$\Psi$, and $\dim \mathbf{H}_\Psi =
%     \rk \Psi + 2 \lvert \Psi^+ \rvert$.
  \item[\textup{(3)}] For every $\Psi$, the number of
    $\mathbf{G}(\F_q)$-conjugacy classes of connected
    centralisers $C_\mathbf{G}(g)^\circ \subset \mathbf{G}$ of
    semi-simple elements $g \in \mathbf{G}(\F_q)$ that are
    $\mathbf{G}$-conjugate to $\mathbf{H}_\Psi$ is at most $\lvert
    W(\Phi) \rvert \lvert \mathcal{C}(\Phi) \rvert$.
  \end{list}
\end{cor}

\begin{proof}
 For each $\Psi \in \mathcal{C}(\mathbf{G},\mathbf{T},\F_q^\alg)$,
 choose $\mathbf{H}_\Psi$ as the connected centraliser
 \ben{$C_\mathbf{G}(g_\Psi)^\circ$} of a semi-simple element $g_\Psi \in
 \mathbf{T}(\F_q^\alg)$ such that the root system associated to
 $\mathbf{H}_\Psi$ is~$\Psi$.

 Clearly, (1) and (2) are satisfied.  It remains to justify~(3).  Fix
 $\Psi \in \mathcal{C}(\mathbf{G},\mathbf{T},\F_q^\alg)$.  We may
 assume that $\mathbf{H} = \mathbf{H}_\Psi$ is itself equal to
 $C_\mathbf{G}(g_0)^\circ$ for a semi-simple element $g_0 \in
 \mathbf{G}(\F_q)$.  Let $\mathbf{T}_0$ be an $\F_q$-defined maximal
 torus of $\mathbf{H}$.  We observe that $g_0 \in
 Z(\mathbf{H}^\circ)\subset \mathbf{T}_0$
 (see~\cite[Proposition~3.5.1]{Ca}) and that $\mathbf{T}_0$ is also an
 $\F_q$-defined maximal torus of $\mathbf{G}$.
 By~\cite[Proposition~3.3.3]{Ca}, the number of
 $\mathbf{G}(\F_q)$-conjugacy classes of $\F_q$-defined maximal tori
 is at most $\lvert W(\Phi) \rvert$.  Thus we only need to bound the
 number of connected centralisers of semi-simple elements $g$ in any
 fixed $\F_q$-defined maximal torus, but this number is clearly
 bounded by~$\lvert \mathcal{C}(\Phi) \rvert$.
\end{proof}

\subsubsection{Non-Central Semi-Simple Conjugacy
  Classes} \label{subsubsec:nc}
Let $\mathbf{G}$ be a connected reductive algebraic group defined over
a finite field~$\F_q$ and let $\mathbf{G}^*$ be the dual group.  By
\cite[Proposition~3.6.8]{Ca}, we have $Z(\mathbf{G}^*(\F_q)) =
Z(\mathbf{G}^*)(\F_q)$ so that central geometric conjugacy classes and
central rational conjugacy classes in~$\mathbf{G}^*(\F_q)$ coincide.

We set
\begin{equation} \label{def:zeta_nc}
  \zeta^{\mathrm{nc}}_{\mathbf{G}(\F_q)}(s) = \sum_{(g)
    \not\subset Z(\mathbf{G}^*(\F_q))} \; \sum_{\chi \in
    \mathcal{E}(\mathbf{G}(\F_q),(g))} \chi(1)^{-s},
\end{equation}
where the outer sum ranges over the non-central semi-simple
$\mathbf{G}^*$-conjugacy classes $(g)$
in~$\mathbf{G}^*(\F_q)$. For a non-trivial root system $\Phi$,
we define
\begin{equation} \label{def:nc_Phi}
  b^\mathrm{nc}(\Phi) = \left\{ ( \ben{\rk(\Phi) - \rk(\Psi)}, \lvert
    \Phi^+ \rvert - \lvert \Psi^+ \rvert)
    \mid \Psi \in \mathcal{C}(\Phi) \smallsetminus \{\Phi\}\right\} \in
  \mathcal{A}^+ \ben{.}
\end{equation}
% \begin{equation} \label{def:nc_Phi}
%   b^\mathrm{nc}(\Phi) = \left\{ (\lvert \comp(\Psi) \rvert -1, \lvert
%     \Phi^+ \rvert - \lvert \Psi^+ \rvert)
%     \mid \Psi \in \mathcal{C}(\Phi) \smallsetminus \{\Phi\}\right\} \in
%   \mathcal{A}^+,
% \end{equation}
% where $\lvert \comp(\Psi) \rvert$ denotes the number of irreducible
% components of $\Psi$.

\begin{lem}\label{lem:nc}
  Let $\Phi$ be a non-trivial root system and let
  $\boldsymbol{\mathcal{G}}_\Phi$ denote the collection of pairs
  $(\mathbf{G},q)$, where $\mathbf{G}$ is a connected semi-simple
  algebraic group over a finite field $\F_q$ with absolute root
  system~$\Phi$.  Then there are $C \in \R$ and
  $b^\mathrm{nc}(\mathbf{G},q)\subset b^\mathrm{nc}(\Phi)$ for
  $(\mathbf{G},q)\in\boldsymbol{\mathcal{G}}_\Phi$ such that for all
  $(\mathbf{G},q),(\mathbf{H},q)\in\boldsymbol{\mathcal{G}}_\Phi$ the
  following hold:
  \begin{list}{}{\setlength{\leftmargin}{\myenumilistleftmargin}
      \setlength{\labelwidth}{20pt} \setlength{\itemsep}{0pt}
      \setlength{\parsep}{1pt}}
  \item[\textup{(1)}] $b^\mathrm{nc}(\mathbf{G},q) \subset
    b^\mathrm{nc}(\mathbf{H},q)$ whenever there is an isogeny
    $\phi\colon \mathbf{G}\rightarrow\mathbf{H}$ defined
    over~$\F_q$,
  \item[\textup{(2)}] $b^\mathrm{nc}(\mathbf{G},q) =
    b^\mathrm{nc}(\Phi)$ whenever $\mathbf{G}$ is split and $q
    \equiv_m 1$, where $m$ is as in
    Proposition~\textup{\ref{pro:centralizers.split}},
  \item[\textup{(3)}] $\zeta^{\mathrm{nc}}_{\mathbf{G}(\F_q)}
    \sim_C \xi_{b^\mathrm{nc}(\mathbf{G},q),q}$,
  \item[\textup{(4)}] $(\rk \Phi, \lvert \Phi^+\rvert) \in
    b^\mathrm{nc}(\Phi)$, and $1 +
    \xi_{b^\mathrm{nc}(\mathbf{G},q),q}(s) \sim_C 1 + q^{\rk \Phi -
      \lvert\Phi^+\rvert s}$.
  \end{list}
\end{lem}

\begin{proof}
  Consider $(\mathbf{G},q)\in\boldsymbol{\mathcal{G}}_\Phi$ and let
  $\mathbf{G}^*$ be the dual group of~$\mathbf{G}$. The centralizer
  $C_{\mathbf{G} ^*}(g)$ of a semi-simple element $g \in
  \mathbf{G}^*(\F_q)$ is a reductive subgroup of maximal rank
  in~$\mathbf{G}^*$.  Furthermore, there is a surjection $\psi_g
  \colon \mathcal{E}(\mathbf{G}(\F_q),(g)) \rightarrow
  \mathcal{E}(C_{\mathbf{G}^*}(g)(\F_q),(1))$ such that the fibers of
  $\psi_g$ have sizes at most $\lvert Z(\mathbf{G}) \rvert$, and the
  degree of an element $\chi \in \mathcal{E}(\mathbf{G}(\F_q),(g))$ is
  given by
  \[
  \chi(1) = \vert \mathbf{G}^* (\F_q) : C_{\mathbf{G}^*}(g)
  (\F_q) \vert_{q'} \cdot (\psi_g(\chi))(1),
  \]
  where $n_{q'}$ denotes the prime-to-$q$ part of a number $n$;
  see~\cite[Theorem~13.23 and Remark~13.24]{DM} and
  \cite[Proposition~4.4.4]{Ca} in the case where $\mathbf{G}$ has
  trivial centre, and~\cite[Proposition 5.1]{Lu} for the general
  case. It follows that
  \[
  \zeta^{\mathrm{nc}}_{\mathbf{G}(\F_q)}(s) \sim_{\lvert
    Z(\mathbf{G}) \rvert} \sum_{(g) \not \subset
    Z(\mathbf{G}^*(\F_q))} \vert \mathbf{G}^*(\F_q) :
  C_{\mathbf{G}^*}(g) (\F_q) \vert_{q'}^{-s} \cdot
  \zeta^\mathrm{unip}_{C_{\mathbf{G}^*}(g)(\F_q)}(s).
  \]

  Consider a semi-simple element $g \in \mathbf{G}^*(\F_q)$.
  The index $\lvert C_{\mathbf{G}^*}(g) : C_{\mathbf{G}^*}(g)^\circ
  \rvert$ is bounded by a constant $C_1 \in \R$, depending only
  on $\Phi$; see~\cite[Remark~2.4]{DM}.  This gives
  \[
  \lvert \mathbf{G}^*(\F_q) : C_{\mathbf{G}^*}(g)
  (\F_q) \rvert_{q'}^{-s} \sim_{C_1} \lvert
  \mathbf{G}^*(\F_q) : C_{\mathbf{G}^*}(g)^\circ
  (\F_q) \rvert_{q'}^{-s}.
  \]
  Moreover, the unipotent characters of $C_{\mathbf{G}^*}(g)(\F_q)$
  are the irreducible characters whose restrictions to
  $C_{\mathbf{G}^*}(g)^\circ(\F_q)$ are sums of unipotent characters.
  From Proposition~\ref{lem:BC.unipotent.representations} we conclude
  that there is a constant $C_2 \in \R$ such that
  \[
  \zeta^\mathrm{unip}_{C_{\mathbf{G}^*}(g)(\F_q)}(s)
  \sim_{C_2} 1.
  \]
  We get that
  \[
  \zeta^{\mathrm{nc}}_{\mathbf{G}(\F_q)}\sim_{\lvert
    Z(\mathbf{G})\rvert C_1C_2} \sum_{(g) \not\subset
    Z(\mathbf{G}^*(\F_q))} \left( \frac{\lvert \mathbf{G}
      ^*(\F_q)\rvert_{q'}}{\lvert C_{\mathbf{G}
        ^*}(g)^{\circ}(\F_q)\rvert_{q'}} \right) ^{-s}.
  \]

  Applying Corollary~\ref{cor:centralizers.uniform} to $\mathbf{G}^*$,
  we obtain $N \in \N$, bounded by $\lvert \mathcal{C}(\Phi) \rvert$,
  and algebraic subgroups $\mathbf{H}_1,\ldots,\mathbf{H}_N \subset
  \mathbf{G}^*$ with the properties described in the corollary.  For
  each $i \in \{1,\ldots,N\}$ denote the absolute root system of
  $\mathbf{H}_i$ by~$\Psi_i\in\mathcal{C}(\Phi)$. Given a non-central
  semi-simple element $g\in \mathbf{G}^*(\F_q)$, the group
  $C_{\mathbf{G}^*}(g)^\circ$ is $\mathbf{G}^*$-conjugate to some
  $\mathbf{H}_i$.  There is a constant $C_3 \in \R$ such that $\lvert
  \mathbf{G}^*(\F_q) \rvert_{q'} \sim_{C_3} q^{\dim \mathbf{G}
    - \lvert \Phi ^+\rvert}$ and $\lvert C_{\mathbf{G}^*}(g)^\circ
  (\F_q) \rvert_{q'} \sim_{C_3} q^{\dim \mathbf{H}_i-\lvert
    \Psi_i^+ \rvert}$. Therefore
  \[
  \frac{\lvert \mathbf{G}^*(\F_q) \rvert_{q'}}{\lvert
    C_{\mathbf{G}^*}(g)^{\circ}(\F_q) \rvert_{q'}} \sim_{C_3^2}
  q^{\dim \mathbf{G} - \lvert \Phi^+ \rvert - \dim \mathbf{H}_i +
    \lvert \Psi_i ^+ \rvert} = q^{\lvert \Phi^+ \rvert - \lvert
    \Psi_i^+ \rvert}
  \]
  for all $i \in \{1,\ldots,N\}$.  Writing $C_4 = \lvert Z(\mathbf{G})
  \rvert C_1 C_2 C_3^2N$, we obtain
  \begin{multline} \label{equ:C_4-inequality}
    \zeta^{\mathrm{nc}}_{\mathbf{G}(\F_q)}(s) \sim_{C_4} \\
    \sum_{i=1}^N \lvert \{ (g) \not \subset Z(\mathbf{G}^*(\F_q)) \mid
    \text{$C_{\mathbf{G}^*}(g)^\circ$ is $\mathbf{G}^*$-conjugate to
      $\mathbf{H}_i$} \} \rvert \cdot q^{-(\lvert \Phi^+ \rvert -
      \lvert \Psi_i^+ \rvert ) s}.
  \end{multline}
  The additional factor $N$ of $C_4$ is included, because
  $C_{\mathbf{G}^*}(g)^\circ$ could be conjugate to more than one
  $\mathbf{H}_i$.

  Denote by $I(\mathbf{G},q)$ the set of all $i \in \{ 1, \ldots, N
  \}$ such that there exists a non-central semi-simple element $g \in
  \mathbf{G}^*(\F_q)$ with $C_{\mathbf{G} ^*}(g)^\circ$ being
  $\mathbf{G}^*$-conjugate to $\mathbf{H}_i$.  Fix $i \in
  I(\mathbf{G},q)$ and let $\mathbf{K}_{i,1}, \ldots,
  \mathbf{K}_{i,n(i)}$ be a selection of such centralisers, forming a
  complete set of representatives up to
  $\mathbf{G}^*(\F_q)$-conjugacy.  By
  Corollary~\ref{cor:centralizers.uniform}, the number $n(i)$ is
  uniformly bounded in terms of~$\Phi$.  The collection of non-central
  semi-simple conjugacy classes $(g) \subset \mathbf{G}
  ^*(\F_q)$ such that $C_{\mathbf{G}^*}(g)^\circ$ is
  $\mathbf{G}^*$-conjugate to $\mathbf{H}_i$ decomposes as a disjoint
  union as follows:
  \begin{multline} \label{equ:H_i.disj.union} \left\{ (g) \not \subset
      Z(\mathbf{G}^*(\F_q)) \mid \text{$C_{\mathbf{G}^*}(g)^\circ$ is
        $\mathbf{G}^*$-conjugate to $\mathbf{H}_i$} \right\} = \\
    \bigsqcup_{j=1}^{n(i)} \left\{ (g) \not \subset
      Z(\mathbf{G}^*(\F_q)) \mid \text{$C_{\mathbf{G}^*}(g)^\circ$ is
        $\mathbf{G}^*(\F_q)$-conjugate to $\mathbf{K}_{i,j}$}
    \right\}.
  \end{multline}

  Fix also $j \in \{1,\ldots,n(i)\}$. Now \cite[Lemma~2.2(ii)]{LS}
  supplies a constant $C_5 \in \R$ such that, for every non-central
  semi-simple conjugacy class $(g) \subset \mathbf{G}^*(\F_q)$
  with $C_{\mathbf{G}^*}(g)^\circ$ being
  $\mathbf{G}^*(\F_q)$-conjugate to $\mathbf{K}_{i,j}$,
  \[
  1 \leq \vert \{ g^x \in (g) \mid C_{\mathbf{G}^*}(g^x)^\circ
  =\mathbf{K}_{i,j} \} \vert \leq C_5.
  \]
  It follows that
  \begin{multline}\label{equ:Kij1}
    \left| \left\{ (g) \not \subset Z(\mathbf{G}^*(\F_q)) \mid
        \text{$C_{\mathbf{G} ^*}(g)^\circ$ is $\mathbf{G}
          ^*(\F_q)$-conjugate to $\mathbf{K}_{i,j}$} \right\}
    \right| \sim_{C_5} \\ \left| \left\{ g \in \mathbf{G}
        ^*(\F_q) \mid C_{\mathbf{G} ^*}(g)^\circ =
        \mathbf{K}_{i,j} \right\} \right|.
  \end{multline}
  If $g\in \mathbf{G}^*(\F_q)$ is semi-simple and
  $C_{\mathbf{G}^*}(g)^\circ =\mathbf{K}_{i,j}$, then $g \in
  \mathbf{K}_{i,j}(\F_q)$. \nir{Choose a maximal torus $\mathbf{T}_{i,j}$ of $\mathbf{K_{i,j}}$. The torus $\mathbf{T}_{i,j}$ is also
  maximal in $\mathbf{G} ^*$.}  Denote the set of roots of $(\mathbf{G}
  ^*,\nir{\mathbf{T}_{i,j}})$ by $\Lambda_{i,j} \subset
  \Hom(\nir{\mathbf{T}_{i,j}},\mathbb{G}_\text{m})$ and the set of roots
  of $(\mathbf{K}_{i,j},\nir{\mathbf{T}_{i,j}})$ by $\Delta_{i,j}
  \subset \Hom(\nir{\mathbf{T}_{i,j}},\mathbb{G}_\text{m})$. Note that
  $\Lambda_{i,j}$ is isomorphic to $\Phi$ and $\Delta_{i,j}$ is
  isomorphic to $\Psi_i$.  The set of elements $g \in
  Z(\mathbf{K}_{i,j})^\circ$ with $C_{\mathbf{G} ^*}(g)^\circ
  =\mathbf{K}_{i,j}$ is the complement of the union of the zero loci
  of the roots in~$\Lambda_{i,j}\smallsetminus \Delta_{i,j}$.  Each
  such zero locus is the extension of a proper sub-torus by a finite
  group.  The order of that finite group, i.e., the number of
  connected components of the zero locus, is bounded by some constant
  depending only on~$\Phi$; see, for example,
  \cite[Corollary~9.7.9]{EGA4}.  For every torus $\mathbf{T}$, we have
  \begin{equation}\label{equ:Nori}
    2^{-\dim \mathbf{T}}q^{\dim \mathbf{T}} \leq |
    \mathbf{T}(\F_q) | \leq 2^{\dim \mathbf{T}}q^{\dim
      \mathbf{T}};
  \end{equation}
  see~\cite[Lemma~3.5]{N}.  Hence there is a constant $D \in \R$ such
  that, for all sufficiently large~$q$,
  \begin{equation}\label{equ:Kij2}
    q^{\dim Z(\mathbf{K}_{i,j})} \lesssim_D | \left\{ g \in
      \mathbf{G}^*(\F_q) \mid C_{\mathbf{G}
        ^*}(g)^\circ=\mathbf{K}_{i,j} \right\} |.
  \end{equation}
  By \cite[Corollary~9.7.9]{EGA4}, the index $\lvert
  Z(\mathbf{K}_{i,j}) : Z(\mathbf{K}_{i,j})^\circ \rvert$ for the
  connected reductive group $\mathbf{K}_{i,j}$ is bounded by some
  constant depending on~$\Phi$ and, of course, $\dim
  Z(\mathbf{K}_{ij}) = \dim Z(\mathbf{H}_i)$.
  Using~\eqref{equ:H_i.disj.union}, \eqref{equ:Kij1},
  \eqref{equ:Nori}, and \eqref{equ:Kij2}, it follows that there is a
  constant $C_6 \in \R$ such that
  \[
  \lvert \{ (g) \not \subset Z(\mathbf{G}^*(\F_q)) \mid
  \text{$C_{\mathbf{G} ^*}(g)^\circ$ is $\mathbf{G} ^*$-conjugate to
    $\mathbf{H}_i$} \} \rvert \sim_{C_6} q^{\dim Z(\mathbf{H}_i)},
  \]
  and from \eqref{equ:C_4-inequality} we obtain
  \[
  \zeta^{\mathrm{nc}}_{\mathbf{G}(\F_q)}(s) \sim_{C_4 C_6} \sum_{i \in
    I(\mathbf{G},q)} q^{\dim Z(\mathbf{H}_i) - (\lvert \Phi^+ \rvert -
    \lvert \Psi_i^+ \rvert)s}.
  \]
  Recalling from~\eqref{def:nc_Phi} the definition of
  $b^\mathrm{nc}(\Phi)$, we put
  \[
  b^\mathrm{nc}(\mathbf{G},q) = \left\{ (\dim Z(\mathbf{H}_i), \lvert
    \Phi^+ \rvert - \lvert \Psi_i^+ \rvert) \mid i \in I(\mathbf{G},q)
  \right\} \subset b^\mathrm{nc}(\Phi)
  \]
  so that
  \[
  \zeta^{\mathrm{nc}}_{\mathbf{G}(\F_q)} \sim_{C_4 C_6 N}
  \xi_{b^\mathrm{nc}(\mathbf{G},q),q}.
  \]
  The extra factor $N$ in the index of the $\sim$-symbol accommodates
  for the fact that $\mathbf{H}_i$ and $\mathbf{H}_j$ may lead to the
  same data even though $i \neq j$. This completes the proof of
  assertion~(3) of the lemma.

  Next we prove~(4).  For $i \in I(\mathbf{G},q)$,
  \cite[Lemma~2.5]{LS} yields that
  \begin{equation} \label{equ:dim.formula} \frac{\dim
      Z(\mathbf{H}_i)}{\lvert \Phi^+ \rvert - \lvert \Psi_i^+ \rvert}
    \leq \frac{\rk \mathbf{G} }{\lvert \Phi^+ \rvert} = \frac{\rk
      \Phi}{\lvert \Phi^+ \rvert}.
  \end{equation}
  Moreover, we have $\dim Z(\mathbf{H}_i) \leq \rk \mathbf{G}$ for
  each $i \in I(\mathbf{G},q)$, and for sufficiently large~$q$
  (depending only on $\Phi$; see~\cite[Proposition~3.6.6]{Ca}) there
  is at least one $\mathbf{H}_i$, $i \in I(\mathbf{G},q)$, namely a
  non-degenerate maximal torus, for which
  \begin{equation}\label{equ:newton}
  (\dim Z(\mathbf{H}_i), \lvert \Phi^+ \rvert - \lvert \Psi_i^+
  \rvert) = (\rk \Phi, |\Phi^+|).
  \end{equation}
  From this we deduce that, for sufficiently large $q$,
  \[
  1 + \xi_{b^\mathrm{nc}(\mathbf{G},q),q}(s) \sim_{N+1}
  1+\xi_{\left\{ (\rk \Phi, | \Phi^+|) \right\},q}(s) = 1 + q^{\rk
    \Phi - |\Phi^+| s};
  \]
  indeed, the inequality $\gtrsim_{N+1}$ is clear
  from~\eqref{equ:newton}, and $\lesssim_{N+1}$ follows from
  Remark~\ref{rem:xi}, noting that $\lvert
  b^{\mathrm{nc}}(\mathbf{G},q) \rvert \leq N$.
  %indeed, for $s = \sigma \in \R$ with $0 < \sigma < \rk \Phi /
  %\lvert \Phi ^+ \rvert$ we obtain the inequality using the facts
  %collected above, and for $s = \sigma \in \R$ with $\sigma \geq \rk
  %\Phi / \lvert \Phi ^+ \rvert$ it suffices to observe that the left
  %hand side is bounded above by $N+1$.
  Thus (4) is proved.

  Assertion~(1) follows from the observation that, if $\phi \colon
  \mathbf{G} \rightarrow \mathbf{H}$ is an isogeny over $\F_q$, then by
  Lemma~\ref{lem:centr-isogeny} we may arrange the labelling so that
  $I(\mathbf{G},q) \subset I(\mathbf{H},q)$.  Assertion~(2) holds by
  virtue of Proposition~\ref{pro:centralizers.split}.
\end{proof}

\subsubsection{Proofs of Theorem~\ref{thm:BC.finite.new} and
  Corollary~\ref{cor:finite.reductive.zeta.approx}}\label{subsubsec:BC.finite.new}

\begin{proof}[Proof of Theorem~\ref{thm:BC.finite.new}]
  Let $\mathbf{G}$ be a connected semi-simple algebraic group over
  $\F_q$ of Lie type $L \in \boldsymbol{\mathcal{L}}_\Phi$ and let
  $\mathbf{G}^*$ be the dual group.  In analysing the right-hand side
  of~\eqref{equ:Lusztig-zeta}, we deal separately with the sum over
  central conjugacy classes and the sum over non-central conjugacy
  classes. In analogy with the definition~\eqref{def:zeta_nc} of
  $\zeta^\mathrm{nc}_{\mathbf{G}(\F_q)}(s)$ we set
 \begin{equation*}
    \zeta^\mathrm{c}_{\mathbf{G}(\F_q)}(s) = \sum_{(g) \subset
      Z(\mathbf{G}^*(\F_q))} \; \sum_{\chi \in
      \mathcal{E}(\mathbf{G}(\F_q),(g))} \chi(1)^{-s},
  \end{equation*}
  where the outer sum ranges over the central
  $\mathbf{G}^*$\nobreakdash-conjugacy classes $(g)$
  in~$\mathbf{G}^*(\F_q)$ corresponding to semi-simple
  elements $g \in Z(\mathbf{G}^*(\F_q))$. Thus
  $\zeta_{\mathbf{G}(\F_q)} =
  \zeta_{\mathbf{G}(\F_q)}^\mathrm{c} +
  \zeta_{\mathbf{G}(\F_q)}^\mathrm{nc}$;
  cf.~\eqref{equ:Lusztig-zeta}.

We first consider $\zeta_{\mathbf{G}(\F_q)}^\mathrm{c}$.
There is a constant $C_0 \in \R$, depending only on~$\Phi$,
such that
  \begin{equation} \label{equ:centre-derived}
  \vert Z(\mathbf{G} ^*(\F_q)) \vert \leq C_0.
  \end{equation}
  Furthermore, if $g\in Z(\mathbf{G} ^*(\F_q))$, then the elements of
  $\mathcal{E}(\mathbf{G}(\F_q),(g))$ are in dimension-preserving
  bijection with the elements of $\mathcal{E}(\mathbf{G}(\F_q),(1))$;
  see~\cite[Theorem~5.1]{Lu}.  Thus,
  \begin{equation}\label{equ:sigma_c}
    \zeta^\mathrm{c}_{\mathbf{G}(\F_q)} = \vert
    Z(\mathbf{G}^*(\F_q)) \vert \,
    \zeta ^\mathrm{unip}_{\mathbf{G} (\F_q)},
  \end{equation}
  and, in combination with~\eqref{equ:centre-derived},
  Proposition~\ref{lem:BC.unipotent.representations} yields $C_1 \in
  \R$ and $b^{\textup{c}}(\mathcal{L}) \in \mathcal{A}^+$, depending
  only on~$\Phi$ respectively $L$, such that
  \begin{equation} \label{eq:unip.characters}
    \zeta^\mathrm{c}_{\mathbf{G}(\F_q)} - \vert Z(\mathbf{G}
    ^*(\F_q)) \vert = \vert Z(\mathbf{G} ^*(\F_q))
    \vert \, (\zeta ^\mathrm{unip}_{\mathbf{G} (\F_q)} - 1)
    \sim_{C_0 C_1} \xi_{b^\textup{c}(\mathcal{L}),q}.
  \end{equation}
  We set $b^\mathrm{c}(\Phi) = b^\textup{c}(\mathcal{L}^\mathrm{sp})$.

  Next we account for~$\zeta^\mathrm{nc}_{\mathbf{G}(\F_q)}$. We
  recall the definition~\eqref{def:nc_Phi} of the set
  $b^{\mathrm{nc}}(\Phi)$ and set
  \begin{equation}\label{def:a(Phi)}
   a(\Phi) = b^\mathrm{c}(\Phi) + b^\mathrm{nc}(\Phi), \qquad
   a(\mathbf{G},q) = b^\mathrm{c}(\mathcal{L}) + b^\mathrm{nc}(\mathbf{G},q),
  \end{equation}
  where $ b^\mathrm{nc}(\mathbf{G},q)$ is defined as in
  Lemma~\ref{lem:nc}. Furthermore, we choose $m$, depending only
  on~$\Phi$, in accordance with
  Proposition~\ref{pro:centralizers.split}.  Then the following weaker
  variants of assertions~(1) and (2) of the theorem are clearly
  satisfied:
  \begin{list}{$\circ$}{\setlength{\leftmargin}{\myenumilistleftmargin}
      \setlength{\itemsep}{0pt} \setlength{\parsep}{1pt}}
  \item[$\textup{(1)}'$] $a(\mathbf{G},q) \subset a(\Phi)$,
  \item[$\textup{(2)}'$] $a(\mathbf{G},q) = a(\Phi)$ whenever
    $\mathbf{G}$ is split and $q \equiv_m 1$.
  \end{list}
  
  Moreover, setting $C_2$ to be the sum of the constant $C_0 C_1$
  from~\eqref{eq:unip.characters} and the constant that
  Lemma~\ref{lem:nc} supplies, we obtain
  \begin{equation} \label{eq:zeta.approx.Z}
    \zeta_{\mathbf{G}(\F_q)}(s) - \lvert Z(\mathbf{G}
      ^*(\F_q)) \rvert \sim_{C_2}
    \xi_{a(\mathbf{G},q),q}(s).
  \end{equation}
  We now prove that, for $q > C_2$,
  \[
  \lvert \mathbf{G} (\F_q) /
  [\mathbf{G}(\F_q),\mathbf{G}(\F_q)] \rvert =
  \lvert Z(\mathbf{G}^*(\F_q)) \rvert.
  \]
  Indeed, put $\delta(q) = \lvert \mathbf{G} (\F_q) /
  [\mathbf{G}(\F_q),\mathbf{G}(\F_q)] \rvert - \lvert
  Z(\mathbf{G}^*(\F_q)) \rvert$.  Fix $q > C_2$ and let $s =
  \sigma \in \R$ tend to $\infty$ in~\eqref{eq:zeta.approx.Z}.  The
  limit of the left-hand side of \eqref{eq:zeta.approx.Z} as $s =
  \sigma \to \infty$ is equal to $\delta(q)$, and $q>C_2$ implies that
  the limit of $C_2^{1+\sigma} \xi_{a(\mathbf{G},q)}(\sigma)$ as
  $\sigma \to \infty$ is equal to~$0$.  Thus $\delta(q) \leq 0$.  On
  the other hand, the limit of the right-hand side
  of~\eqref{eq:zeta.approx.Z} as $s = \sigma \to \infty$ is equal to
  $0$, and the limit of $C_2^{1+\sigma}
  (\zeta_{\mathbf{G}(\F_q)}(\sigma) - | Z(\mathbf{G}
  ^*(\F_q))| )$ as $\sigma \to \infty$ is non-negative only if
  $\delta(q) \geq 0$.  Hence $\delta(q) = 0$.

  Next we show that in the approximations that we seek it is not
  necessary to distinguish between different members of the isogeny
  class of~$\mathbf{G}$, as we have done up to this point.  We define
  \[
  a(\mathcal{L},q) = a(\mathbf{G}^\textup{ad},q) = b^\mathrm{c}(\mathcal{L}) +
  b^\mathrm{nc}(\mathbf{G}^\textup{ad},q),
  \]
  where $\mathbf{G}^\mathrm{ad}$ denotes the adjoint quotient
  of~$\mathbf{G}$ and claim that assertions (1), (2), (3) in the
  theorem hold.  The statements (1), (2) are clearly special cases of
  $(1)'$, $(2)'$.  In order to establish (3) we observe that
  Lemma~\ref{lem:nc} already gives
  \[
  \xi_{a(\mathbf{G}^\mathrm{sc},q),q} \lesssim_1 \xi_{a(\mathbf{G},q),q}
  \lesssim_1   \xi_{a(\mathbf{G}^\mathrm{ad},q),q} = \xi_{a(\mathcal{L},q),q}.
  \]
  where $\mathbf{G}^\mathrm{sc}$ denotes the simply connected member
  in the isogeny class of $\mathbf{G}$.

  Thus it suffices to show that there is a constant $C_3 \in \R$,
  depending only on $\Phi$, such that
  \[
  \zeta_{\mathbf{G}^\mathrm{sc}(\F_q)} - \lvert
  \mathbf{G}^\mathrm{sc}(\F_q) /
  [\mathbf{G}^\mathrm{sc}(\F_q),\mathbf{G}^\mathrm{sc}(\F_q)]
  \rvert \gtrsim_{C_3}
  \zeta_{\mathbf{G}^\mathrm{ad}(\F_q)} - \lvert
  \mathbf{G}^\mathrm{ad}(\F_q) /
  [\mathbf{G}^\mathrm{ad}(\F_q),\mathbf{G}^\mathrm{ad}(\F_q)]
  \rvert.
  \]
  Write $G = \mathbf{G}^\mathrm{ad}(\F_q)$.  Let $H$ denote the image
  of $\mathbf{G}^\mathrm{sc}(\F_q)$ under the natural homomorphism
  $\mathbf{G}^\mathrm{sc}(\F_q) \rightarrow
  \mathbf{G}^\mathrm{ad}(\F_q)$.  Then $\lvert G \rvert = \lvert
  \mathbf{G}^\mathrm{sc}(\F_q) \rvert$, thus $\lvert G : H \rvert =
  \lvert Z(\mathbf{G}^\mathrm{sc}(\F_q)) \rvert$ is bounded by $C_3 =
  \lvert Z(\mathbf{G}^\mathrm{sc}) \rvert$, depending only on $\Phi$.
  From \eqref{eq:zeta.approx.Z} we observe that the smallest degree of
  a non-linear character of $G$ is given by an increasing function
  in~$q$.  Taking as a normal subgroup the trivial subgroup, we apply
  part (3) of Lemma~\ref{lem:fin.index.BK} to deduce that
  \[
  \zeta_{\mathbf{G}^\mathrm{sc}(\F_q)} - \lvert
  \mathbf{G}^\mathrm{sc}(\F_q) /
  [\mathbf{G}^\mathrm{sc}(\F_q),\mathbf{G}^\mathrm{sc}(\F_q)]
  \rvert \gtrsim_1
  \zeta_H - \lvert H/[H,H] \rvert \gtrsim_{C_3}
  \zeta_G - \lvert G/[G,G] \rvert.
  \]

  Finally, it remains to deduce the
  estimate~\eqref{eq:thm3.2-generous-estimate}.  We need to show that
  for some constant $C_4 \in \R$, depending only on $\Phi$,
  \[
  \zeta_{\mathbf{G}(\F_q)}(s) \sim_{C_4} 1 + q^{\rk \Phi -
    |\Phi^+| s}.
  \]
  By part (4) of Lemma~\ref{lem:nc}, there is a constant $C_5 \in \R$
  such that
  \begin{equation}\label{equ:1+xi-estimate}
  1 + \xi_{b^\mathrm{nc}(\mathbf{G},q),q}(s) \sim_{C_5} 1 + q^{\rk
    \Phi - |\Phi^+| s}.
  \end{equation}
  Furthermore, using equation~\eqref{equ:sigma_c}, part (2) of
  Proposition~\ref{lem:BC.unipotent.representations}, part (3) of
  Lemma~\ref{lem:nc} and equations~\eqref{equ:centre-derived} and
  \eqref{equ:1+xi-estimate}, we conclude that there is a constant $C_6
  \in \R$, depending only on $\Phi$, such that
  \begin{align*}
    \zeta_{\mathbf{G}(\F_q)}(s) & =
    \zeta_{\mathbf{G}(\F_q)}^\mathrm{c}(s) +
    \zeta_{\mathbf{G}(\F_q)}^\mathrm{nc}(s) = \lvert
    Z(\mathbf{G}^*(\F_q)) \rvert
    \zeta_{\mathbf{G}(\F_q)}^\mathrm{unip}(s) +
    \zeta_{\mathbf{G}(\F_q)}^\mathrm{nc}(s) \\ & \sim_{C_6}
    \lvert Z(\mathbf{G} ^*(\F_q)) \rvert -1 + 1 +
    \xi_{b^\mathrm{nc}(\mathbf{G},q),q}(s) \sim_{C_0 + C_5} 1 + q^{\rk
      \Phi - |\Phi^+| s},
  \end{align*}
  proving the claim. This concludes the proof of
  Theorem~\ref{thm:BC.finite.new}.
\end{proof}

\begin{proof}[Proof of
  Corollary~\ref{cor:finite.reductive.zeta.approx}] We use the
  inequality
  \[
  2^{-\dim \mathbf{H}} q^{\dim \mathbf{H}} \leq \lvert
  \mathbf{H}(\F_q) \rvert \leq 2^{\dim \mathbf{H}} q^{\dim
    \mathbf{H}}
  \]
  true for every connected algebraic group $\mathbf{H}$ over a finite
  field $\F_q$; cf.~\cite[Lemma~3.5]{N}.  Let $C_1$ be the constant
  from Theorem~\ref{thm:BC.finite.new}. The derived subgroup
  $\mathbf{G} '$ is connected, semi-simple, and has absolute root
  system $\Phi$. The quotient $\mathbf{G} / \mathbf{G} '$ is a torus
  of dimension~$\dim Z(\mathbf{G})$. By the Lang--Steinberg
  theorem~\cite[Theorem~3.10]{DM}, there is a short exact sequence $1
  \rightarrow \mathbf{G} '(\F_q) \rightarrow \mathbf{G}(\F_q)
  \rightarrow \mathbf{G} / \mathbf{G} '(\F_q) \rightarrow 1$, which
  implies that
  \[
  \zeta_{\mathbf{G}(\F_q)}(s) \lesssim_1 | \mathbf{G} /
  \mathbf{G} '(\F_q)| \zeta_{\mathbf{G} '(\F_q)}(s)
  \lesssim_{2^{\dim Z(\mathbf{G})}} q^{\dim
    Z(\mathbf{G})}\zeta_{\mathbf{G}'(\F_q)} \lesssim_{C_1}
  q^{\dim Z(\mathbf{G})} (1 + q^{\rk \Phi - |\Phi^+| s}).
  \]

  For the opposite inequality, choose a maximal $\F_q$-rational torus
  $\mathbf{T} \subset \mathbf{G} / Z(\mathbf{G})$ and a Borel subgroup
  $\mathbf{B} \subset \mathbf{G} / Z(\mathbf{G})$ containing
  $\mathbf{T}$. Let $\widetilde{\mathbf{T}},\widetilde{\mathbf{B}}
  \subset \mathbf{G}$ be the pre-images of $\mathbf{T},\mathbf{B}$
  under the quotient map. By the proof of~\cite[Lemma~8.4.2]{Ca}, there
  is~$C_2 \in \R$, depending only on $\Phi$, such that the set of
  characters $\theta$ of $\widetilde{\mathbf{T}}(\F_q)$ for which the
  parabolic induction
  $\Ind_{\widetilde{\mathbf{B}}(\F_q)}^{\mathbf{G}(\F_q)}(\theta)$ is
  irreducible has size at least
  $C_2|\widetilde{\mathbf{T}}(\F_q)|$. Thus, we get at least $C_2
  \cdot 2^{-\dim \widetilde{\mathbf{T}}}q^{\dim
    \widetilde{\mathbf{T}}}=C_2 \cdot 2^{-\dim
    \widetilde{\mathbf{T}}}q^{\dim Z(\mathbf{G})+\rk \Phi}$
  irreducible representations of dimension
  \[
  \lvert \mathbf{G}(\F_q) \rvert / \lvert
  \widetilde{\mathbf{B}}(\F_q) \rvert \leq 2^{\dim \mathbf{G}
    +\dim \widetilde{\mathbf{B}}} q^{\dim \mathbf{G} - \dim
    \widetilde{\mathbf{B}}}=2^{\dim \mathbf{G} +\dim
    \widetilde{\mathbf{B}}} q^{\lvert \Phi^+ \rvert}.
  \]
  Finally, $\mathbf{G}(\F_q)$ has $| \mathbf{G} / \mathbf{G} '
  (\F_q) | \geq 2^{-\dim Z( \mathbf{G})}q^{\dim
    Z(\mathbf{G})}$ one-dimensional representations factoring through
  $\mathbf{G} / \mathbf{G} '(\F_q)$. Combining these two
  classes of representations, we get that
  $\zeta_{\mathbf{G}(\F_q)}(s) \gtrsim_{C_3} q^{\dim
    Z(\mathbf{G})}(1 + q^{\rk \Phi - |\Phi^+| s}) $ for a suitable
  $C_3 \in \R$.
\end{proof}

\subsection{Applications to Finite Quotients of Arithmetic
  Groups} \label{sec:applications}

In order to apply Theorem~\ref{thm:BC.finite.new} in a global context,
we recall the definition of a Chebotarev set.

\begin{defn} \label{defn:Chebotarev} Let $K$ be a number field with
  ring of integers~$O$, and let $\rho \colon \Gal_K \rightarrow G$ be
  a continuous homomorphism from the absolute Galois group of $K$ into
  a finite group~$G$.  Let $P(\rho)$ denote the set of primes $\fp \in
  \Spec(O)$ such that $\mathfrak{p}$ is unramified in the extension $K
  \subset \overline{K}^{\ker \rho}$ and the Frobenius conjugacy class
  $(\Frob_\fp) \subset \Gal_K$ associated to $\fp$ lies in the kernel
  of~$\rho$.  A set $P \subset \Spec(O)$ is a Chebotarev set if it
  almost contains a set of the form $P(\rho)$, in the sense that
  $P(\rho) \smallsetminus P$ is finite.
\end{defn}

The intersection of two Chebotarev sets is a Chebotarev set.
Furthermore, by Chebotarev's Density Theorem, every Chebotarev set has
positive analytic density. We record the following corollary of
Theorem~\ref{thm:BC.finite.new}; compare the definition of $a(\Phi)
\in \mathcal{A}^+$ in~~\eqref{def:a(Phi)}.

\begin{cor} \label{cor:BC.finite.new.cbt} Let $\Phi$ be a non-trivial
  root system, and let $C \in \R$ and $a(\Phi) \in \mathcal{A}^+$ be
  as in Theorem~\textup{\ref{thm:BC.finite.new}}. Let $K$ be a number
  field with ring of integers $O$. For every affine group scheme
  $\mathbf{G}$ over $O$ whose generic fiber is connected semi-simple
  with absolute root system~$\Phi$, the set of primes $\fp \in
  \Spec(O)$ such that
  \[
  \zeta_{\mathbf{G}(O/\fp)}-| \mathbf{G}(O/\fp) /
  [\mathbf{G}(O/\fp),\mathbf{G}(O/\fp)] | \sim_C \xi_{a(\Phi),| O/\fp
    |}
  \]
  is a Chebotarev set.
\end{cor}

\begin{proof}
  Let $\mathbf{G}$ be as in the corollary, and $m\in\N$ as in
  Theorem~\ref{thm:BC.finite.new}. The set of
  primes $\fp$ such that the reduction of $\mathbf{G}$ modulo
  $\fp$ is connected, semi-simple, and has absolute root
  system $\Phi$ is cofinite. If $M \supset K$ is a splitting field of
  $\mathbf{G}$ and $\fp$ splits completely in $M$, then the
  reduction of $\mathbf{G}$ modulo $\fp$ is
  split. Hence, there exists a Chebotarev set $P
  \subset \Spec(O)$ such that, for all $\fp\in P$,
  \begin{list}{}{\setlength{\leftmargin}{\myenumilistleftmargin}
      \setlength{\labelwidth}{20pt} \setlength{\itemsep}{0pt}
      \setlength{\parsep}{1pt}}
  \item[\textup{(1)}] The reduction of $\mathbf{G}$ modulo
    $\fp$ is connected, split, semi-simple, and has absolute
    root system $\Phi$.
  \item[\textup{(2)}] The field $O/\fp$ contains primitive $m$th roots
    of unity, that is $|O/\fp|\equiv_m 1$.
  \end{list}
  Let $\fp\in P$ and set $q = |O/\fp|$. By property (1), the group
  $\mathbf{G}(O/\fp)$ has Lie type $\mathcal{L} =
  \mathcal{L}^{\mathrm{sp}} = (\Phi,\Id)$. Property (2) together with
  parts (3) and (2) of Theorem~\ref{thm:BC.finite.new} yields
  \[
  \zeta_{\mathbf{G}(O/\fp)}-| \mathbf{G}(O/\fp) /
  [\mathbf{G}(O/\fp),\mathbf{G}(O/\fp)] | \sim_C
  \xi_{a(\Phi),q}.
  \]
\end{proof}

Furthermore, an argument similar to the one used in the proof of
Theorem~\ref{thm:BC.global}, now based on
Theorem~\ref{thm:BC.finite.new} and
Corollary~\ref{cor:BC.finite.new.cbt} instead of
Theorem~\ref{thm:zeta.approx}, gives the following.

\begin{cor} \label{cor:prod.fin.:ie.type} Let $K \subset L$ be
  number fields with rings of integers~$O_K \subset O_L$, and let
  $\mathbf{G}$ be an affine group scheme defined over~$O_K$ whose
  generic fiber is connected and simply connected
  semi-simple.  Then $\alpha (\prod_{\fp \in \Spec(O_K)}
  \mathbf{G}(O_K/\fp) ) = \alpha(\prod_{\fq \in
    \Spec(O_L)} \mathbf{G}(O_L/\fq))$.
\end{cor}

We remark that one can use Deligne-Lusztig theory to pin down the
precise value of $\alpha (\prod_\fp \mathbf{G}(O/\fp) )$ as in
Corollary~\ref{cor:prod.fin.:ie.type}.  For instance, if $\mathbf{G}$
is simple of type~$A_\ell$, then the abscissa of convergence for the
product $\prod_{\fp} \mathbf{G}(O/\fp)$ of finite groups of Lie type
is equal to $2/\ell$ and, in particular, tends to $0$ as $\ell \to
\infty$.  This behavior stands in contrast to the fact the abscissa
of convergence $\alpha(\mathbf{G}(O))$ is known to be bounded away
from~$0$; cf.~\cite[Theorem~8.1]{LL}. This underlines that the
investigation of $\alpha(\Gamma)$ for arithmetic groups $\Gamma =
\mathbf{G}(O_S)$ requires a more careful analysis.

%%%%% NO CHANGES BEYOND THIS LINE BK 31 January 2014

\section{Relative Zeta Functions, Kirillov Orbit Method and
Model Theoretic Background} \label{sec:preliminaries}

\subsection{Relative Zeta Functions and Cohomology}

Throughout this section, let $G$ be a group such that $R_n(G)$ is
finite for all~$n\in\N$.  Let $N \subset G$ be a normal subgroup, and
let $\theta$ be an irreducible finite-dimensional complex
representation of~$N$.  Recall from Definition~\ref{def:rel.zeta.func}
that $\Irr(G \vert \theta)$ denotes the set of (equivalence classes
of) finite-dimensional irreducible complex representations $\rho$ of
$G$ such that $\theta$ is a constituent of $\Res^G_N\rho$, the notion
of the relative zeta function $\zeta_{G \vert \theta}(s) = \sum_{\rho
  \in \Irr(G \vert \theta)} \left( \frac{\dim \rho}{\dim \theta}
\right)^{-s}$, and the notation $R_n(G \vert \theta)$ for the number
of representations $\rho \in \Irr(G \vert \theta)$ such that $\dim
\rho \leq n \dim \theta$.

Suppose further that, up to equivalence, $\theta$ is $G$-invariant.
It is typically not true that the relative zeta function $\zeta_{G
  \vert \theta}$ is equal to the zeta function of the quotient $G/N$;
for instance, if $G$ is non-abelian, step-$2$ nilpotent, $N = [G,G]$
and $\theta$ is non-trivial, then $\zeta_{G \vert \theta}$ is not
equal to $\zeta_{G/N}$.  We describe now a situation in which the two
zeta functions are equal.  Recall that $\theta$ defines an element in
the second cohomology group $H^2(G/N, \C^\times)$ of $G/N$ with values
in $\C^ \times$, also known as the Schur multiplier of $G/N$.  The
construction is as follows; see~\cite[Chapter~11]{Is}. Suppose $\theta
\colon N \rightarrow \GL_d(\C)$. Pick a coset representative
$\widetilde{a}\in G$ for every element $a$ of $G/N$ such
that~$\widetilde{1}=1$. For every $a\in G/N$, the representations
$\theta$ and $\theta ^{\widetilde{a}}$ are equivalent, and we choose
$T_a\in \GL_d(\C)$ such that $T_a \theta T_a ^{-1} = \theta
^{\widetilde{a}}$; for $T_1$ we choose the identity.  Then one checks
that, for all $a_1, a_2 \in G/N$, the transformation $T_{a_1 a_2}^{-1}
T_{a_1} T_{a_2} \theta ( \widetilde{a_1} \widetilde{a_2}
(\widetilde{a_1 a_2}) ^{-1} )$ commutes with each element of
$\theta(N)$ and thus defines a scalar $\beta(a_1,a_2) \in \C^\times$.
The map $\beta \colon G/N \times G/N \rightarrow \C^\times$ is a
$2$-cocycle representing the cohomology class associated to
$(G,N,\theta)$, which is independent of the choices involved.

\begin{lem} \label{lem:zeta.relative.quotient} Let $G$ be a
  profinite group with an open normal subgroup $N \triangleleft G$,
  and let $\theta \in \Irr(N)$ be a complex representation of $N$
  which is $G$-invariant up to equivalence.  If the cohomology class
  in $H^2(G/N, \C ^ \times)$ associated to $(G,N,\theta)$
  vanishes, then $\zeta_{G \vert \theta} = \zeta_{G/N}$.
\end{lem}

\begin{proof} By~\cite[Theorem~11.7]{Is}, the vanishing of the
  cohomology class implies that $\theta$ can be extended to a
  representation $\widetilde{\theta}$ of $G$. By~\cite[Theorem~6.16]{Is}, the map $\Irr(G/N) \rightarrow \Irr(G| \theta)$ given by
  $\tau \mapsto \tau \otimes \widetilde{\theta}$ is a bijection, and the
  claim of the lemma follows.
\end{proof}

\begin{lem} \label{lem:p.power} Let $p$ be a prime number. Let $G$ be
  a profinite group with an open normal pro-$p$ subgroup $N
  \triangleleft G$, and let $\theta\in\Irr(N)$ be $G$-invariant up to
  equivalence. Then the cohomology class in $H^2(G/N, \C ^ \times)$
  associated to $(G,N,\theta)$ has order a power of $p$.
\end{lem}

\begin{proof} The dimension of $\theta$ is a power of $p$ and, for
  every $h\in N$, the scalar $\det(\theta(h))$ is a $p^n$th root of
  unity, for some $n$.  For all $a_1, a_2 \in G/N$, taking
  determinants in the definition of the cocycle $\beta$, we get
  $\beta(a_1,a_2)^{\dim \theta} = \det ( T_{a_1 a_2}^{-1} T_{a_1}
  T_{a_2} \theta ( \widetilde{a_1} \widetilde{a_2} (\widetilde{a_1
    a_2})^{-1}))$. Since we are free to arrange $\det(T_x)=1$ for~$x
  \in G/N$, we get that $\beta(a,b)$ is a root of unity of order a
  power of~$p$.
\end{proof}

\begin{lem} \label{lem:SS} Let $\ell$ be a prime number. Suppose that
  $G$ is a finite group and that $N \subset G$ is a central subgroup
  such that $\gcd ( \lvert N \rvert , \ell ) = 1$.  Then $\gcd( \lvert
  H^2(G, \C ^ \times ) \rvert, \ell ) = 1$ if and only if
  $\gcd ( \lvert H^2(G/N, \C ^ \times )\rvert , \ell ) = 1$.
\end{lem}

\begin{proof} Let $(E_n^{p,q})$ be the Lyndon--Hochschild--Serre
  spectral sequence associated to the central extension $1 \rightarrow
  N \rightarrow G \rightarrow G/N \rightarrow 1$. Since the order of
  $N$ is prime to $\ell$, so are the orders of $H^1(N,\C ^
  \times)$ and $H^2(N,\C ^ \times)$.  Therefore, the orders of
  \[
  E_2^{0,2}=H^0(G/N,H^2(N,\C ^ \times)) = H^2(N, \C ^
  \times)
  \]
  and
  \[
  E_2^{1,1}=H^1(G/N,H^1(N,\C ^ \times)) =
  \Hom(G/N,H^1(N,\C ^ \times))
  \]
  are prime to $\ell$ and, hence, so are the orders of
  $E_\infty^{0,2}$ and $E_\infty^{1,1}$. The fact that
  $E_\infty^{p,q}$ converges to $H^*(G,\C ^ \times)$ implies
  that $\lvert H^2(G,\C ^ \times) \rvert = \lvert
  E_\infty^{0,2} \rvert \; \lvert E_\infty^{1,1} \rvert \; \lvert
  E_\infty^{2,0} \rvert$, so $\lvert H^2(G,\C ^ \times)
  \rvert$ is prime to $\ell$ if and only if $\lvert E_\infty^{2,0}
  \rvert$ is prime to~$\ell$.

  Since the order of $E_2^{0,1}=H^0(G/N,H^1(N,\C ^
  \times))=H^1(N,\C^\times)$ is prime to $\ell$, we get that the order
  of $E_\infty^{2,0}=E_3^{2,0}$ is prime to $\ell$ if and only if the
  order of $E_2^{2,0}=H^2(G/N,\C^\times)$ is prime to~$\ell$, yielding
  the result.
\end{proof}

\begin{lem} \label{lem:bdd.Schur} For every root system $\Phi$ there
  is a constant $C \in \R$ such that, for every finite field
  $\F _q$ of characteristic greater than $C$ and for every
  connected reductive $\F_q$-algebraic group $\mathbf{G}$ with
  absolute root system $\Phi$, the size of $H^2(\mathbf{G}(\F
  _q), \C ^ \times)$ is prime to~$q$.
\end{lem}

\begin{proof} Let $\F_q$ be a finite field of characteristic
  $p$ and let $\mathbf{G}$ be a connected reductive
  $\F_q$-algebraic group with absolute root system $\Phi$.

  \nir{Assume first that $\mathbf{G}$ is semi-simple.  Then there are
    almost simple groups $\mathbf{G}_1,\ldots,\mathbf{G}_n$ such that
    $\mathbf{G}$ is a quotient of $\mathbf{G}_1 \times \cdots \times
    \mathbf{G}_n$ by a central subgroup $\mathbf{Z}$, and both $n$ and
    the ranks of the groups $\mathbf{G}_i$ are bounded in terms
    of~$\Phi$.  In particular, the size of $\mathbf{Z}$ is bounded in
    terms of~$\Phi$.  From the exact sequence
    \[
    0 \rightarrow \mathbf{Z}(\mathbb{F}_q) \rightarrow
    \prod\nolimits_{i=1}^n \mathbf{G}_i(\mathbb{F}_q) \rightarrow
    \mathbf{G}(\mathbb{F}_q) \rightarrow
    H^1(\Gal_{\mathbb{F}_q},\mathbf{Z})
    \]
    we conclude that both the kernel and the cokernel of the map
    $\prod_{i=1}^n \mathbf{G}_i(\mathbb{F}_q) \rightarrow
    \mathbf{G}(\mathbb{F}_q)$ have sizes bounded in terms of $\Phi$. }

%  Assume first that $\mathbf{G}$ is semi-simple. Then there are almost
% simple groups $\mathbf{G}_1,\ldots,\mathbf{G}_n$ such that
%  $\mathbf{G}(\F_q)$ is a quotient of $\mathbf{G}_1(\F_q)
%  \times \ldots \times \mathbf{G}_n(\F_q)$ by a central subgroup, and
%  both $n$ and the ranks of the groups $\mathbf{G}_i$ are bounded in
%  terms of~$\Phi$.  In particular, the size of the kernel of the
%  quotient map is bounded in terms of~$\Phi$.  

It is known that the
  sizes of $H^1(\mathbf{G}_i(\F_q),\C ^ \times)$ and
  $H^2(\mathbf{G}_i(\F_q),\C ^ \times)$ are bounded in terms
  of~$\Phi$; see, for example, \cite[Table~5]{Atlas}. By the K\"unneth
  formula, the sizes of $H^1(\prod_{i=1}^n \mathbf{G}_i(\F_q),
  \C ^ \times)$ and $H^2(\prod_{i=1}^n \mathbf{G}_i(\F_q),
  \C ^ \times)$ are bounded by some constant~$C_1 \in \R$.  In
  particular, if $p$ is greater than~$C_1$, then the size of
  $H^2(\prod_{i=1}^n \mathbf{G}_i(\F_q), \C ^ \times)$ is
  prime to $q$.  By Lemma~\ref{lem:SS}, the same is true for the size
  of $H^2(\mathbf{G}(\F_q), \C ^ \times)$ if $p$ is
  larger than the size of the kernel \nir{and cokernel} of the quotient map
  $\prod_{i=1}^n \mathbf{G}_i(\F_q) \rightarrow
  \mathbf{G}(\F_q)$.

  Now assume that $\mathbf{G}$ is merely reductive. Let
  $\mathbf{S}=[\mathbf{G},\mathbf{G}]$ be the derived subgroup of
  $\mathbf{G}$ and let $\mathbf{T}=Z(\mathbf{G})^{\circ}$. Then
  $\mathbf{T}$ is a torus, $\mathbf{S}$ is semi-simple, and
  $\mathbf{G}(\F_q)$ is a quotient of
  $\mathbf{T}(\F_q) \times \mathbf{S}(\F_q)$ by a
  central subgroup, whose size is bounded in terms of~$\Phi$. As shown
  above, if $p$ is sufficiently large, then the size of
  $H^2(\mathbf{S}(\F _q),\C ^ \times)$ is prime to
  $q$; a similar claim for $H^1(\mathbf{S}(\F _q),\C ^
  \times)$ also holds. Since the size of $\mathbf{T}(\F _q)$
  is prime to $q$, so are the orders of its first and second
  cohomology groups. By the K\"unneth formula, the size of $H^2 \left(
    \mathbf{T}(\F _q) \times \mathbf{S}(\F _q),
    \C ^ \times \right)$ is prime to $q$. By
  Lemma~\ref{lem:SS}, if we assume, in addition, that $p$ is larger
  than the size of the kernel of $\mathbf{T}(\F_q) \times
  \mathbf{S}(\F_q) \rightarrow \mathbf{G}(\F_q)$, then
  the size of $H^2(\mathbf{G}(\F_q), \C ^ \times)$ is
  prime to $q$.
\end{proof}

\subsection{Kirillov Orbit Method}

All pro-$p$ groups in this section are open subgroups of
$\mathbf{G}(O_{L, \fq})$, where $\mathbf{G}$ is an affine group scheme
over the ring of integers $O_L$ of a number field $L$ and $O_{L,\fq}$
is the completion of $O_L$ at a prime $\fq$ lying above a rational
prime~$p$.  We fix an embedding $\mathbf{G} \subset \GL_N$, for a
suitable $N \in \N$, and denote by $\mathfrak{g} \subset
\mathfrak{gl}_N$ the Lie algebra of~$\mathbf{G}$.

\begin{defn} \label{def:good} Let $\fq$ be a prime
  of~$O_L$.  We say that a pro-$p$ subgroup $H \subset
  \mathbf{G}(O_{L,\fq})$ is good if the following two
  conditions hold.
  \begin{list}{}{\setlength{\leftmargin}{\myenumilistleftmargin}
      \setlength{\labelwidth}{20pt} \setlength{\itemsep}{0pt}
      \setlength{\parsep}{1pt}}
  \item[\textup{(1)}] The logarithm series
    $
    \log(X) = \sum_{n=1}^\infty (-1)^{n-1} \frac{(X-1)^n}{n}
    $
    converges on $H$, setting up an injective map $\log \colon H
    \rightarrow \log H \subset \mathfrak{g}(O_{L,\fq})$, and
    the exponential series
    $
    \exp(X) = \sum_{n=0}^\infty \frac{X^n}{n!}
    $
    converges on $\log H$, yielding the inverse map $\exp \colon
    \log H \rightarrow H$.
  \item[\textup{(2)}] The image $\log H$ is closed under addition and
    the Lie bracket, thus forming a $\mathbb{Z}_p$-Lie lattice.  It is
    also closed under the adjoint action of~$H$. For all $A,B \in
    \log H$, the Hausdorff formula (e.g., see~\cite[Chapter~V]{Jac})
    holds:
    \[
    \log (\exp (A)\cdot \exp (B)) = \sum_{m=1}
    ^{\infty}\frac{(-1)^m}{m}\sum_{r_i+s_i >0}\frac{\left (
        \sum_{i=1}^m (r_i+s_i) \right ) ^{-1}}{r_1!\cdot s_1! \cdot
      \ldots \cdot r_m!\cdot s_m!} R_{r_1,s_1,\ldots ,r_m,s_m}(A,B),
    \]
    where the Lie polynomials $R_{r_1,s_1,\ldots
      ,r_m,s_m}(A,B)$ are defined by
    \[
    R_{r_1,s_1,\ldots ,r_m,s_m}(A,B) =
    \begin{cases} \ad(A)^{r_1} \ad(B)^{s_1} \cdots \ad(A)^{r_m}(B) &
      \text{if $s_m=1$,} \\ \ad(A)^{r_1} \ad(B)^{s_1} \cdots
      \ad(B)^{r_{m-1}}(A) & \text{if $r_m=1$, $s_m=0$,} \\ 0 &
      \text{otherwise.}
    \end{cases}
    \]
  \end{list}
\end{defn}

\begin{lem} \label{lem:good.subgroup} Let $\mathbf{G} \subset \GL_N$
  be as above and let $\fq$ be a prime of $O_L$ extending a rational
  prime $p > [L:\Q] N^2$.  Then every pro-$p$ subgroup $H \subset
  \mathbf{G}(O_{L, \fq})$ is good.
\end{lem}

\begin{proof} Pro-$p$ groups which are saturable in the sense of
  Lazard -- for recent characterizations see~\cite{K,GK} -- are good
  in the sense of Definition~\ref{def:good}.  The assertion thus
  follows from~\cite[Corollary~1.5]{K} which implies that every
  pro-$p$ subgroup $H \subset \GL_N(O_{L, \fq})$ is saturable.
\end{proof}

A good pro-$p$ group $H$ acts on $\mathfrak{h} = \log H$ via the
adjoint action $\Ad \colon H \rightarrow \Aut(\mathfrak{h})$.  We
denote by $\Ad(h)(A)$ the image of $A \in \mathfrak{h}$ under the
adjoint action of $h \in H$.  The adjoint action induces the
co-adjoint action of $H$ on the Pontryagin dual $\mathfrak{h}^\vee =
\Hom_\text{cont}(\mathfrak{h},\C^\times)$, consisting of all
continuous homomorphisms from the abelian pro-$p$ group $\mathfrak{h}$
to $\C^\times$.  Concretely, for $h \in H$, $A \in \mathfrak{h}$, and
$\theta \colon \mathfrak{h} \rightarrow \C^\times$, one defines
\[
\left(\Ad^*(h)(\theta)\right)(A) = \theta \left( \Ad(h^{-1})(A)
\right).
\]
Equivalence classes of irreducible representations of $H$ are
parametrized by the corresponding characters.  Accordingly, we use the
notation $\Irr(H)$ in a flexible way to denote also the set of
irreducible complex characters of~$H$.  The Kirillov orbit method for
$p$-adic analytic pro-$p$ groups yields the following description of
the irreducible characters of~$H$.

\begin{prop} \label{prop:orbit.method} Let $\mathbf{G} \subset \GL_N$
  be as above. For almost all primes $\mathfrak{q}$ of $O_L$, every
  pro-$p$ subgroup $H \subset \mathbf{G}(O_{L,\fq})$ is good and,
  setting $\mathfrak{h} = \log H$, the following hold.
  \begin{list}{}{\setlength{\leftmargin}{\myenumilistleftmargin}
      \setlength{\labelwidth}{20pt} \setlength{\itemsep}{0pt}
      \setlength{\parsep}{1pt}}
  \item[\textup{(1)}] There is a function $\Omega \colon \mathfrak{h}
    ^\vee \rightarrow \Irr(H)$ which is constant on co-adjoint orbits
    and induces a bijection between the set of co-adjoint orbits in
    $\mathfrak{h}^\vee$ and the set of irreducible characters of~$H$.
  \item[\textup{(2)}] For $\theta \in \mathfrak{h} ^\vee$, the character
    $\Omega(\theta)$ is given by
    \[
    \Omega(\theta)(h) = \frac{1}{\lvert \Ad^*(H)(\theta)
      \rvert^{1/2}} \sum_{\phi \in \Ad^*(H)(\theta)} \phi(\log(h))
    \qquad (h \in H).
    \]
    In particular, the degree of $\Omega(\theta)$ is
    $\dim\Omega(\theta) = \lvert \Ad^*(H)(\theta) \rvert^{1/2}$.
  \item[\textup{(3)}] Every $g \in \mathbf{G}(O_{L,\fq})$ that
    normalizes the subgroup $H$ also normalizes the Lie lattice
    $\mathfrak{h}$ and $\Omega(\theta)^{\, g} =
    \Omega(\Ad^*(g^{-1})(\theta))$ for $\theta \in \mathfrak{h}
    ^\vee$.
  \end{list}
\end{prop}

\begin{proof} Let $S$ be the finite set of primes $\fq$ of $O_L$ that
  extend a rational prime~$p$ with $p \leq [L:\Q] N^2$.  The
  assumptions then imply that $H \subset \mathbf{G}(O_{L,\fq})$ is
  saturable for $\fq \notin S$ (see~\cite[Theorem~A]{GK}), in
  particular good in the sense of Definition~\ref{def:good}. As
  saturable pro-$p$ groups of dimension at most $p$ are potent, the
  assertions follow from~\cite[Theorem~5.2]{G}.
\end{proof}

We refer to the map $\Omega$ in Proposition~\ref{prop:orbit.method} as
the orbit method map.

\begin{lem} \label{lem:zeta.to.relative} Let $\mathbf{G} \subset
  \GL_N$ be as above, and let $G$ be an open subgroup of
  $\mathbf{G}(O_{L,\fq})$ for some prime $\fq$ of $O_L$, with open
  normal subgroups $K \subset H \subset G$. Suppose that $H$ and $K$
  are good pro-$p$ groups, with Lie lattices $\mathfrak{h}=\log H$ and
  $\mathfrak{k}=\log K$, and that the irreducible characters of $H$
  and $K$ are described by the orbit method as in
  Proposition~\textup{\ref{prop:orbit.method}}. Then
  \[
  \zeta_G(s) = \sum_{\theta \in \mathfrak{h}^\vee} \frac{1}{\lvert
    \Ad^*(G) (\theta) \rvert}(\dim \Omega(\theta))^{-s} \zeta_{G \vert
    \Omega(\theta)}(s).
  \]
  Furthermore, if $\tau \in \mathfrak{k} ^\vee$, then
  \[
  \zeta_{G \vert \Omega(\tau)}(s) = \sum_{\theta \in
    \mathfrak{h}^\vee, \hspace{0.1cm} \theta \vert_\mathfrak{k} =
    \tau} \frac{\lvert \Ad^*(H) (\tau) \rvert}{\lvert \Ad^*(G)
    (\theta) \rvert} \left( \frac{\dim \Omega (\theta)}{\dim \Omega
      (\tau)}\right)^{-s} \zeta_{G|\Omega (\theta)}(s).
  \]
\end{lem}

\begin{proof} For each $\chi \in \Irr(H)$, the set $\Irr(G \vert
  \chi)$ depends only on the $G$-orbit $\chi^G$.  Moreover, the
  sets $\Irr(G| \chi)$, indexed by the orbits $\chi^G$ of
  irreducible characters of $H$, form a partition of
  $\Irr(G)$. Choosing representatives $\chi_i$, $i\in I$, for the
  $G$-orbits, we get
  \[
  \zeta_G(s)=\sum_{i\in I} (\dim \chi_i)^{-s} \zeta_{G| \chi_i}(s)
  = \sum_{ \chi \in \Irr(H)} \lvert \chi ^G \rvert^{-1} (\dim
  \chi)^{-s} \zeta_{G| \chi }(s).
  \]
  Consider the orbit method map $\Omega \colon \mathfrak{h} ^\vee
  \rightarrow \Irr(H)$. Its fibers are the $H$-coadjoint
  orbits. Moreover, $\Omega$ is $G$-equivariant, so the pre-images of
  the $G$-orbits in $\Irr(H)$ are the $G$-orbits in $\mathfrak{h}
  ^\vee$. Hence,
  \begin{align*}
    \zeta_G(s) & = \sum_{\chi \in \Irr(H)} \lvert \chi^G
    \rvert^{-1} (\dim \chi)^{-s} \zeta_{G \vert \chi}(s) \\
    & = \sum_{\theta \in \mathfrak{h}^\vee} \frac{1}{\lvert \Ad^*(H)
      (\theta) \rvert \cdot \lvert \Omega(\theta)^G \rvert}(\dim
    \Omega (\theta))^{-s} \zeta_{G \vert \Omega (\theta)}(s) \\
    & =\sum_{\theta \in \mathfrak{h}^\vee} \frac{1}{\lvert \Ad^*(G)
      (\theta) \rvert} (\dim \Omega (\theta))^{-s} \zeta_{G \vert
      \Omega (\theta)}(s).
  \end{align*}
  The proof of the second statement is similar, using the following
  consequence of Proposition~\ref{prop:orbit.method}: for $\tau \in
  \mathfrak{k}^\vee$ and $\theta \in \mathfrak{h} ^\vee$, the
  character $\Omega(\tau)$ is a constituent of the restriction of
  $\Omega(\theta)$ to $K$ if and only if $\theta \vert_\mathfrak{k} =
  \tau^h$ for a suitable $h \in H$.
\end{proof}

\subsection{Quantifier-Free Definable Sets and
  Functions}\label{subsec:def.fun}

  We use several
notions from model theory, which we summarize below.  For more details
we refer to~\cite{CK}.  Fix a first-order language and a theory $\Th$,
that is a consistent set of sentences, in that language.  Let
$(\mathrm{Models}_{\,\Th})$ be the category whose objects are models
of $\Th$ and whose morphisms are elementary embeddings.  Let $x =
(x_1,\ldots,x_n)$ denote an $n$-tuple of variables.  A formula
$\phi(x)$ gives rise to a functor $\mathcal{X} = \mathcal{X}_\phi
\colon (\mathrm{Models}_{\,\Th}) \rightarrow (\mathrm{Sets})$ that
sends a model $M$ of $\Th$ to the set
\[
\mathcal{X}(M) = \left\{ (a_1,\ldots,a_n) \in M^n \mid \text{the
    sentence $\phi(a_1,\ldots,a_n)$ holds in $M$} \right\}.
\]
We call such a functor $\mathcal{X}$ a \emph{definable functor} or a
\emph{definable set} in~$\Th$, and denote $\mathcal{X}_\phi$ also by
$\left\{ x \mid \phi(x) \right\}$.  By G\"odel's completeness theorem,
the functors associated to formulae $\phi$ and $\psi$ are equal if and
only if the sentence $\forall x \, (\phi(x) \leftrightarrow \psi(x))$
can be proved from $\Th$.  We say that a definable set is
\emph{quantifier-free definable} if it is associated to some
quantifier-free formula. For a definable set $\mathcal{X}$ arising
from a formula~$\phi$, the notation $a \in \mathcal{X}$ is employed in
two different contexts: if $a
\in M^n$ for some model $M$ of $\Th$ it means $\phi(a)$ holds in~$M$,
whereas if $a$ is a tuple of terms in the underlying language it
simply stands for~$\phi(a)$.

The usual pointwise operators $\cap$, $\cup$, and $\times$ on functors
to sets take definable sets to definable sets. If $\mathcal{X}$ and
$\mathcal{Y}$ are definable sets, we write $\mathcal{X} \subset
\mathcal{Y}$ if $\mathcal{X}(M) \subset \mathcal{Y}(M)$ for all models
$M$ of~$\Th$. If $\mathcal{X}$ and $\mathcal{Y}$ are associated to the
formulae $\phi$ and $\psi$, then $\mathcal{X} \subset \mathcal{Y}$ if
and only if the sentence $\forall x \, (\phi(x) \rightarrow \psi(x))$
can be proved from~$\Th$.

\begin{exa}\label{exa:def_sets}
  We consider the first-order language of rings and the theory
  $\Th_\text{fields}$ of fields.
  \begin{list}{}{\setlength{\leftmargin}{\myenumilistleftmargin}
      \setlength{\labelwidth}{20pt} \setlength{\itemsep}{0pt}
      \setlength{\parsep}{1pt}}
  \item[\textup{(1)}] For every $n\in\N$, the formula $0 = 0$ (in variables $x_1,
    \ldots, x_n$) yields a definable set $\mathcal{X}$ with
    $\mathcal{X}(F) = F^n$ for every field~$F$. The associated
    definable set is called \emph{$n$-dimensional affine space}, and
    denoted by~$\mathbb{A} ^n$.
  \item[\textup{(2)}] More generally, every affine scheme $\mathbf{X}$
    over $\mathbb{Z}$ can be considered as a quantifier-free definable
    set.  This means that there is a quantifier-free definable set
    $\mathcal{X}$ such that $\mathcal{X}(F) = \mathbf{X}(F)$ for every
    field~$F$.
  \item[\textup{(3)}] The definable set $\mathcal{Y}$ arising from the
    formula $\exists y \, (y^2=x)$ is not quantifier-free.  Indeed,
    every quantifier-free definable set in the theory of fields is a
    Boolean combination of affine varieties.  It follows that if
    $\mathcal{X} \subset \mathbb{A}^1$ is quantifier-free, then there
    is a constant $C \in \R$ such that $\lvert \mathcal{X}(\F_p)
    \rvert$ or $\lvert \F_p \smallsetminus \mathcal{X}(\F_p) \rvert$
    is bounded by $C$, uniformly for all primes~$p$. However, $\lvert
    \mathcal{Y}(\F_p) \rvert = (p+1)/2$ if $p>2$.
  \end{list}
\end{exa}

The dimension of a definable set $\mathcal{X} \subset \mathbb{A}^n$ is
the dimension of its Zariski closure in $\mathbb{A}^n$;
cf.~\cite[Section~3]{CL}.  Let $\mathcal{X}$ and $\mathcal{Y}$ be
definable sets in a theory $\Th$.  A natural transformation $f\colon
\mathcal{X} \rightarrow \mathcal{Y}$ is called a \emph{definable
  function} if the functor sending $M\in(\mathrm{Models}_{\,\Th})$ to
the graph of the map $f(M)$ is definable.  This means that the graph
of $f$, considered as a functor, is definable.  On some occasions, if
$f \colon \mathcal{X} \rightarrow \mathcal{Y}$ is a definable
function, we say that $\mathcal{X}$ is a \emph{definable family} of
(definable) sets with base $\mathcal{Y}$.  For $y \in \mathcal{Y}$, we
denote the fiber $f ^{-1}(y)$ by $\mathcal{X}_y$.  It is a definable
set in the enriched language obtained by adding the coordinates of $y$
as constants.

Throughout the remainder of the section, let $R$ denote a commutative
unital ring.

\begin{defn}\label{exa:theory}
  Consider the language of rings enriched by constant symbols $c_a$
  for all $a\in R$.  Let $\Th_{\fields,R}$ be the theory consisting of
  the axioms of fields and the statements $c_a \cdot c_b=c_{ab}$,
  $c_a+c_b=c_{a+b}$ for all $a,b\in R$.
\end{defn}

A model for $\Th_{\fields,R}$ is a field $F$ together with a ring
homomorphism $R \rightarrow F$, which we routinely omit from the
notation.  As in Example~\ref{exa:def_sets}, every affine scheme over
$R$ gives rise to a quantifier-free definable set in
$\Th_{\fields,R}$.

We explain now how projective spaces arise as definable sets in the
theory~$\Th_{\fields,R}$.  Let $\mathfrak{g}$ be a free $R$-module of
finite rank $n$ with an $R$-basis~$e_1,\dots,e_n$. We view
$\mathfrak{g}$ as a definable set in $\Th_{\fields,R}$ via
$\mathfrak{g}(F)=\mathfrak{g}\otimes_RF$ for $F \in
(\mathrm{Models}_{\,\Th_{\fields,R}})$.  The chosen basis allows us to
identify $\mathfrak{g}$ with~$\mathbb{A}^{n}$.  The definable set
\[
\mathcal{P} = \left\{ (1) \right\} \times \mathbb{A}^{n-1} \, \sqcup \, \left\{
  (0,1) \right\} \times \mathbb{A} ^{n-2} \, \sqcup \, \ldots \,
\sqcup \, \left\{ (0,0,\ldots,1) \right\} \times \mathbb{A} ^0 \subset
\mathbb{A} ^{n}
\]
plays the role of \emph{projective space} over $\mathfrak{g}$ in the
category of definable sets, in the following sense. Consider the
definable family $\cv{\mathcal{V}} \subset \mathcal{P}
\times\mathfrak{g}$ given by the condition that
$((a_1,\ldots,a_n),(v_1,\ldots,v_n)) \in \mathcal{P} \times
\mathfrak{g}$ is in \cv{$\mathcal{V}$} if and only if all minors
$\left| \begin{matrix} a_i & a_j \\ v_i & v_j \end{matrix} \right|$,
$1 \leq i< j \leq n$, vanish.  For every definable set $\mathcal{S}$
and every definable set $\mathcal{X} \subset \mathcal{S} \times
\mathfrak{g}$ with the property that, for every $s \in \mathcal{S}$,
the fiber $\mathcal{X}_s$ is a line in $\mathfrak{g}$ -- that is, a
one-dimensional linear subfunctor of $\mathfrak{g}$ -- there is a
unique definable map $f \colon \mathcal{S} \rightarrow \mathcal{P}$
such that the pull-back $f^* \cv{\mathcal{V}} \subset \mathcal{S} \times
\mathfrak{g}$ of \cv{$\mathcal{V}$} via $f$ is equal to
$\mathcal{X}$. Choosing a different (definable) basis for
$\mathfrak{g}$, we obtain a different universal family
over~$\mathcal{P}$, but there is a quantifier-free definable map from
$\mathcal{P}$ to itself that interchanges the two universal families.

More generally, for $0 \leq d \leq n$, there is a quantifier-free
definable set that functions as the Grassmannian of $d$-dimensional
subspaces in $\mathfrak{g}\otimes_R F$ for every model $F$ of
$\Th_{\fields,R}$.  The union of these, over all dimensions $d$, is the
\emph{Grassmannian} $\Gr(\mathfrak{g})$ of~$\mathfrak{g}$.

\begin{defn}\label{def:grassmannian}
 Let $\mathfrak{g} \subset \mathfrak{gl}_N(R)$ be a free $R$-module
 which is closed under Lie brackets. In the theory~$\Th_{\fields,R}$,
 let $\Grass(\mathfrak{g})$ be the subfunctor of $\Gr(\mathfrak{g})$
 given by
 \[
 F \mapsto \left\{ \text{Lie subalgebras of $\mathfrak{g} \otimes_R
     F$}\right\},
 \]
 and let $\Grass^\mathrm{nilp}(\mathfrak{g})$ be the subfunctor of
 $\Grass(\mathfrak{g})$ given by
 \[
 F \mapsto \left\{ \text{Lie subalgebras of $\mathfrak{g} \otimes_R
     F$ consisting of nilpotent matrices} \right\}.
 \]
\end{defn}

\begin{prop} \label{prop:grass.operations} Let $R$ be a
  commutative unital ring, and let $\mathfrak{g} \subset
  \mathfrak{gl}_N(R)$ be a free $R$-module which is closed under Lie
  brackets. View $\mathfrak{g}$ as a definable set in $\Th_{\fields,R}$
  via $\mathfrak{g}(F) = \mathfrak{g} \otimes_R F$ for $F \in
  (\mathrm{Models}_{\,\Th_{\fields,R}})$.  Then the following hold.
  \begin{list}{}{\setlength{\leftmargin}{\myenumilistleftmargin}
      \setlength{\labelwidth}{20pt} \setlength{\itemsep}{0pt}
      \setlength{\parsep}{1pt}}
  \item[\textup{(1)}] $\Grass(\mathfrak{g})$ and
    $\Grass^\mathrm{nilp}(\mathfrak{g})$ are quantifier-free definable
    subfunctors of~$\Gr(\mathfrak{g})$.
  \item[\textup{(2)}] The natural transformation $\Grass(\mathfrak{g})
    \rightarrow \Grass^\mathrm{nilp}(\mathfrak{g})$ induced by taking
    a Lie algebra $h \subset \mathfrak{g}(F)$ to the subalgebra of
    nilpotent matrices in the solvable radical $\Rad({h})$ of ${h}$ is
    quantifier-free definable.
  \item[\textup{(3)}] The natural transformation
    $\Grass^\mathrm{nilp}(\mathfrak{g}) \rightarrow
    \Grass(\mathfrak{g})$ induced by taking a Lie algebra $h \subset
    \mathfrak{g}(F)$ consisting of nilpotent matrices to its
    normalizer $N_{\mathfrak{g}(F)}(h)$ in $\mathfrak{g}(F)$ is
    quantifier-free definable.
  \item[\textup{(4)}] There are a constant $p_0 \in \N$ and a
    quantifier-free definable function from $\Grass(\mathfrak{g})$ to
    the set of all root systems of rank at most $N$ such that the
    following is true: for every Lie algebra $h \subset
    \mathfrak{g}(F)$ with $F$ satisfying $\charac(F) = 0$ or
    $\charac(F) \geq p_0$, the Lie algebra~$h/\Rad(h)$ is semi-simple
    of classical type and the value of the function at $h$ is the
    absolute root system of~$h/\Rad(h)$.
  \item[\textup{(5)}] Suppose that $\mathfrak{g}$ is the Lie algebra
    of an affine group scheme $\mathbf{G}$ over~$R$.  Then there is a
    quantifier-free definable subset of $\mathbf{G} \times
    \Grass(\mathfrak{g})$ whose fiber over a Lie algebra $h \subset
    \mathfrak{g}(F)$ is the normalizer $N_{\mathbf{G}(F)}(h)$ of $h$
    in $\mathbf{G}(F)$.
  \end{list}
\end{prop}

We use the following well-known characterization of quantifier-free
definable sets; e.g., see~\cite[Theorem~8.11]{Kaz_MI}.

\begin{lem} \label{lem:criterion.for.q.f.}  Let $R$ be a commutative
  unital ring, and let $\mathcal{X} \subset \mathbb{A} ^n$ be a
  definable set in the theory $\Th_{\fields,R}$. Then the following are
  equivalent.
  \begin{list}{}{\setlength{\leftmargin}{\myenumilistleftmargin}
      \setlength{\labelwidth}{20pt} \setlength{\itemsep}{0pt}
      \setlength{\parsep}{1pt}}
  \item[\textup{(1)}] For every two models $F \subset E$ of
    $\Th_{\fields,R}$, we have
    \[
    \mathcal{X}(F)=\mathcal{X}(E) \cap F^n.
    \]
  \item[\textup{(2)}] For every model $F$ of $\Th_{\fields,R}$, we have
    \[
    \mathcal{X}(F)=\mathcal{X}(F^\alg)\cap F^n.
    \]
  \item[\textup{(3)}] $\mathcal{X}$ is quantifier-free definable.
  \end{list}
\end{lem}

\begin{proof}
  The implications $(1) \Rightarrow (2)$ and $(3) \Rightarrow (1)$ are
  clear. We prove $(2) \Rightarrow (3)$. By elimination of quantifiers
  over algebraically closed fields \cv{with coefficients in $R$},
  there is a quantifier-free definable set $\mathcal{Y}$ such that
  $\mathcal{X}(E)=\mathcal{Y}(E)$ for every algebraically closed model
  $E$ of $\Th_{\fields,R}$. Let $F$ be a model of
  $\Th_{\fields,R}$. By $(2)$ and the implication $(3) \Rightarrow
  (1)$, applied to $\mathcal{Y}$, we obtain
  $\mathcal{X}(F)=\mathcal{X}(F^\alg) \cap F^n=\mathcal{Y}(F^\alg)
  \cap F^n = \mathcal{Y}(F)$, so $\mathcal{X}=\mathcal{Y}$.
\end{proof}

\begin{rem} \label{rem:criterion.q.f.perfect} We also use variants of
  Lemma~\ref{lem:criterion.for.q.f.} which characterize
  quantifier-free definable sets in (i) the theory
  $\Th_{\text{perf.-fields},p,R}$ of perfect fields of characteristic
  $p$ together with a homomorphism from~$R$ and (ii) the theory
  $\Th_{\mathrm{Hen},K,0}$ of Henselian valued fields of residue
  characteristic~$0$ together with a homomorphism from field~$K$; cf.\
  Definition~\ref{def:theories}.
%Section~\ref{sec:valued.fields}.  
  The proof proceeds in the same way.
\end{rem}

\begin{proof}[Proof of Proposition \ref{prop:grass.operations}]
  We show in each case that the functor or the graph of the natural
  transformation in question is definable.  In some cases, the
  formulae that we supply are polynomial, and hence quantifier-free.
  In the remaining cases, we explain how to apply
  Lemma~\ref{lem:criterion.for.q.f.} in order to obtain that the
  functor or graph in question is actually quantifier-free definable.

  Put $n = \dim \mathfrak{g}$ and fix an $R$-basis $e_1,\dots,e_n$
  of~$\mathfrak{g}$.  There are a finite Zariski open affine cover
  $(\mathcal{U}_\iota)_{\iota \in I}$ of $\Gr(\mathfrak{g})$,
  non-negative integers $(d_\iota)_{\iota \in I}$, and regular
  functions $x_{\iota,1},\ldots,x_{\iota,n} \colon \mathcal{U}_\iota
  \rightarrow \mathfrak{g}$, for $\iota \in I$, such that for each
  model $F$ of $\Th_{\fields,R}$,
  \begin{list}{$\circ$}{\setlength{\leftmargin}{\mylistleftmargin}
      \setlength{\labelwidth}{10pt} \setlength{\itemsep}{0pt}
      \setlength{\parsep}{1pt}}
  \item the dimension of every $h \in \mathcal{U}_\iota(F)$ is equal to
    $d_\iota$,
  \item for every $h \in \mathcal{U}_\iota(F)$, the elements
    $x_{\iota,1}(h),\ldots,x_{\iota,d_\iota}(h)$ yield a linear basis
    for $h$ and the elements $x_{\iota,1}(h),\ldots,x_{\iota,n}(h)$
    yield a linear basis for $\mathfrak{g}(F)$.
  \end{list}
  In order to prove that a subfunctor $\mathcal{X} \subset
  \Gr(\mathfrak{g})$ is definable, it is enough to show that
  $\mathcal{X} \cap \mathcal{U}_\iota$ is definable for every
  $\iota\in I$.  A similar claim holds for natural transformations.
  Fix therefore $\iota, \iota' \in I$ and write
  \[
  \mathcal{U} = \mathcal{U}_\iota, \, \mathcal{U}' =
  \mathcal{U}_{\iota'}, \quad d = d_\iota, \, d' = d_{\iota'}, \quad
  \text{and} \quad x_i = x_{\iota,i}, \, x'_i = x_{\iota',i} \text{
    for $i \in \{1,\ldots,n\}$.}
  \]

  \medskip

  (1) Consider the $\left(\binom{d+1}{2} \times n \right)$-matrix $A$
  whose first $d$ rows record the coordinates of the functions
  $x_1,\dots,x_d$ and whose last $\binom{d}{2}$ rows record the
  coordinates of the Lie brackets $[x_i,x_j]$, $1\leq i<j \leq a$, all
  with respect to the basis~$e_1, \dots, e_n$.  The functor
  $\Grass(\mathfrak{g})\cap \mathcal{U}$ is definable by the polynomial
  condition $\rk A = d$.

  The functor $\Grass^\mathrm{nilp}(\mathfrak{g}) \cap \mathcal{U}$ is
  definable by the conjunction of the previous condition and the
  polynomial condition that all products of the $x_i$ of length $N$
  vanish, i.e., that $\prod_{j=1}^N x_{i_j}=0$ for all
  $(i_1,\dots,i_N)\in\{1,\dots,n\}^N$.

  \medskip

  (2) Let $F$ be a model of $\Th_{\fields,R}$ and $h \in
  (\Grass(\mathfrak{g}) \cap \mathcal{U})(F)$. Recall that an element
  $X\in h$ is in the solvable radical $\Rad({h})$ of $h$ if and only
  if the Lie ideal $[X,h]$ generated by $X$ in $h$ is solvable.  As an
  $F$-vector space, $[X,h]$ is spanned by the set $\mathcal{S}$
  consisting of all elements $[X,x_{i_1}(h),\dots,x_{i_{d-1}}(h)]$,
  where $i_1,\dots,i_{d-1} \in \{1,\dots,d\}$.  The Lie words $w_i$,
  $i \in \N$, defining the terms of the derived series are
  $w_1(z_1,z_2) = [z_1,z_2]$ and $w_i(z_1,\ldots,z_{2^i}) =
  [w_{i-1}(z_1,\ldots,z_{2^{i-1}}),
  w_{i-1}(z_{2^{i-1}+1},\ldots,z_{2^i})]$ for $i \geq 2$.  By
  linearity, $[X,h]$ is solvable if and only if
  $w_n(Z_1,\ldots,Z_{2^n}) = 0$ for all $Z_1,\dots,Z_{2^n} \in
  \mathcal{S}$.

  Using (1), we deduce that the functor
  \begin{equation} \label{eq:radical} F \mapsto \left\{ (h,X)\in
      (\Grass(\mathfrak{g}) \cap \mathcal{U})(F) \times
      \mathfrak{g}(F) \mid \text{$X\in \Rad(h)$}\right\}
  \end{equation}
  is quantifier-free definable.  Clearly, we can express nilpotency of
  an element $X \in h$ by a polynomial formula.  Consequently, also
  the functor
  \[
  F \mapsto \left\{ (h,X) \in (\Grass(\mathfrak{g}) \cap
    \mathcal{U})(F) \times \mathfrak{g}(F) \mid \text{$X\in \Rad(h)$
      is nilpotent}\right\}
  \]
  is quantifier-free definable.  Using quantifiers, we deduce that the
  functor
  \begin{equation}\label{equ:radical2}
    F \mapsto \left\{
      \begin{array}{l}
        (h,k) \in (\Grass(\mathfrak{g}) \cap \mathcal{U})(F)
        \times (\Grass^\mathrm{nilp}(\mathfrak{g}) \cap \mathcal{U}')(F) \mid \\
        \quad \text{$k$
          is the collection of nilpotent elements in $\Rad(h)$}
      \end{array} \right\}
  \end{equation}
  is definable and thus the graph of a natural transformation,
  namely the one we are interested in.

  It remains to prove that the functor~\eqref{equ:radical2} is
  \emph{quantifier-free} definable.  By
  Lemma~\ref{lem:criterion.for.q.f.}, it suffices to consider the
  following.  Let $F \subset E$ be models of $\Th_{\fields,R}$, and
  for subalgebras $h,k \subset \mathfrak{g}(F)$ write $h_E = h
  \otimes_F E$ and $k_E = k \otimes_F E$.  We claim: if $k_E$ consists
  of the nilpotent elements of $\Rad(h_E)$, then $k$ consists of the
  nilpotent elements of $\Rad(h)$.

  Clearly, $X \in h$ is nilpotent as an element of $h_E$ if and only
  if it is nilpotent as an element of $h$.  Moreover, because the
  functor~\eqref{eq:radical} is quantifier-free definable, we have
  $\Rad(h_E) \cap \mathfrak{g}(F) = \Rad(h)$.  This proves the claim.

  \medskip

  (3) We first show that the functor $\mathcal{N}$ given by
  \[
  F \mapsto \left\{ (h,k) \in \Gr(\mathfrak{g})(F)^2 \mid \forall Y\in
    \mathfrak{g}(F) \colon \left( Y \in k \; \leftrightarrow \; \left(
        \forall X \in h \colon [X,Y]\in h \right) \right) \right\}
  \]
  is definable.  As before, it is enough to prove this Zariski locally
  around fixed elements $h \in \mathcal{U}(F)$ and~$k \in
  \mathcal{U}'(F)$, of dimensions $d$ and~$d'$.  The intersection
  $\mathcal{N} \cap (\mathcal{U} \times \mathcal{U}')$ is given by the
  formula
  \begin{multline} \label{equ:normalizer.qfd}
    \forall a_1,\ldots,a_n  \\
    \left( \left( \bigwedge\nolimits_{i=1}^d \exists b_1,\dots,b_d
        \left( \left[x_i,\sum\nolimits_{j=1}^n a_j x_j' \right] =
          \sum\nolimits_{m=1}^d b_m x_m \right) \right) \;
      \longleftrightarrow \; a_{d'+1} = \ldots = a_n = 0 \right).
  \end{multline}
  We now show that this formula is equivalent to a quantifier-free
  formula.  To this end, consider first the $(d(d'+1) \times
  n)$-matrix $A$ whose first $d$ rows record the coordinates of
  $x_1,\dots,x_d$ and whose last $dd'$ rows record the coordinates of
  the Lie brackets $[x_i,x_j']$, where $1\leq i\leq d$ and $1\leq j
  \leq d'$, all with respect to the basis~$e_1, \ldots, e_n$.  The
  polynomial condition $\rk(A) = d$ ensures that $\langle
  x_1',\dots,x_{d'}' \rangle$ is contained in the Lie normalizer of
  $\langle x_1\dots,x_d \rangle$.  Consider now, for $1 \leq i \leq d$
  and coordinates $a_{d'+1},\dots,a_n$, the $((d+1) \times n)$-matrix
  $B_i(a_{d'+1},\dots,a_n)$ whose first $d$ rows record the
  coordinates of $x_1,\dots,x_d$ and whose last row records the
  coordinates of $[x_i,\sum_{j=d'+1}^n a_j x_j']$.
  Formula~\eqref{equ:normalizer.qfd} is equivalent to
  \begin{multline*}
    \left( \rk(A) = d \right) \, \wedge \\
    \forall a_{d'+1},\dots,a_n \left( \left(
        \bigwedge\nolimits_{i=1}^d \rk(B_i(a_{d'+1},\dots,a_n))=d
      \right) \; \longleftrightarrow \; a_{d'+1} = \dots = a_n = 0
    \right).
  \end{multline*}
  We claim that this is equivalent to a polynomial condition in the
  coordinate functions $x_1, \ldots, x_d$, $x_1', \ldots, x_{d'}'$.
  Indeed, for every $i \in \{1,\ldots,d\}$, the condition
  $\rk(B_i(a_{d'+1},\dots,a_n)) = d$ is linear in $a_{d'+1},\dots,a_n$
  and polynomial in $x_1, \ldots, x_d$, $x_1', \ldots, x_{d'}'$.  The
  condition that the resulting system of linear equations for
  $a_{d'+1},\dots,a_n$ only has the trivial solution is polynomial in
  $x_1, \ldots, x_d$, $x_1', \ldots, x_{d'}'$.

  \medskip (4) The theory of modular Lie algebras, and in particular
  the classification of semi-simple Lie algebras, is rather more
  involved than in characteristic~$0$.  However, there exists $p_0 \in
  \N$ such that the classification of semi-simple Lie algebras of
  dimension at most $n$ over every algebraically closed field of
  characteristic at least $p_0$ is completely analogous to the
  well-known classification in characteristic~$0$: the Lie algebras
  are of classical type and parametrized by suitable root systems.
  One can deduce this, for instance, from Robinson's Principle,
  according to which a first order statement in the language of fields
  is true in algebraically closed fields of characteristic~$0$ if and
  only if it is true in algebraically closed fields of sufficiently
  large characteristic.  We may thus restrict attention to fields $F$
  with $\charac(F) = 0$ or $\charac(F) \geq p_0$.

  Since the absolute root system of a Lie algebra does not change
  under extension of scalars, it suffices, by
  Lemma~\ref{lem:criterion.for.q.f.}, to produce a \emph{definable}
  function with the required properties.

  For every root system $\Phi$, let $\mathfrak{s}_\Phi$ be the split
  Lie algebra, given in terms of a corresponding Chevalley basis, with
  root system~$\Phi$; see~\cite[Chapter~IV]{Jac}.  For every
  subalgebra $h$ of $\mathfrak{gl}_N(F)$, the algebra $h / \Rad(h)$ is
  semi-simple and has rank at most~$N$.  By adjoining the
  characteristic roots of basis elements of a Cartan subalgebra of $h$
  in the adjoint action, one obtains a splitting field $E$ for~$h$.
  Clearly, the degree of such a field $E$ over $F$ is bounded in terms
  of~$N$.  Therefore, it suffices to show that, for any given $m \in
  \N$, the functor
  \begin{equation*}
    F \mapsto \left\{
      \begin{array}{l}
        h \in \Grass(\mathfrak{g})(F)
        \mid \\
        \quad \exists \text{ field ext.\ $F \subset E$ of degree $m$ such
          that } (h/\Rad(h)) \otimes_F E \cong
        \mathfrak{s}_\Phi (E)
      \end{array}
    \right\}
  \end{equation*}
  is definable.  Using Boolean combinations of formulae, it is enough
  to show that the functor
  \begin{equation} \label{eq:functor.Phi} F \mapsto \left\{
      \begin{array}{l}
        h \in \Grass(\mathfrak{g})(F) \mid \\
        \quad \text{$\exists$ field ext.\ $F \subset E$ of degree $m$
          and an embedding $\mathfrak{s}_\Phi(E) \rightarrow h
          \otimes_F E$}
      \end{array}
    \right\}
  \end{equation}
  is definable.

  It is enough to check definability on an open
  neighborhood~$\Grass(\mathfrak{g}) \cap \mathcal{U}$ of an element
  $h \in \Grass(\mathfrak{g})(F)$.  Field extensions $F \subset E$ of
  degree $m$ can be modelled on the vector space $F^m$ via a set of
  structure constants $\underline{c} = (c_{ij}^k)_{i,j,k=1}^{m}$
  in~$F$.  The latter supply a binary operation
  \[
  F^m \times F^m \rightarrow F^m, \quad (a_1,\dots,a_m) \ast
  (b_1,\dots,b_m) = \left( \sum\nolimits_{i,j=1}^m a_ib_j c_{ij}^k
  \right)_{k=1}^m.
  \]
  The condition that the multiplication $\ast$, together with vector
  addition, defines a field extension of $F$ is a first-order
  condition on~$\underline{c}$.  Furthermore, for every
  $\underline{c}$ giving rise to an extension $F \subset E$, an
  $E$-linear map $T \colon \mathfrak{s}_\Phi(E) \rightarrow h
  \otimes_F E$ can be described, locally, by a $(\dim
  \mathfrak{s}_\Phi \times d)$-matrix over $F^m$, with respect to the
  Chevalley basis of $\mathfrak{s}_\Phi$ and the basis
  $x_1(h),\ldots,x_d(h)$ of~$h$.  The condition that the map $T$ is an
  embedding and preserves Lie brackets is polynomial in the entries
  of the corresponding $(\dim \mathfrak{s}_\Phi \times d \times
  m)$-array~$\underline{t}$ over~$F$.  Thus the
  functor~\eqref{eq:functor.Phi} is indeed definable.

  \medskip

  (5) Let $\mathcal{X}$ be the functor
  \[
  F \mapsto \left\{ (g,h) \in \mathbf{G}(F) \times
    \Grass(\mathfrak{g})(F) \mid \Ad(g)h=h \right\} .
  \]
  The functor $\mathcal{X} \cap (\mathbf{G} \times
  (\Grass(\mathfrak{g})\cap \mathcal{U})))$ is definable by the
  formula
  \begin{equation}\label{equ:5.qfd}
    (\forall a_1,\ldots,a_d) (\exists b_1,\ldots,b_d)
    \left( \Ad (g)
      \left(\sum\nolimits_{j=1}^d a_j x_j(h) \right) =
      \sum\nolimits_{k=1}^d b_k x_k(h) \right),
  \end{equation}
  where the operation of $\Ad(g)$ is given by polynomial expressions
  involving the entries of the matrix~$g$.  To see that the functor is
  quantifier-free definable, consider the $(2d \times n)$-matrix $A =
  A(g,h)$, whose first $d$ rows record the coordinates of
  $x_1,\dots,x_d$ and whose last $d$ rows record the coordinates of
  $\Ad(g) x_1,\dots,\Ad(g) x_d$, all with respect to the
  basis~$e_1,\ldots,e_n$.  The polynomial condition $\rk(A(g,h)) = d$
  is equivalent to~\eqref{equ:5.qfd}.
\end{proof}

The following proposition can be found, for example,
in~\cite[Th\'eor\`eme~6.4]{Cha} or~\cite[Main Theorem]{CvdDM}.

\begin{prop} \label{prop:size.definable.set} Suppose that $\phi(x,y)$
  is a first-order formula in the language of rings, where $x =
  (x_1,\ldots,x_m)$ and $y=(y_1,\ldots,y_n)$.  There is a constant $C
  \in \R$ such that, for every finite field $\F_q$ and every
  $a \in \F_q^n$, there is a natural number $d$ such that the
  size of the set $\left\{ x \in \F_q^m \mid \phi(x,a) \textup{
      holds in } \F_q \right\}$ is either $0$, or between
  $C^{-1} q^d$ and $Cq^d$.
\end{prop}

\begin{lem} \label{lem:size.group} Let $\mathbf{G}$ be an affine
  algebraic group over a finite field $\F_q$ with at most $C$
  connected components, and let $\mathfrak{g}$ be the Lie algebra
  of~$\mathbf{G}$.
  \begin{list}{}{\setlength{\leftmargin}{\myenumilistleftmargin}
      \setlength{\labelwidth}{20pt} \setlength{\itemsep}{0pt}
      \setlength{\parsep}{1pt}}
  \item[\textup{(1)}] Writing $D_1 = C \, 2^{\dim \mathbf{G}}$, the
    estimates $D_1^{\, -1} \lvert \mathfrak{g}(\F_{q^n}) \rvert <
    \lvert \mathbf{G}(\F_{q^n}) \rvert < D_1 \lvert
    \mathfrak{g}(\F_{q^n}) \rvert$ hold for every finite extension
    $\F_q \subset \F_{q^n}$.
  \item[\textup{(2)}] Suppose that $\mathbf{G}$ acts on a variety
    $\mathbf{X}$ in such a way that the stabilizer $\mathbf{H}$ of a
    point $x \in \mathbf{X}(\F_{q^n})$ in~$\mathbf{G}$ has less than
    $C$ connected components.  Writing $\mathfrak{h}$ for Lie algebra
    of $\mathbf{H}$ and $D_2 = C^2 2^{\dim \mathbf{G}}$, the following
    estimates hold:
    \[
    D_2^{\, -1} \frac{\lvert \mathfrak{g}(\F_{q^n}) \rvert}{\lvert
      \mathfrak{h}(\F_{q^n}) \rvert} \leq \lvert \mathbf{G}(\F_{q^n})x
    \rvert \leq D_2 \frac{\lvert \mathfrak{g}(\F_{q^n}) \rvert}{\lvert
      \mathfrak{h}(\F_{q^n}) \rvert}.
    \]
  \end{list}
\end{lem}

\begin{proof} We use the
  inequality
  \[
  2^{-\dim \mathbf{G}} q^{n \dim \mathbf{G}} \leq \lvert
  \mathbf{G}^\circ(\F_{q^n}) \rvert \leq 2^{\dim \mathbf{G}}
  q^{n \dim \mathbf{G}}
  \]
  for the connected $\F_q$-algebraic group $\mathbf{G}^\circ$;
  cf.~\cite[Lemma~3.5]{N}.  Observing that $\lvert \mathbf{G}^\circ
  (\F_{q^n}) \rvert \leq \lvert \mathbf{G}(\F _{q^n}) \rvert \leq
  \lvert \mathbf{G}/\mathbf{G}^\circ|\cdot|\mathbf{G}^\circ(\F _{q^n})
  \rvert$ and $\lvert \mathfrak{g}(\F_{q^n}) \rvert = q^{n \dim
    \mathbf{G}}$, we deduce (1).  Claim (2) follows from~(1), by
  applying the orbit stabilizer theorem.
\end{proof}

\subsection{Valued Fields} \label{sec:valued.fields}

We use the Denef--Pas language of valued fields; see, for example
\cite[Section~2]{CL}.  It is a three-sorted, first-order language.
The three sorts are the valued field sort $\mathsf{F}$, the residue
field sort $\mathsf{k}$, and the value group sort~$\mathsf{\Gamma}$.
The function symbols are
\begin{align*}
  +_{\val},\times_{\val} &&& \text{from pairs of valued field sort
    variables to one valued field sort variable,} \\
  +_{\res},\times_{\res} &&& \text{from pairs of residue field sort
    variables to one residue field sort variable,} \\
  +_{\gr} &&& \text{from pairs of value group sort variables to one
    value group sort variable,} \\
  \val &&& \text{from one valued field sort to one value group sort,} \\
  \ac &&& \text{from one valued field sort to one residue field sort.}
\end{align*}
In addition there is one binary relation symbol, $<$, between two
value group sort variables.

For us, the important structures for the language of valued fields
come from discrete valuation fields.  Given a discrete valuation field
$F$ with a uniformizer $\varpi$, we interpret the valued field sort
$\mathsf{F}$ as $F$, the residue field sort $\mathsf{k}$ as the
residue field of $F$, and the value group sort $\mathsf{\Gamma}$ as
the value group of $F$ which we identify with~$\mathbb{Z}$.  The
functions $+_{\val},\times_{\val},+_{\res},\times_{\res},+_{\gr}$, and
the relation $<$ are interpreted as the usual operations and order
relation.  Finally, the function symbol $\val$ is interpreted as the
valuation map, and the function symbol $\ac$ is interpreted as the
angular component map
\[
\ac(x) \equiv x \varpi^{-\val(x)} \pmod{\varpi} \qquad \text{for $x \in
  F \smallsetminus \{0\}$.}
\]
The values of $\val(0)$
and $\ac(0)$ are chosen to be $\infty$ and~$0$.

We will only use theories for which $\mathsf{F}$ is a valued field
with residue field $\mathsf{k}$ and value group
$\mathsf{\Gamma}$. Definable sets and functions are introduced
similarly as for languages with only one variable sort. The definable
set $\{x \in \mathsf{F} \mid \val(x) \geq 0\}$ is denoted
by~$\mathcal{O}$. We let $\red: \mathcal{O} \rightarrow \mathsf{k}$
denote the reduction modulo the maximal ideal, i.e.\ the definable map
\[
\red(x) = \begin{cases} \ac(x), & \textrm{ if }\val(x)=0, \\ 0, &
  \textrm{ if }\val(x)>0.\end{cases}
\]
When several valuation rings are involved, we sometimes use subscripts
to distinguish between the various realizations of the reduction
map. We use the function symbol $\red(\cdot)$ also to denote the
componentwise reduction of a matrix or a tuple. In the latter case we
also write $\red^{\times n}$ to highlight the $n$-arity. Likewise, we
write $\ac^{\times n}(x)$ and $\val^{\times n}(x)$ when we apply the
map $\ac$ or $\val$ coordinatewise to a tuple~$x \in \mathsf{F}^n$.

For every discrete valuation field~$F$, the set $\mathcal{O}(F)$ is
the valuation ring~$O_F$.  Every $O_F$-scheme $\mathbf{X}$ gives rise
to three definable sets in the Denef--Pas language augmented by
constants from~$O_F$.  Indeed, let $x = (x_1,\ldots,x_n)$ denote an
$n$-tuple of variables and suppose that $\mathbf{X}$ is given as the
vanishing set of polynomials $f_1(x),\ldots,f_m(x) \in O_F[x]$.  The
first definable set is the set $\mathbf{X}_\mathsf{F}$ of all zeros of
$f_i(x)$, $1 \leq i \leq m$, in $\mathsf{F}^n$.  The second is the set
$\mathbf{X}_\mathsf{k}$ of zeros of the reductions of $f_i(x)$, $1
\leq i \leq m$, in $\mathsf{k}^n$.  The third is
$\mathbf{X}_\mathcal{O} = \mathbf{X}_\mathsf{F} \cap \mathcal{O}^n$.
For instance, if $\mathbf{G}$ is an affine group scheme defined
over~$O_F$, and $F \subset E$ is a finite extension of discrete
valuation fields, with ring of integers $O_E$ and residue field
$\mathbb{F}_q$, then
\[
\mathbf{G}_\mathsf{F}(E) = \mathbf{G}(E), \qquad \mathbf{G}_\mathsf{k}(E) =
\mathbf{G}(\mathbb{F}_q), \qquad \mathbf{G}_\mathcal{O}(E) =
\mathbf{G}(O_{E}).
\]
%$\mathbf{G}_\mathcal{O}(E)$ is equal to~$\mathbf{G}(O_{E})$.

The constructions of Section~\ref{subsec:def.fun} can be applied to
definable sets of sorts $\mathsf{F}$ and $\mathsf{k}$. For example,
let $F$ be a valued field and let $\mathfrak{g}$ be a Lie algebra
scheme over $O_F$. Applying Proposition~\ref{prop:grass.operations} to
the quantifier-free definable set $\mathfrak{g}_{\mathsf{k}}$, we get
a quantifier-free definable set $\Grass(\mathfrak{g}_{\mathsf{k}})$
such that, for every extension $F \subset E$, the set
$\Grass(\mathfrak{g}_{\mathsf{k}})(E)$ is the collection of all Lie
subalgebras of $\mathfrak{g}_{\mathsf{k}}(E)=\mathfrak{g}
\otimes_{O_F} \mathsf{k}(E)$. The assertions of
Proposition~\ref{prop:grass.operations} carry over to analogous
statements for definable sets of sorts $\mathsf{F}$ and $\mathsf{k}$.

\begin{defn}\label{def:theories}
  Let $\Th_{\mathrm{Hen},0}$ be the theory of Henselian valued fields
  of residue characteristic~$0$, that is the theory generated by the
  axioms stating that a discrete valuation field is Henselian (i.e.,
  the valuation ring satisfies the conclusions of Hensel's Lemma), and
  that its residue characteristic is not equal to $p$, for every
  rational prime~$p$.  Furthermore, given a field $K$ of
  characteristic~$0$, we consider also the Denef--Pas language
  enriched by constant symbols $c_a$ of valued field sort for all $a
  \in K$ and $\Th_{\mathrm{Hen},K,0}$ denotes the theory of Henselian
  valued fields of residue characteristic~$0$ together with the
  statements $c_a \cdot c_b = c_{ab}$, $c_a+c_b = c_{a+b}$ for all
  $a,b\in K$; cf.\ Definition~\ref{exa:theory}.
\end{defn}

The theory $\Th_{\mathrm{Hen},0}$ admits partial elimination of
quantifiers; see~\cite[Theorem~4.1]{Pas} or~\cite{HK}.  By a standard
argument, the same holds for the theory $\Th_{\mathrm{Hen},K,0}$ in
every extended language, as discussed above. \cv{We record this fact
  as follows.}

\begin{thm} \label{thm:Pas} Every Denef--Pas formula $\phi$ is
  $\Th_{\mathrm{Hen},0}$-equivalent to a formula $\psi$ without valued
  field quantifiers.  For every field $K$ of characteristic~$0$, the
  analogous statement holds true for $\Th_{\mathrm{Hen},K,0}$.
\end{thm}

Consider a number field $L$.  None of the local fields $L_\fq$, where
$\fq$ ranges over the primes of $O_L$, is a model
for~$\Th_{\mathrm{Hen},0}$.  Nevertheless, we will use theorems proved
in $\Th_{\mathrm{Hen},K,0}$ in the following way.  Suppose that
$\mathcal{X}$ and $\mathcal{Y}$ are two definable sets, and that
$\mathcal{X} = \mathcal{Y}$ holds in~$\Th_{\mathrm{Hen},K,0}$.  Then
the equivalence of $\mathcal{X}$ and $\mathcal{Y}$ can be proved by
only finitely many axioms of~$\Th_{\mathrm{Hen},K,0}$.  Hence,
$\mathcal{X}(E) = \mathcal{Y}(E)$ is true is true for $K \subset E$,
assuming only that the valued field $E$ is Henselian and that the
characteristic of the residue field of $E$ is greater than some
constant (depending on $\mathcal{X}$ and $\mathcal{Y}$).  In
particular, we deduce that $\mathcal{X}(L_\fq) = \mathcal{Y}(L_\fq)$
for almost all primes $\fq$ of~$O_L$.

%%%%%

\section{Parametrizing Representations} \label{sec:parameterize} In
this section, we consider an affine group scheme $\mathbf{G} \subset
\GL_N$ over the ring of integers $O_K$ of a number field $K$ whose
generic fiber is semi-simple.  Let $\mathfrak{g} \subset
\mathfrak{gl}_N$ be the Lie algebra of~$\mathbf{G}$.  We consider
$\mathbf{G}$ and $\mathfrak{g}$ as quantifier-free definable sets over
the first-order language of valued fields enriched by adding constant
symbols for the elements of~$K$, and work in the theory
$\Th_{\mathrm{Hen},K,0}$ of Henselian fields over~$K$ with residue
characteristic~$0$.  Our aim is to prove the following result.

\begin{thm} \label{thm:int.approx} Let $K$ be a number field with ring
  of integers~$O_K$, and let $\mathbf{G} \subset \GL_N$ be an affine
  group scheme over $O_K$ whose generic fiber is semi-simple. There
  are a $\left( \dim \mathbf{G}+1 \right)$-dimensional quantifier-free
  definable set $\scZ \subset \mathcal{O}^{\dim \mathbf{G}+1}$,
  quantifier-free definable functions $f_1,f_2 \colon \scZ \to
  \mathsf{\Gamma}$, and a constant $C\in\R$ such that, for every
  finite field extension $K \subset L$ and almost all primes~$\fq$
  of~$O_L$,
  \[
  \zeta_{\mathbf{G}(O_{L,\fq })}(s) - \zeta_{\mathbf{G}(O_L/\fq)}(s)
  \sim_C \int_{\scZ(L_\fq )} \lvert O_L/\fq \rvert^{f_1(z) - f_2(z) s}
  d\lambda(z),
  \]
  where $\lambda$ is the additive Haar measure on $L_{\fq}^{\dim
    \mathbf{G}+1}$ normalized so that $\lambda(O_{L,\fq}^{\dim
    \mathbf{G}+1}) = 1$.
\end{thm}

Throughout the proof of Theorem~\ref{thm:int.approx}, there are places
where we omit finitely many primes $\fq$ of~$O_L$.
Observation~\ref{obs:almost-all-q} collects most of the restrictions
that we impose.  In addition, we exclude in the proof of
Proposition~\ref{prop:stab.conn.components} and all consequences
thereof finitely many primes that are not specified explicitly; this
is due to partial elimination of quantifiers. In the actual proof of Theorem~\ref{thm:int.approx} in
Section~\ref{subsec:proof.int.approx}, an application of
Proposition~\ref{prop:grass.operations} also requires us to disregard
finitely many primes. Throughout we collectively write ``for almost
all primes'' to refer to these restrictions.  The choice of primes we
omit may depend on~$L$.  However, we emphasize that the definable set
$\scZ$ and the definable functions $f_1,f_2$ in
Theorem~\ref{thm:int.approx} do not depend on the choice of omitted
primes.

\subsection{Relative Orbit Method}\label{subsec:rel.orbit.method}

We continue to use the notation set up to formulate
Theorem~\ref{thm:int.approx}.

\begin{defn} \label{defn:X} Let $\scX$ be the quantifier-free definable
  set $\mathfrak{g}_\mathcal{O}\times(\mathcal{O}\smallsetminus\{0\})$.
\end{defn}

Throughout we fix a non-degenerate, invariant and
$\Ad(\mathbf{G})$-invariant bilinear form $\langle \cdot, \cdot
\rangle$ on~$\mathfrak{g}$, e.g., the Killing form, and we consider
finite extensions~$K \subset L$.  For $\fq \in \Spec(O_L)$ lying above
a rational prime~$p$, we denote by $\mathfrak{g}(O_{L,\fq })^\vee$ the
Pontryagin dual of the abelian pro-$p$ group $\mathfrak{g}(O_{L,\fq
})$.

\begin{obs} \label{obs:almost-all-q} By omitting finitely many primes
  of $O_L$, we may concentrate in the proof of
  Theorem~\ref{thm:int.approx} on $\fq \in \Spec(O_L)$ such that the
  following conditions hold.
  \begin{list}{$\circ$}{\setlength{\leftmargin}{\mylistleftmargin}
      \setlength{\labelwidth}{10pt} \setlength{\itemsep}{0pt}
      \setlength{\parsep}{1pt}}
  \item $p\Z = \fq \cap \Z$ is unramified in $\Q \subset L$, i.e., the
    valuation of the integer $p$ in $L_ \mathfrak{q}$ is 1;
    furthermore, $p > N+1$, and $p \nmid \lvert
    H^2(\mathbf{H}(\F_q),\C^\times) \rvert$ for every finite field
    $\F_q$ of characteristic $p$ and every connected reductive
    $\F_q$-algebraic group $\mathbf{H}$ with $\dim \mathbf{H} \leq
    \dim \mathbf{G}$; see Lemma~\ref{lem:bdd.Schur}.
  \item The form $\langle \cdot , \cdot \rangle$ is non-degenerate
    on~$\mathfrak{g}(O_{L,\fq})$.
  \item There is a surjective map $\Pi_{\fq} \colon \scX(L_\fq ) \to
    \mathfrak{g}(O_{L,\fq })^\vee$, taking the pair $x = (A_x,z_x)$ to
    the homomorphism of abelian groups
    \[
    \Pi_\fq(x) \colon \mathfrak{g}(O_{L,\fq}) \rightarrow \C^\times,
    \quad B \mapsto \exp\left(2\pi i\cdot \textup{Tr}_{L_\fq \vert
        \Q_p} \left(\frac{\langle A_x, B \rangle}{z_x}\right)\right).
    \]
  \item Every pro-nilpotent Lie subring of $\mathfrak{g}(O_{L,\fq})$
    containing the $1$st principal congruence Lie sublattice
    $\mathfrak{g}^{(1)}(O_{L,\fq})$ \cv{yields} a pro-$p$ subgroup of
    $\mathbf{G}(O_{L,\fq})$ via the exponential map, and the Kirillov
    orbit method applies to pro-$p$ subgroups of
    $\mathbf{G}(O_{L,\fq})$ as described in
    Proposition~\ref{prop:orbit.method}.  (For instance, $p > [L:\Q]
    N^2$ suffices.)
  \end{list}
\end{obs}

In particular, by restricting one of the homomorphisms $\Pi_\fq(x)$ to
$\mathfrak{g}^{(1)}(O_{L,\fq }) = \fq \cdot \mathfrak{g}(O_{L,\fq})$
and applying the orbit method map $\Omega$, we get an irreducible
character $\Xi_\fq(x) = \Omega ( \Pi_\fq(x)
\vert_{\mathfrak{g}^{(1)}(O_{L,\fq})})$ of the $1$st principal
congruence subgroup~$\mathbf{G}^{(1)}(O_{L,\fq })$.  When the prime
$\fq$ is clear from the context, it may be dropped from the notation.

For every finite extension $K \subset L$ and every $\fq \in
\Spec(O_L)$, the set $\scX(L_\fq)$ is an open subset of $L_\fq^{\dim
  \mathbf{G}+1}$.  We normalize the additive Haar measure on $L_\fq$
so that the ring of integers $O_{L,\fq}$ has measure~$1$, and denote
by $\lambda$ the product measure on $L_\fq ^{\dim \mathbf{G}+1}$.
In~\cite[Lemma~4.1 and Corollary~4.6]{Jai}, Jaikin--Zapirain proved
the following result.

\begin{thm}\label{thm:Jaikin} 
  There exist quantifier-free definable functions $\phi_1,\phi_2 \colon \scX
  \rightarrow \mathsf{\Gamma}$ such that, for every finite extension
  $K \subset L$, almost all $\fq \in \Spec(O_L)$, and every $x\in
  \scX(L_\fq)$,
  \begin{list}{}{\setlength{\leftmargin}{\myenumilistleftmargin}
      \setlength{\labelwidth}{20pt} \setlength{\itemsep}{0pt}
      \setlength{\parsep}{1pt}}
  \item[\textup{(1)}] $\lambda(\Pi_\fq^{-1}(\Pi_\fq(x))) = \lvert
    O_L/\fq \rvert^{\phi_1(x)}$,
  \item[\textup{(2)}] $\dim \Xi_\fq(x) = \lvert
    \Ad^*(\mathbf{G}^{(1)}(O_{L,\fq})) (\Pi_\fq (x)) \rvert^{1/2} =
    \lvert O_L/\fq \rvert^{\phi_2(x)}$.
  \end{list}
\end{thm}

\begin{rem}
  More explicitly, one can take $\phi_1(x) = \dim \mathbf{G} \,
  \val(z_x)$, and $\phi_2(x) = \frac{1}{2} \val(\alpha(x))$, where
  $\alpha$ is essentially the function appearing
  in~\cite[Corollary~4.6]{Jai}.  According to its definition
  in~\cite{Jai}, the function~$\alpha$, and hence $\phi_2$, is
  quantifier-free.
\end{rem}

We need a generalization of the construction leading to
Theorem~\ref{thm:Jaikin}.  Employing the notation introduced in
Sections \ref{subsec:def.fun} and~\ref{sec:valued.fields} (compare, in particular, Definition~\ref{def:grassmannian}), the
definable set  $\Grass(\mathfrak{g}_\mathsf{k})$ is the
Grassmannian of Lie subalgebras of~$\mathfrak{g}_\mathsf{k}$, and the
definable set $\Grass^\mathrm{nilp}(\mathfrak{g}_\mathsf{k})$ is the
subset of $\Grass(\mathfrak{g}_\mathsf{k})$ parametrizing Lie
subalgebras that consist of nilpotent matrices.  By applying
Proposition~\ref{prop:grass.operations}, part (1), we see that both
sets are, in fact, quantifier-free definable with respect to the
theory~$\Th_{\mathrm{fields},O_K}$.

Suppose that $\mathcal{R} \colon \scX \to
\Grass^\mathrm{nilp}(\mathfrak{g}_\mathsf{k})$ is a definable
function.  We denote by $\widetilde{\mathcal{R}} \subset \scX \times
\mathfrak{g}_\mathcal{O}$ the definable set of tuples $(x,X)$ such
that the reduction of $X$ to $\mathfrak{g}_\mathsf{k}$ is in
$\mathcal{R}(x)$.  Recall that~$\mathbf{G} \subset \GL_N$.  By
Observation~\ref{obs:almost-all-q}, we may assume that the residue
field characteristic $p$ satisfies $p > N$ so that, over the residue
field, one can evaluate without problems the logarithm series up to
its $N$th term.  We define $\exp \mathcal{R} \subset \scX \times
\mathbf{G}_\mathsf{k}$ to be the definable set of pairs $(x,g)$ such
that $g$ is unipotent and $\log g \in \mathcal{R}(x)$.  By
Observation~\ref{obs:almost-all-q}, we may assume that the residue
field characteristic $p$ is unramified in $L$ and satisfies~$p-1 > N$.
In this case, a result of Lazard implies that $\log(g)$ converges for
elements $g$ of any pro-$p$ subgroup of $\mathbf{G}(O_{L,\fq})$,
cf.~\cite[Lemma~B.1]{K}.  Denoting by $\widetilde{\mathcal{R}}_x$ the
fiber of $\widetilde{\mathcal{R}}$ at $x$, we define
$\exp\widetilde{\mathcal{R}} \subset \scX \times
\mathbf{G}_\mathcal{O}$ to consist of all pairs $(x,g)$ such that the
reduction of $g$ to $\mathbf{G}_\mathsf{k}$ is unipotent and $\log(g)
\in \widetilde{\mathcal{R}}_x$.  More generally, if $\mathcal{S}
\subset \scX \times \mathbf{G}_\mathsf{k}$ is a definable family over
$\scX$, let $\widetilde{\mathcal{S}} \subset \scX \times
\mathbf{G}_\mathcal{O}$ be the definable set of all pairs $(x,g)$ such
that the reduction of $g$ to $\mathbf{G}_\mathsf{k}$ lies in
$\mathcal{S}_x$. We observe that $\exp \widetilde{\mathcal{R}} =
\widetilde{\exp \mathcal{R}}$ and will use the lighter notation.

By Observation~\ref{obs:almost-all-q}, for every finite extension $K
\subset L$, for almost all primes $\fq$ of $O_L$, and for every $x \in
\scX(L_\fq)$, the additive group $\widetilde{\mathcal{R}}_x(L_\fq)$ is
closed under Lie commutators and it is the Lie ring associated to the
pro-$p$ group $\exp\widetilde{\mathcal{R}}_x(L_\fq)$.

\begin{defn}\label{def:Pi.Xi}
  Denote by $\Pi_{\mathcal{R},\fq}(x)$ the restriction of
  $\Pi_{\fq}(x)$ to $\widetilde{\mathcal{R}}_x(L_\fq)$.  For almost
  all primes~$\fq$, the orbit method map applied to
  $\Pi_{\mathcal{R},\fq}(x)$ yields an irreducible character
  $\Xi_{\mathcal{R},\fq}(x)$ of the group
  $\exp\widetilde{\mathcal{R}}_x(L_\fq)$.  When the prime $\fq$ is
  clear from the context, it may be dropped from the notation.
\end{defn}

Note that, if $\mathcal{R} \colon \scX \to
\Grass^\mathrm{nilp}(\mathfrak{g}_\mathsf{k})$ is the constant
function with common value~$\{0\}$, then the irreducible character
$\Xi_{\left\{ 0 \right\} ,\fq}(x)$ coincides with the previously
defined~$\Xi_{\fq}(x)$.  For general~$\mathcal{R}$, we are interested
in possible extensions of the character $\Xi_{\mathcal{R},\fq}(x)$ to
its stabilizer in the normalizer $N_{\mathbf{G}(O_{L,\fq})}(\exp
\widetilde{\mathcal{R}}_x(L_\fq))$.  This stabilizer is known as the
inertia group of~$\Xi_{\mathcal{R},\fq}(x)$.

We summarize the described set-up for $x \in \scX (L_\fq)$ in the
following diagram.
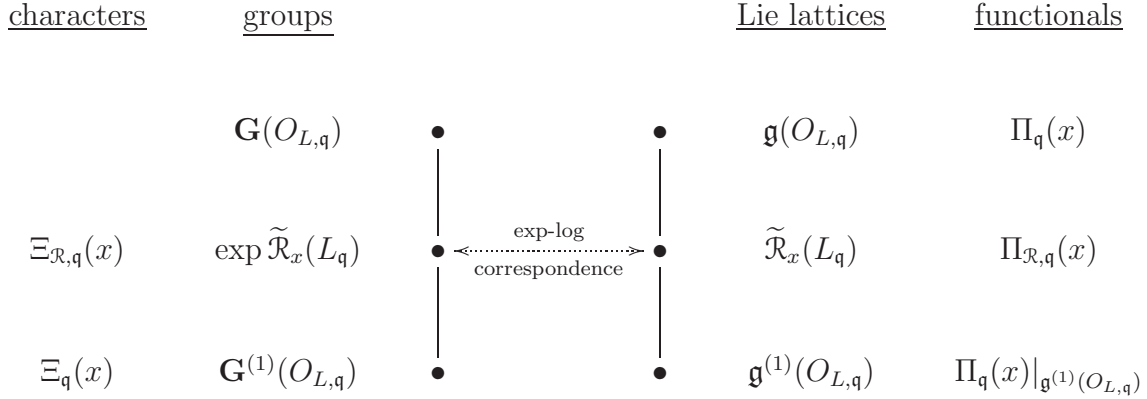
\[
\xymatrix@C-5pt{ {\text{\underline{characters}}} &
  {\text{\underline{groups}}} & & & & &
  {\text{\underline{Lie lattices}}} & {\text{\underline{functionals}}} \\
  {} & {\mathbf{G}(O_{L,\fq})} & {\bullet} \ar@{-}[d] & & & {\bullet}
  \ar@{-}[d] & {\mathfrak{g}(O_{L,\fq})} &
  {\Pi_{\fq}(x)} \\
  {\Xi_{\mathcal{R},\fq}(x)} & {\exp \widetilde{\mathcal{R}}_x(L_\fq)}
  & {\bullet} \ar@{<.>}[rrr]^{\text{exp-log}}_{\text{correspondence}}
  & & & {\bullet} & {\widetilde{\mathcal{R}}_x(L_\fq)} &
  {\Pi_{\mathcal{R},\fq}(x)} \\
  {\Xi_\fq(x)} & {\mathbf{G}^{(1)}(O_{L,\fq})} & {\bullet}
  \ar@{-}[u] & & & {\bullet} \ar@{-}[u] &
  {\mathfrak{g}^{(1)}(O_{L,\fq})} & {\Pi_{\fq}(x)
    \vert_{\mathfrak{g}^{(1)}(O_{L,\fq})}} }
\]

%%%

\subsection{The Stabilizer of $\Xi_\mathcal{R} $}
\label{subsec:the.stabilizer}

Suppose now that $\mathcal{R} \colon \scX \rightarrow
\Grass^\mathrm{nilp}(\mathfrak{g}_\mathsf{k})$ is a
\emph{quantifier-free} definable function with respect to the
theory~$\Th_{\mathrm{Hen},K,0}$.  This means that $\mathcal{R}$ can be
described by a quantifier-free definable formula for all models
$L_\fq$ with sufficiently large residue field characteristic.  We
remark that many of the results in this section do not yet depend
essentially on $\mathcal{R}$ being quantifier-free, but it is
convenient to focus on this situation in preparation of
Section~\ref{subsec:Lie.stabilizer}, where the extra condition plays a
crucial role. By
Proposition~\ref{prop:grass.operations}, parts (3) and~(5), there are
quantifier-free definable sets $\mathcal{N}_{\mathcal{R},\text{Lie}}
\subset \scX \times \mathfrak{g}_\mathsf{k}$ and
$\mathcal{N}_\mathcal{R} \subset \scX \times \mathbf{G}_\mathsf{k}$
whose fibers over a point $x \in \scX(L_\fq)$ are the stabilizers of
$\mathcal{R}(x)$ under the adjoint actions in the Lie algebra and in
the group respectively.

By Proposition~\ref{prop:orbit.method}, part (3), the stabilizer of
$\Xi_{\mathcal{R},\fq}(x)$ in $\mathbf{G}(O_{L,\fq})$ is
equal to the stabilizer in $\mathbf{G}(O_{L,\fq})$ of the
$\exp \widetilde{\mathcal{R}}_x(L_\fq )$-orbit of
$\Pi_{\mathcal{R},\fq}(x)$.  Hence, this stabilizer is the
product of $\exp\widetilde{\mathcal{R}}_x(L_\fq)$ and the
stabilizer of $\Pi_{\mathcal{R},\fq}(x)$ in
$\mathbf{G}(O_{L,\fq})$.  Writing
\[
N_{\mathbf{G}(O_{L,\fq})}
(\widetilde{\mathcal{R}}_x(L_\fq)) = \{g\in
\mathbf{G}(O_{L,\fq}) \mid \Ad(g) (
\widetilde{\mathcal{R}}_x(L_\fq) )
=\widetilde{\mathcal{R}}_x(L_\fq)\},
\]
we have
\begin{multline*}
  \Stab_{\mathbf{G}(O_{L,\fq})} (
  \Pi_{\mathcal{R},\fq}(x) ) = \\
  \left\{g\in N_{\mathbf{G}(O_{L,\fq})}
    (\widetilde{\mathcal{R}}_x(L_\fq)) \mid \forall Y \in
    \widetilde{\mathcal{R}}_x(L_\fq) \colon
    \Pi_{\mathcal{R},\fq}(x)(Y) = \Pi_{\mathcal{R},\fq}(x)(\Ad(g)Y)
  \right\}.
\end{multline*}

In the following we consider a prime $\fq$ of $O_L$, lying above a
prime $\fp$ of $O_K$ and different from the previously-omitted primes;
see Observation~\ref{obs:almost-all-q}.  In addition we fix an element
$x = (A_x,z_x) \in \scX(L_\fq)$. Recall that $\red_{\mathcal{O} \vert \mathsf{k}}:\mathcal{O} \to \mathsf{k}$ denotes the reduction map (applied component-wise to entries of a matrix or vector).

Let
$\widehat{\mathcal{S}} = \widehat{\mathcal{S}}_{\mathcal{R},\fq,x}
\subset \mathbf{G}_\mathcal{O}$ be the definable group, in
$\Th_{\mathrm{Hen},K,0}$, given by the formula
\begin{equation}\label{equ:phi(x,g)}
\phi(x,g) =_\text{def} \, \red_{\mathcal{O} \vert \mathsf{k}}(g) \in
(\mathcal{N}_\mathcal{R})_x \, \wedge \, \left( \forall Z\in
  \widetilde{\mathcal{R}}_x \big( \val\left(\langle A_x^g-A_x,Z
      \rangle\right) > \val(z_x) \big) \right),
\end{equation}
where we think of $x$ as a parameter and $g$ as a ``free
variable''.
% We can choose a finite set $\{Z_1,\ldots,Z_n\} \subset
% \mathfrak{g}(O_{K,\fp})$ such that, for every extension $K_\fp
% \subset L_\fq$, the Lie ring $\widetilde{\mathcal{R}}_x(L_\fq)$ is
% the $O_{L,\fq}$-module freely generated by the $Z_1, \ldots, Z_n$.
% In particular, we see in this way that $\phi_{\fp,x}$ is equivalent
% to a quantifier-free formula.

We denote by $\mathcal{S} = \mathcal{S}_{\mathcal{R},\fq,x}$ the
reduction of $\widehat{\mathcal{S}}$ modulo the maximal ideal, i.e.,
the definable subgroup of $\mathbf{G}_\mathsf{k}$, in
$\Th_{\mathrm{Hen},K,0}$, given by the formula
\begin{equation} \label{equ:psi(x,y)}
\psi(x,h) =_\text{def} \, \exists g \in \mathbf{G}_\mathcal{O} \left(
  \red_{\mathcal{O} \vert \mathsf{k}}(g) = h \, \wedge \,
  \phi(x,g) \right),
\end{equation}
where again we think of $x$ as a parameter and $h$ as a ``free
variable''.

Denote the characteristic of the residue field $O_K/ \fp$ by~$p$.  The
functor of $p$-Witt vectors
\[
\F \mapsto \Witt(\F) = \varprojlim \Witt_n(\F)
\]
associates to every perfect field $\F$ of characteristic~$p$
canonically a strict $p$-ring $\Witt(\F)$ with residue field~$\F$;
see~\cite[Chapter~II]{Se}.  The integral domain $\Witt(\F)$ is
complete and Hausdorff with respect to the $p$-adic topology and
$\Witt_n(\F) \cong \Witt(\F) / p^n\Witt(\F)$ is called the ring of
truncated Witt vectors of length~$n$.  For short we denote by
$\FWitt(\F)$ the field of fractions of~$\Witt(\F)$.

The functor $\Witt$ is pro-representable; the underlying set is
represented by $\prod_{i=1}^\infty \mathbb{A}^1$;
\nir{cf.~\cite{Gr}}.  Using this fact, one sees that there are a
pro-algebraic group scheme $\widehat{\mathbf{S}} =
\widehat{\mathbf{S}}_{\mathcal{R},\fq,x}$ and a definable group
$\overline{\mathcal{S}} = \overline{\mathcal{S}}_{\mathcal{R},\fq,x}$,
in $\Th_{\text{perf.-fields},p,O_L/\fq}$, such that, for every
(possibly infinite) perfect extension $\F$ of~$O_L / \fq$,
\begin{equation}\label{equ:pro-algebraic-scheme}
  \begin{split}
    \widehat{\mathbf{S}}(\F) & = \widehat{\mathcal{S}}(\FWitt(\F))
    \subset
    \mathbf{G}_\mathcal{O}(\FWitt(\F)), \\
    \overline{\mathcal{S}}(\F) & = \mathcal{S}(\FWitt(\F)) \subset
    \mathbf{G}_\mathsf{k}(\FWitt(\F)) \cong \mathbf{G}(\F).
  \end{split}
\end{equation}

Our next goal, Proposition~\ref{prop:stabilizer.is.algebraic}, is to
show that there exists an algebraic group $\mathbf{S}$ over $O_L/\fq$
such that, for every perfect extension $\F$ of~$O_L/\fq$, one has
$\mathbf{S}(\F) = \overline{\mathcal{S}}(\F)$.  The following table
summarizes some of the notation that will feature in this discussion.

\[ \xymatrix@C-10pt@R-20pt{
  {\text{\underline{$\Th_{\mathrm{Hen},K,0}$-def.}}} &
  {\text{\underline{$\Th_{\text{perf.-fields},p,O_L/\fq}$-def.}}}
  & {\text{\underline{$\mathrm{Witt}_n(\mathbb{F})$-algebraic}}} &
  {\text{\underline{$\mathbb{F}$-(pro-)algebraic}}} \\
  {\widehat{\mathcal{S}} \subset \mathbf{G}_\mathcal{O}} & &
  {\mathbf{S}_n} & {\widehat{\mathbf{S}} = \varprojlim
    \mathcal{F}_n(\mathbf{S}_n)} \\
  {\mathcal{S} \subset \mathbf{G}_\mathsf{k}} &
  {\overline{\mathcal{S}}} & & {\mathbf{S}} }
\]

We briefly recall further details regarding the functor Witt;
compare~\cite{Gr}, or~\cite[p.~276]{BoLuRa} for a summary.  As above,
let $\F$ be a perfect field of characteristic~$p > 0$ and fix $n \in
\N$.  For an $\F$-scheme $\mathbf{X}$, let $\mathcal{G}_n(\mathbf{X})$
be the locally ringed space whose underlying topological space is the
same as the topological space of $\mathbf{X}$ and whose sheaf of rings
is the sheaf of germs of morphisms $\mathbf{X} \rightarrow \Witt_n$.
For example,
\[
\mathcal{G}_n(\Spec(\F)) = \Spec(\Witt_n(\F)) \quad \text{and} \quad
\mathcal{G}_n(\Spec(\F[e]/ (e^2))) = \Spec(\Witt_n(\F)[e] / (e^2)),
\]
which we will shortly put to good use.  By the main theorem
of~\cite[\S4]{Gr}, there is a functor
\[
\mathcal{F}_n \colon (\text{$\Witt_n(\F)$-schemes of finite type})
\rightarrow (\text{$\F$-schemes of finite type}),
\]
the Greenberg functor of degree~$n$, such that, for every $\F$-scheme
$\mathbf{X}$ of finite type and every $\Witt_n(\F)$-scheme
$\mathbf{Y}$ of finite type, there is a natural bijection
\begin{equation} \label{eq:Greenberg.adjoint} \Hom_{\Spec(\F)}
  (\mathbf{X}, \mathcal{F}_n(\mathbf{Y}) ) \cong
  \Hom_{\Spec(\Witt_n(\F))} (\mathcal{G}_n(\mathbf{X}),\mathbf{Y}).
\end{equation}

Returning to the situation at hand, by
Observation~\ref{obs:almost-all-q}, we may suppose that $p\Z = \fq
\cap \Z$ is unramified in $\mathbb{Q} \subset L$.  For $n \in \N$,
there is an affine group scheme $\mathbf{S}_n$ over $\Witt_n(O_L/\fq)$
such that, for every perfect extension $\F$ of $O_L/ \fq$,
\begin{multline*}
  \mathbf{S}_n (\Witt_n(\F)) = \big\{ g \in \mathbf{G}(\Witt_n(\F))
  \mid \red_{\Witt_n(\F) \vert \F} (g) \in
  (\mathcal{N}_\mathcal{R})_x(\F) \, \wedge \\
  \forall Z\in \widetilde{\mathcal{R}}_x(\FWitt(\F)) \big( \val
  (\langle A_x^g-A_x,Z \rangle ) > \min \{ n,\val(z_x) \} \big)
  \big\}.
\end{multline*}
Here $\red_{\Witt_n(\F) \vert \F} \colon \Witt_n(\F) \rightarrow
\Witt_1(\F) \cong \F$ denotes the natural reduction map.  The
pro-algebraic group scheme $\widehat{\mathbf{S}} =
\widehat{\mathbf{S}}_{\mathcal{R},\fq,x}$ in
\eqref{equ:pro-algebraic-scheme} is the inverse limit of the $O_L/
\fq$-group schemes $\mathcal{F}_n(\mathbf{S}_n)$, $n \in \N$.

Next we discuss the pro-Lie algebra $\widehat{\mathbf{T}}$ of
$\widehat{\mathbf{S}}$.  We use the following consequence
of~\eqref{eq:Greenberg.adjoint}.  If $\mathbf{V} =
\Spec(\Witt_n(\F)[x_1,\ldots,x_d] / (f_1,\ldots,f_m))$ is an affine
$\Witt_n(\F)$-scheme and $v \colon \Spec(\Witt_n(\F)) \rightarrow
\mathbf{V}$ a $\Witt_n(\F)$-point, then the tangent space of
$\mathcal{F}_n(\mathbf{V})$ at $\mathcal{F}_n(v)$ is the affine
subspace of $\mathcal{F}_n(\mathbb{A}_{\Witt_n(\F)}^d) \cong
\mathbb{A}_{\F}^{nd}$ defined by the ``polynomials''
$\mathcal{F}_n(\mathrm{d} f_i(v))$, $i \in \{1,\ldots,m\}$.  This can
be seen from applying \eqref{eq:Greenberg.adjoint} to $\mathbf{X} =
\Spec(\F[e]/(e^2))$ and $\mathbf{Y} = \mathbf{V}$.  In particular, the
Lie algebra of $\mathcal{F}_n(\mathbf{S}_n)$ is isomorphic to
$\mathcal{F}_n(\mathbf{T}_n)$, where $\mathbf{T}_n$ denotes the
$\Witt_n(O_L / \fq)$-scheme satisfying, for every perfect extension
$\F$ of $O_L/ \fq$,
\begin{multline*}
  \mathbf{T}_n(\Witt_n(\F)) = \big\{ X \in \mathfrak{g}(\Witt_n(\F))
  \mid \red_{\Witt_n(\F) \vert \F}(X) \in
  (\mathcal{N}_{\mathcal{R},\text{Lie}})_x (\F) \, \wedge \\
  \forall Z \in \widetilde{\mathcal{R}}_x(\FWitt(\F)) \big( \val (
  \langle [A_x,X],Z \rangle) > \min \{ n, \val(z_x) \} \big) \big\}.
\end{multline*}
The pro-Lie algebra $\widehat{\mathbf{T}}$ is the inverse limit of
the $O_L/\fq$-Lie algebra schemes $\mathcal{F}_n(\mathbf{T}_n)$, $n
\in \N$.

Next we show that the definable group $\overline{\mathcal{S}}$ in
\eqref{equ:pro-algebraic-scheme} is actually an algebraic group.

\begin{prop} \label{prop:stabilizer.is.algebraic} Let $\fq \in
  \Spec(O_L)$, $x \in \scX(L_\fq)$ and $\overline{\mathcal{S}} =
  \overline{\mathcal{S}}_{\mathcal{R},\fq,x}$ be as above.  In
  particular, suppose that $\fq$ satisfies the conditions listed in
  Observation~\textup{\ref{obs:almost-all-q}}.  Then
  $\overline{\mathcal{S}}$ is (equivalent to) an algebraic group
  over~$O_L / \fq$.
\end{prop}

\begin{proof} Recall that $\mathbf{G} \subset \GL_N$.  We show that
  $\overline{\mathcal{S}}$ is quantifier-free in
  $\Th_{\text{perf.-fields},p,O_L/\fq}$, using
  Remark~\ref{rem:criterion.q.f.perfect}.  Since constructable
  subgroups are Zariski-closed, this implies that
  $\overline{\mathcal{S}}$ is (equivalent to) a Zariski-closed
  subgroup of the algebraic group~$\GL_N$ over $O_L / \fq$.

  We need to check that
  \begin{equation}\label{eq:condition.q.f.stabilizer}
    \overline{\mathcal{S}}(\F) = \overline{\mathcal{S}}(\F^\alg) \cap
    \mathbf{G}(\F)
  \end{equation}
  for every (possibly infinite) perfect extension $\F$ of $O_L / \fq$.
  Fix such an extension~$\F$, set $\mathbb{O} = \Witt(\F)$ with
  residue field~$\mathbb{O} /p \mathbb{O} \cong \F$, and write
  $\mathbb{L} = \FWitt(\F)$.  Let $\mathbb{O}^\unr = \Witt(\F^\alg)$
  and $\mathbb{L}^\unr = \FWitt(\F^\alg)$ denote the maximal
  unramified extensions.  Since $L_\fq$ is unramified over~$\Q_p$, we
  have $L_\fq \subset \mathbb{L}$.

  The inclusion $\subset$ in \eqref{eq:condition.q.f.stabilizer} is
  clear.  To prove the other inclusion we consider $A \in
  \mathfrak{g}(\mathbb{O})$, $\overline{g}\in \mathbf{G}(\F)$, and
  $z\in \mathbb{O} \smallsetminus \left\{ 0 \right\}$.  Writing $x =
  (A,z) \in \scX(\mathbb{L})$ and $\gamma=\val(z)$, we suppose that
  there exists $\widetilde{g} \in \mathbf{G}(\mathbb{O}^\unr)$ such
  that $\red_{\mathcal{O} \vert \mathsf{k}}(\widetilde{g}) =
  \overline{g}$ and $\val( \langle A^{\widetilde{g}}-A,X \rangle )
  \geq \gamma+1$ for all $X \in
  \widetilde{\mathcal{R}}_x(\mathbb{L}^\unr)$.  The task is to produce
  $g \in \mathbf{G}(\mathbb{O})$ with the same properties as
  $\widetilde{g}$.  Clearly, there exists $g_1 \in
  \mathbf{G}(\mathbb{O})$ such that $\red_{\mathcal{O} \vert
    \mathsf{k}}(g_1) = \overline{g}$.  Put $B = A^{g_1} \in
  \mathfrak{g}(\mathbb{O})$ and consider the definable set
  \[
  \mathcal{Y} = \mathcal{Y}_{A,B,\mathcal{R}(x),\gamma} = \left\{g \in
    \mathbf{G}_\mathcal{O}^{(1)} \mid \forall Z \in
    \widetilde{\mathcal{R}}_x \left(\val( \langle B^g-A, Z \rangle )
      \geq \gamma+1 \right) \right\},
  \]
  where the labeling is permissible, because $\mathcal{R}(x)$
  determines $\widetilde{\mathcal{R}}_x$.  Clearly, $g_1^{-1}
  \widetilde{g} \in \mathcal{Y}(\mathbb{O}^\unr)$, and it suffices to
  show that $\mathcal{Y}(\mathbb{O})$ is not empty.  Furthermore, by
  forming the quotient of $\mathcal{Y}$ by the $(\gamma+1)$st
  principal congruence subgroup $\mathbf{G}_\mathcal{O}^{(\gamma+1)}$
  of $\mathbf{G}_\mathcal{O}$, we obtain a
  $\Witt_{\gamma+1}(O_L/\fq)$-scheme.  Using the Greenberg functor,
  this quotient can be identified with an algebraic variety
  $\mathbf{Y} = \mathbf{Y}_{A,B,\mathcal{R}(x),\gamma}$
  over~$O_L/\fq$.  Our aim $\mathcal{Y}(\mathbb{O}) \neq \varnothing$
  is equivalent to $\mathbf{Y}(\F) \neq \varnothing$.

  It is convenient to treat first the special case $\mathcal{R}(x) =
  \{ 0 \}$.  This means that $\widetilde{\mathcal{R}}_x$ gives the
  $1$st principal congruence Lie sublattice $\mathfrak{g}^{(1)}$.
  Since the form $\langle \cdot , \cdot \rangle$ is non-degenerate,
  the defining condition of $\mathcal{Y}$ is equivalent to $B^g
  \equiv A \pmod{p^\gamma}$.  By induction on $\gamma$, we may further
  assume that $B \equiv A \pmod{p^{\gamma-1}}$, that is $B = A +
  p^{\gamma-1}D$ for some $D \in \mathfrak{g}(\mathbb{O})$.  By
  Observation~\ref{obs:almost-all-q}, the logarithm map is a
  well-defined polynomial map on $\mathbf{G}^{(1)}(\mathbb{O}^\unr)$.
  For $g \in \mathbf{G}^{(1)}(\mathbb{O}^\unr)$, the formula
  \begin{equation} \label{equ:B^g-log-expansion} B^g = B +
    \sum_{i=1}^\infty \frac{1}{i!}
    [B,\underbrace{\log(g),\ldots,\log(g)}_{i}] \equiv B + [B,\log(g)]
    \pmod{p^{\delta +1}},
  \end{equation}
  where $\delta = \sup \{ d \in \N \mid [B,\log(g)] \equiv 0
  \pmod{p^d} \}$, shows that the defining condition of $\mathcal{Y}$
  can be replaced by
  \[
  [A,\log(g)] + p^{\gamma-1}D \equiv 0 \pmod{p^\gamma}.
  \]
  Passing to $\mathbf{Y}$, this translates into a system of linear
  equations
  over~$\F$.  Since $\mathbf{Y}(\F^\alg)$ is non-empty we deduce that
  $\mathbf{Y}(\F)$ is non-empty.

  Now we return to the general case.  Based on the trivial inclusion
  $\{0\} \subset \mathcal{R}(x)$, the special case yields $g_2 \in
  \mathbf{G}^{(1)}(\mathbb{O})$ such that $B^{g_2} \equiv A
  \pmod{p^\gamma}$.  Thus we may assume that $B$ itself is already of
  the form $B = A + p^\gamma E$ for some $E \in
  \mathfrak{g}(\mathbb{O})$.  By choosing a basis for the Lie lattice
  $\widetilde{\mathcal{R}}_x(\mathbb{L})$, the defining condition of
  $\mathcal{Y}$ can be phrased as $\ell(B^g - A) \equiv 0
  \pmod{p^{\gamma+1}}$ for a linear operator $\ell =
  \ell_{\mathcal{R}(x)} \colon \mathfrak{g} \rightarrow \mathfrak{g}$.
  The elementary divisors of $\ell$ are $1$ and $p$, with
  multiplicities $\dim \mathcal{R}(x)$ and $\dim
  \mathfrak{g}_\mathsf{k} - \dim \mathcal{R}(x)$.  Thus
  \eqref{equ:B^g-log-expansion} implies that the defining condition of
  $\mathcal{Y}$ can be replaced by
  \[
  \ell([A,\log(g)] + p^\gamma E) \equiv 0 \pmod{p^{\gamma+1}},
  \]
  because necessarily $[A,\log(g)] \equiv 0 \pmod{p^\gamma}$ and all
  higher terms vanish modulo $p^{\gamma+1}$.  Passing to $\mathbf{Y}$,
  this translates once more into a system of linear
  equations
  over~$\F$.  Since $\mathbf{Y}(\F^\alg)$ is non-empty we deduce that
  $\mathbf{Y}(\F)$ is non-empty.
\end{proof}

\begin{rem} \label{rem:stabilizer.is.alg}
  The proof of Proposition~\ref{prop:stabilizer.is.algebraic} admits
  the following short interpretation.  In the special case
  $\mathcal{R}(0) = \{0\}$, one can regard $\mathbf{Y} =
  \mathbf{Y}_{A,B,\mathcal{R}(x),\gamma}$ as a torsor of the connected
  unipotent group
  \[
  \mathcal{U}_{A,\gamma} = \left\{g\mathbf{G}_\mathcal{O} ^
    {(\gamma+1)}\in\mathbf{G}_\mathcal{O}
    ^{(1)}/\mathbf{G}_\mathcal{O} ^ {(\gamma+1)} \mid A^g-A \equiv 0
    \pmod{p^\gamma} \right\},
  \]
  and the logarithm map sets up a bijection between $\mathbf{Y}$ and
  the affine space
  \begin{equation*}
    \{ Z + \mathfrak{g}_\mathcal{O}^{(\gamma+1)} \in
    \mathfrak{g}_\mathcal{O}^{(1)} /
    \mathfrak{g}_\mathcal{O}^{(\gamma+1)} \mid [A,Z] \equiv 0
    \pmod{p^\gamma} \}.
  \end{equation*}
  It is known that connected unipotent groups have trivial first
  Galois cohomology groups, and thus every torsor over such a group
  has a rational point.
\end{rem}

We denote the algebraic group equivalent to $\overline{\mathcal{S}} =
\overline{\mathcal{S}}_{\mathcal{R},\fq,x}$ by $\mathbf{S} =
\mathbf{S}_{\mathcal{R},\fq,x}$ and refer to it as the stabilizer of
$\Xi_\mathcal{\mathcal{R},\fq}(x)$ modulo the $1$st principal
congruence subgroup.

\begin{prop}\label{prop:stab.conn.components} There is a constant
  $C \in \R$, depending only on $K$ and $\mathbf{G}$, such that, for
  every finite extension $K \subset L$, almost all primes $\fq$
  of~$O_L$, every quantifier-free definable function
  $\mathcal{R}\colon \scX \rightarrow
  \Grass^\mathrm{nilp}(\mathfrak{g}_\mathsf{k})$, and every $x =
  (A_x,z_x) \in \scX(L_\fq)$, the number of connected components of
  $\mathbf{S} = \mathbf{S}_{\mathcal{R},\fq,x}$ is less than~$C$.
\end{prop}

\begin{proof}
  Recall that the notation $\ac^{\times n}$, respectively
  $\val^{\times n}$, indicates that the angular component map,
  respectively valuation map, is to be applied coordinatewise to
  vectors of length $n$.  We apply partial elimination of quantifiers
  in the theory $\Th_{\mathrm{Hen},K,0}$ (cf.\ Theorem~\ref{thm:Pas})
  by treating the entries of $x$, which involve elements of $L_\fq$,
  as parameters: after omitting finitely many primes, the formula
  \eqref{equ:psi(x,y)} defining $\mathcal{S} =
  \mathcal{S}_{\mathcal{R},\fq,x}$ is equivalent to a formula of the
  form
  \begin{multline*}
    \eta (x,h) =_\text{def} \\ \bigvee_{i=1}^M \big( H_i(A_x) = 0 \,
    \wedge \, \phi_i(\ac^{\times m(i)}(H'_i(A_x)),\mathcal{R}(x),h) \,
    \wedge \, \psi_i(\val^{\times n(i)}(H''_i(A_x)),\val(z_x)) \big),
  \end{multline*}
  where the $H_i$, $H'_i = (H'_{i,1}, \ldots, H'_{i,m(i)})$ and $H''_i
  = (H''_{i,1}, \ldots, H''_{i,n(i)})$ are polynomial functions over
  $K$ (of sort~$\mathsf{F}$), the $\phi_i$ are formulae in the
  language of rings (of sort~$\mathsf{k}$), and the $\psi_i$ are
  formulae in the language of ordered groups (of
  sort~$\mathsf{\Gamma}$).  This can be proved by induction on the
  length of the formula.

  It follows that $\mathcal{S}$ is a finite union of some of the
  definable sets $\left\{ h \mid \phi_i(\ac^{\times
      m(i)}(H'_i(A_x)),\mathcal{R}(x),h) \right\}$.  Write
  $\overline{\mathcal{S}} =
  \overline{\mathcal{S}}_{\mathcal{R},\fq,x}$ and $\F_q = O_L / \fq$.
  Since $\overline{\mathcal{S}}(\F_q)$ always contains the identity,
  it is non-empty.  Proposition~\ref{prop:size.definable.set} implies
  that there is a constant $C \in \R$ such that, for every unramified
  finite extension $L_\fq \subset M_\mathfrak{r}$ with residue field
  $\F _{q^r}$, there exists $d \in \N_0$ such that $C^{-1} q^{rd} \leq
  \lvert \overline{\mathcal{S}}(M_\mathfrak{r}) \rvert = \lvert
  \overline{\mathcal{S}}(\F_{q^r}) \rvert \leq C q^{rd}$.

  Using Proposition~\ref{prop:stabilizer.is.algebraic}, we regard
  $\overline{\mathcal{S}}$ as an algebraic variety~$\mathbf{S}$ over
  $\mathbb{F}_q$ and apply the Lang--Weil bound~\cite{LW}.  There are
  infinitely many $r \in \N$ such that all absolutely irreducible
  components of $\mathbf{S}$ are defined over $\F_{q^r}$, and $\lvert
  \mathbf{S}(\F_{q^r}) \rvert = (\lvert \mathbf{S} : \mathbf{S}^\circ
  \rvert + O(q^{-r/2})) q^{r \dim \mathbf{S}}$ for such~$r$.
  Comparing the two estimates for $\lvert
  \overline{\mathcal{S}}(\F_{q^r}) \rvert = \lvert
  \mathbf{S}(\F_{q^r}) \rvert$ as $r$ tends to infinity, we obtain
  $\lvert \mathbf{S} : \mathbf{S}^\circ \rvert \leq C$.
\end{proof}

Next we identify the Lie algebra of $\mathbf{S} =
\mathbf{S}_{\mathcal{R},\fq,x}$, where $\mathcal{R}$ and $x =
(A_x,z_x) \in \scX(L_\fq)$ are as above.  In analogy to
\eqref{equ:phi(x,g)}, let $\widehat{\mathcal{T}} =
\widehat{\mathcal{T}}_{\mathcal{R},\fq,x} \subset
\mathfrak{g}_\mathcal{O}$ be the definable set given by
\[
\xi(x,X) =_\text{def} \, \red_{\mathcal{O} \vert \mathsf{k}}(X) \in
(\mathcal{N}_{\mathcal{R},\mathrm{Lie}})_x \, \wedge \, \left( \forall
  Z \in \widetilde{\mathcal{R}}_x \left( \val( \langle [A_x,X],Z
    \rangle) > \val(z_x) \right) \right),
\]
where we think of $x$ as a parameter and $X$ as a free variable.
Furthermore, let $\mathcal{T} = \mathcal{T}_{\mathcal{R},\fq,x}$
denote the reduction of $\widehat{\mathcal{T}}$ modulo the maximal
ideal, a definable subset of $\mathfrak{g}_\mathsf{k}$;
compare~\eqref{equ:psi(x,y)}.  Arguing similarly as in the proof of
Proposition~\ref{prop:stabilizer.is.algebraic}, we see that for every
finite extension $K \subset L$, almost all primes $\fq$ of~$O_L$, and
every $x \in \scX(L_\fq)$, the set $\mathcal{T}(L_\fq)$ is
a linear space over the residue field of~$L_\fq $.  In fact, we give
an independent proof of this fact in
Corollary~\ref{cor:q.f.Lie.algebras.of.stabilizers}.

\begin{prop} \label{prop:conn.stab} For every finite extension $K
  \subset L$, almost all primes $\fq$ of~$O_L$, and every $x \in
  \scX(L_\fq)$, the definable set $\mathcal{T}_{\mathcal{R},\fq,x}$ is
  (equivalent to) the Lie algebra of $\mathbf{S}_{\mathcal{R},\fq,x}$.
\end{prop}

\begin{proof} Write $x =(A_x,z_x)$ and $\gamma = \val(z_x)$.  We drop
  all subscripts $\mathcal{R},\fq,x$.  Recall the construction of the
  pro-algebraic group $\widehat{\mathbf{S}} = \varprojlim
  \mathcal{F}_n(\mathbf{S}_n)$ and its pro-Lie algebra
  $\widehat{\mathbf{T}} = \varprojlim \mathcal{F}_n(\mathbf{T}_n)$
  via Witt vectors.  The reduction map $\red_{\mathcal{O} \vert
    \mathsf{k}} \colon \widehat{\mathcal{S}} \rightarrow \mathcal{S}$
  translates into $\widehat{\mathbf{S}} \rightarrow \mathbf{S}$
  which factors through the homomorphism $f_\gamma \colon
  \mathcal{F}_\gamma (\mathbf{S}_\gamma) \rightarrow \mathbf{S}$ of
  $O_L/\fq$-algebraic groups.  By definition, this homomorphism is
  onto and therefore flat; see~\cite[Proposition~6.1.5]{EGA4_2}.  As
  indicated in Remark~\ref{rem:stabilizer.is.alg}, the fibers of
  $f_\gamma$, each isomorphic to the kernel of $f_\gamma$, are affine
  spaces, and hence smooth.  By~\cite[Theorem~17.5.1]{EGA4}, the map
  $f_\gamma$ is smooth, and so its differential at $1$ is surjective.
  It follows that $\mathcal{F}_\gamma(\mathbf{T}_\gamma)$ maps onto
  the Lie algebra $\mathbf{T}$ of $\mathbf{S}$.  In this way we can
  identify $\mathcal{T}$ with~$\mathbf{T}$.
\end{proof}

Using Proposition~\ref{prop:stab.conn.components} and
Lemma~\ref{lem:size.group}, we obtain the following consequence.

\begin{cor} \label{cor:size.stab.Lie} There is a constant $C \in
  \R$ such that, for every finite extension $K \subset L$, almost all
  primes $\fq$ of~$O_L$, and every $x \in \scX(L_\fq)$,
  \[
  C^{-1} \lvert \mathcal{T}_{\mathcal{R},\fq,x}(L_\fq ) \rvert \leq
  \lvert \mathcal{S}_{\mathcal{R},\fq,x}(L_\fq ) \rvert \leq C \lvert
  \mathcal{T}_{\mathcal{R},\fq,x}(L_\fq ) \rvert.
  \]
\end{cor}

%%%

\subsection{The Lie Algebra Associated to the Stabilizer of $\Xi_\mathcal{R}$}
\label{subsec:Lie.stabilizer}

We continue to work in the set-up introduced in
Sections~\ref{subsec:rel.orbit.method}
and~\ref{subsec:the.stabilizer}.  In particular, we consistently omit
finitely many primes, as specified by
Observation~\ref{obs:almost-all-q} and in the proof of
Proposition~\ref{prop:stab.conn.components}.  It is easier, and more
transparent, to handle the Lie algebra associated to the stabilizer
$\mathbf{S}_{\mathcal{R},\fq,x}$ of $\Xi_{\mathcal{R},\fq}(x)$ modulo
the $1$st principal congruence subgroup, rather than the stabilizer
itself.  For this purpose, we introduce the following cover of~$\scX$.

\begin{defn} \label{defn:Y} Let $\scY \subset \scX \times
  (\mathsf{\Gamma} \cup \{\infty\})^{\dim \mathfrak{g}} \times
  (\Aut(\mathfrak{g})_\mathcal{O})^2$ be the quantifier-free definable
  set consisting of tuples $((A,z),
  (\gamma_1,\ldots,\gamma_{\dim\mathfrak{g}}), (U_1,U_2))$ such that,
  in the chosen standard basis of~$\mathfrak{g}$, the linear operator
  $T = U_1 \, (\ad A) \, U_2$ is diagonal and the valuations of the
  diagonal elements are given by
  $(\gamma_1,\ldots,\gamma_{\dim\mathfrak{g}})$.
\end{defn}

Let $\mathcal{R} \colon \scY \to
\Grass^\mathrm{nilp}(\mathfrak{g}_\mathsf{k})$ be a definable
function.  The following definitions are analogous to those in
Section~\ref{subsec:rel.orbit.method}.  Let $\widetilde{\mathcal{R}}
\subset \scY \times \mathfrak{g}_\mathcal{O}$ be the definable set of
tuples $(y,X)$ such that the reduction of $X$ to
$\mathfrak{g}_\mathsf{k}$ is in~$\mathcal{R}(y)$.  Recall that
$\mathbf{G} \subset \GL_N$ and that, by virtue of
Observation~\ref{obs:almost-all-q}, the residue field
characteristic~$p$ satisfies $p>N$.  We define $\exp \mathcal{R}
\subset \scY \times \mathbf{G}_\mathsf{k}$ to be the set of pairs
$(y,g)$ such that $g$ is unipotent and $\log g \in \mathcal{R}(y)$.
Finally, again based on Observation~\ref{obs:almost-all-q}, we define
$\exp \widetilde{\mathcal{R}}\subset \scY\times\mathbf{G}_\mathcal{O}$
to consist of all pairs $(y,g)$ such that the reduction of $g$ to
$\mathbf{G}_\mathsf{k}$ is unipotent and $\log
g\in\widetilde{\mathcal{R}}_y$.

We denote by $\pr_{\scY \downarrow \scX} \colon \scY \rightarrow \scX$
the natural projection.  For every finite extension $K \subset L$,
almost all primes $\fq$ of~$O_L$, and for every $y \in \scY(L_\fq)$,
the additive group $\widetilde{\mathcal{R}}_y(L_{\fq})$ is closed
under Lie commutators and it is the Lie ring of the pro-$p$
group~$\exp \widetilde{\mathcal{R}}_y(L_{\fq})$.  Given $y \in
\scY(L_\fq)$, we write $\Pi_{\mathcal{R},\fq}(y)$ for the restriction
of $\Pi_{\fq}(\pr_{\scY \downarrow \scX}(y))$ to
$\widetilde{\mathcal{R}}_y(L_{\fq})$ and we denote by
$\Xi_{\mathcal{R},\fq}(y)$ the resulting irreducible character
of~$\exp \widetilde{\mathcal{R}}_y(L_{\fq})$.  As in the analogous
Definition~\ref{def:Pi.Xi}, the subscript $\fq$ is sometimes
omitted. Similarly, we write $\mathbf{S}_{\mathcal{R},\fq,y} =
\mathbf{S}_{\mathcal{R},\fq,x}$, where $x= \pr_{\scY \downarrow
  \scX}(y)$.

The following lemma is evident.

\begin{lem} \label{lem:submodule} Let $\mathfrak{O}$ be a complete,
  discrete valuation ring with a uniformizer~$\varpi$.  Let $M =
  \mathfrak{O}^n$ be a free $\mathfrak{O}$-module of rank $n \in \N$,
  and let $\overline{N} \subset M/\varpi M$ be a linear subspace with
  pre-image $N$ in~$M$. Assume that $T$ is an endomorphism of $M$
  which, in the standard basis, is given by a diagonal matrix with
  diagonal entries $\varpi ^{\gamma_1},\ldots,\varpi ^{\gamma_n}$,
  where $\gamma_1, \ldots, \gamma_n \in \Z$ satisfy $0 \leq \gamma_1
  \leq \ldots \leq \gamma_n$.  For~$l \in \Z$, let $i(l) = \max (\{0\}
  \cup \{ i \mid \gamma_i \leq l \})$ and $j(l) = \min (\{n+1\} \cup
  \{ j \mid \gamma_j \geq l \})$. Then the following hold.
  \begin{list}{}{\setlength{\leftmargin}{\myenumilistleftmargin}
      \setlength{\labelwidth}{20pt} \setlength{\itemsep}{0pt}
      \setlength{\parsep}{1pt}}
  \item[\textup{(1)}] The pre-image $T ^{-1}(\varpi^l N)$ is equal to
   \begin{multline*}
     \Big\{(a_1,\ldots,a_n) \in M \mid \forall i \in \{1,\ldots,i(l)\}:
     \val(a_i) \geq l-\gamma_i \\ \textup{and } (\overline{
         \varpi ^{\gamma_1-l}a_1},\ldots,\overline{\varpi
         ^{\gamma_{i(l)}-l}a_{i(l)}},0,\ldots,0)\in \overline{N}
     \Big\}.
   \end{multline*}
 \item[\textup{(2)}] The reduction of $T^{-1}(\varpi^l N)$ modulo
   $\varpi$ is the set of all $(\overline{a_1},\ldots,\overline{a}_n)
   \in (\mathfrak{O}/ \varpi \mathfrak{O})^n$ such that
   $\overline{a_1} = \ldots = \overline{a_{j(l)-1}}=0$ and
   $(0,\ldots,0,\overline{a_{j(l)}},\ldots,\overline{a_{i(l)}},0,\ldots,0)
   \in \overline{N}$.
 \end{list}
\end{lem}

\begin{prop} \label{prop:q.f.Lie.algebras.of.stabilizers} Let
  $\mathcal{R} \colon \scY \rightarrow
  \Grass^\mathrm{nilp}(\mathfrak{g}_\mathsf{k})$ be a quantifier-free
  definable function.  Then there is a quantifier-free definable
  function $\mathcal{L} \colon \scY \rightarrow
  \Grass(\mathfrak{g}_\mathsf{k})$ such that, for every finite
  extension $K \subset L$, almost all primes $\fq$ of~$O_L$, and every
  $y \in \scY(L_\fq)$, the Lie algebra of the stabilizer
  $\mathbf{S}_{\mathcal{R},\fq,y}$ is given by $\mathcal{L}(y)$.
\end{prop}

\begin{proof} For every $y \in \scY(L_\fq)$, the Lie algebra of the
  stabilizer $\mathbf{S}_{\mathcal{R},\fq,y}$ is the sum of
  $\mathcal{R}(y)$ and the Lie algebra of the stabilizer of
  $\Pi_{\mathcal{R},\fq}(y)$ modulo the maximal ideal.  Thus it
  suffices to prove that there is a quantifier-free definable function
  $\mathcal{L} \colon \scY \rightarrow
  \Grass(\mathfrak{g}_\mathsf{k})$ such that the image
  $\mathcal{L}(y)$ of $y \in \scY(L_\fq)$ is the Lie algebra of the
  stabilizer of $\Pi_{\mathcal{R},\fq}(y)$ modulo the maximal ideal.
  For ~$y = ((A_y,z_y), (\gamma_{y,1},\ldots,\gamma_{y,\dim
    \mathfrak{g}}), (U_{y,1},U_{y,2})) \in \scY(L_\fq)$, this is the
  intersection of the normalizer of $\mathcal{R}(y)$, which is given
  by a quantifier-free function, and the reduction modulo the maximal
  ideal of
  \begin{align*}
    \mathcal{V}_y(L_\fq) & = \left\{Y \in
      \mathfrak{g}_\mathcal{O}(L_\fq) \mid \forall X \in
      (\widetilde{\mathcal{R}})_y(L_\fq) \;
      \big( \Pi_{\mathcal{R},\fq}(y)([X,Y]) = 1 \big) \right\} \\
    & = \left\{Y \in \mathfrak{g}_\mathcal{O}(L_\fq) \mid \forall X
      \in (\widetilde{\mathcal{R}})_y(L_\fq) \;
      \big( \val(\langle A_y, [X,Y] \rangle) \geq \val(z_y) \big) \right\} \\
    & = \left\{Y \in \mathfrak{g}_\mathcal{O}(L_\fq) \mid \forall X
      \in (\widetilde{\mathcal{R}})_y(L_\fq) \; \big( \val(\langle
      X,[A_y,Y] \rangle ) \geq \val(z_y) \big) \right\}.
  \end{align*}
  The set $\mathcal{V}_y$ can be interpreted as the fiber of a
  definable set $\cv{\mathcal{V}} \subset \scY \times
  \mathfrak{g}_\mathcal{O}$, and it remains to prove that the
  reduction of $\mathcal{V}_y$ is quantifier-free.  Let
  $\mathcal{R}^\perp(y)$ be the orthogonal subspace to
  $\mathcal{R}(y)$ in $\mathfrak{g}_\mathsf{k}$ with respect to the
  form $\langle \cdot , \cdot \rangle$.  Recall that by omitting
  finitely many $\fq$, we arranged that $\langle \cdot , \cdot
  \rangle$ is non-degenerate on $\mathfrak{g}_\mathsf{k}(L_\fq)$.  We
  observe that $\mathcal{V}_y(L_\fq)$ is the pre-image of $z_y \,
  \big( (\widetilde{\mathcal{R}^\perp})_y(L_\fq) \big)$ under the map
  $\ad(A_y)$. According to the definition of $\scY$, we have $\ad(A_y)
  = U_{y,1}^{-1} T_y U_{y,2}^{-1}$, where $U_{y,1}, U_{y,2} \in
  \Aut(\mathfrak{g})_\mathcal{O}(L_\fq)$ and $T_y$ is diagonal with
  respect to the standard basis.  Thus $U_{y,2}^{-1}(V(L_\fq))$ is the
  pre-image of $z_y \, U_{y,1} \big(
  (\widetilde{\mathcal{R}^\perp})_y(L_\fq) \big)$ under~$T_y$.  By
  Lemma~\ref{lem:submodule}, statement~(2), its reduction modulo the
  maximal ideal is quantifier-free.  Hence, the reduction of
  $\mathcal{V}_y$ is quantifier-free.
\end{proof}

\begin{cor} \label{cor:q.f.Lie.algebras.of.stabilizers} For every
  quantifier-free definable function $\mathcal{R} \colon \scX
  \rightarrow \Grass^\mathrm{nilp}(\mathfrak{g}_\mathsf{k})$, there is
  a quantifier-free definable function $\mathcal{L} \colon \scX
  \rightarrow \Grass(\mathfrak{g}_\mathsf{k})$ such that for every
  finite extension $K \subset L$, almost all primes $\fq$ of $O_L$,
  and every $x\in \scX(L_\fq)$, the Lie algebra of the stabilizer
  $\mathbf{S}_{\mathcal{R},\fq,x}$ is given by $\mathcal{L}(x)$.
\end{cor}

\begin{proof} Pre-composing $\mathcal{R}$ with the projection
  $\pr_{\scY \downarrow \scX} \colon \scY \rightarrow \scX$ we get a
  quantifier-free definable function $\mathcal{R}' \colon \scY
  \rightarrow \Grass^\mathrm{nilp}(\mathfrak{g}_\mathsf{k})$.
  Applying Proposition~\ref{prop:q.f.Lie.algebras.of.stabilizers}, we
  obtain a quantifier-free definable function $\mathcal{L}' \colon
  \scY \rightarrow \Grass(\mathfrak{g}_\mathsf{k})$ such that, for all
  $y\in \scY(L_\fq)$, the vector space $\mathcal{L}'(y)$ is the Lie
  algebra of the stabilizer $\mathbf{S}_{\mathcal{R},\fq,y}$.  Since
  $\mathcal{L}'(y)$ depends only on $\pr_{\scY \downarrow \scX}(y)$,
  we get a definable function $\mathcal{L} \colon \scX \rightarrow
  \Grass(\mathfrak{g}_\mathsf{k})$ such that $\mathcal{L}(x)$ is the
  Lie algebra of the stabilizer $\mathbf{S}_{\mathcal{R},\fq,x}$ for
  all $x\in\scX(L_\fq)$.

  It suffices to show that $\mathcal{L}$ is quantifier-free definable.
  By an analogue of Lemma~\ref{lem:criterion.for.q.f.} for valued
  fields (cf.\ Remark~\ref{rem:criterion.q.f.perfect}), it is enough to
  show that if $F \subset E$ is an extension of valued fields, $x \in
  \scX(F)$, and $v \in \Grass(\mathfrak{g}_\mathsf{k})(F)$, then
  $\mathcal{L}(x) = v$ holds in $F$ if and only if $\mathcal{L}(x) =
  v$ holds in~$E$.  Since the map $\scY \rightarrow \scX$ is onto, we
  can choose $y \in \scY(F) \subset \scY(E)$ such that $\pr_{\scY
    \downarrow \scX}(y) = x$.  Then the following assertions are
  pairwise equivalent: $\mathcal{L}(x) = v$ holds in $F$;
  $\mathcal{L}'(y) = v$ holds in $F$; $\mathcal{L}'(y) = v$ holds in
  $E$ (because $\mathcal{L}'$ is quantifier-free); $\mathcal{L}(x) =
  v$ holds in~$E$.
\end{proof}

\begin{prop} \label{prop:size.coadjoint.orbit} Let
  $\mathcal{R}_1,\mathcal{R}_2 \colon \scX \rightarrow
  \Grass^\mathrm{nilp}(\mathfrak{g}_\mathsf{k})$ be quantifier-free
  definable maps and assume that, for every $x \in \scX$, the Lie
  algebra $\mathcal{R}_1(x)$ is normalized by $\mathcal{R}_2(x)$. Then
  there is a quantifier-free definable function $\phi \colon \scX
  \rightarrow \mathsf{\Gamma}$ such that, for every finite extension
  $K \subset L$, almost all primes $\fq$ of $L$, and every
  $x\in\scX(L_\fq )$,
  \[
  \left| \Ad^*(\exp
    \widetilde{\mathcal{R}_2}(x))(\Pi_{\mathcal{R}_1,\fq}(x)) \right|
  = \lvert O_L/\fq \rvert^{\phi(x)}.
  \]
\end{prop}

\begin{proof} Similarly to the proof of
  Proposition~\ref{prop:q.f.Lie.algebras.of.stabilizers}, the first
  claim of Lemma~\ref{lem:submodule} gives a quantifier-free function
  $\phi_1 \colon \scY \rightarrow \mathsf{\Gamma}$ such that, for
  every $y \in \scY(L_\fq)$,
  \[
  \left|\Ad^*({\exp
      \widetilde{\mathcal{R}_2}(\pr(y))})(\Pi_{\mathcal{R}_1
      \pr,\fq}(y))\right| = \lvert O_L/\fq \rvert^{\phi_1(y)},
  \]
  where $\pr = \pr_{\scY \downarrow \scX} \colon \scY \rightarrow
  \scX$ is the projection.  Since $\phi_1(y)$ depends only on the image
  of $y$ in $\scX$, we get a definable function $\phi_2\colon \scX
  \rightarrow \mathsf{\Gamma}$ such that
  \[
  \left|
    \Ad^*(\exp\widetilde{\mathcal{R}_2}(x))(\Pi_{\mathcal{R}_1,\fq}(x))\right|
  = \lvert O_L/\fq \rvert^{\phi_2(x)}.
  \]
  An argument similar to the one in
  Corollary~\ref{cor:q.f.Lie.algebras.of.stabilizers} shows that
  $\phi_2$ is a quantifier-free definable function, so we can
  take~$\phi=\phi_2$.
\end{proof}

\begin{rem}
  We can now indicate a proof of part~(1) of Theorem~\ref{thm:Jaikin}.

  We write $x=(A_x,z_x)$.  Then $\Pi_{\left\{
      0\right\}}^{-1}(\Pi_{\left\{ 0\right\}}(x))$ consists of the
  pairs $(B,w)$ such that, for all $X \in \mathfrak{g}^1(O_{L,\fq})$,
  \[
  \val\left( \left \langle \frac{A_x}{z_x}-\frac{B}{w}, X\right
    \rangle \right)>0,
  \]
  or, equivalently, such that $\val(A_x w - B z_x) \geq \val(z_x
  w)$.  The claim follows.
\end{rem}

%%%

\subsection{Proof of
  Theorem~\ref{thm:int.approx}}\label{subsec:proof.int.approx}

We continue to work in the same set-up as in the previous sections and
 recall that, if $\mathcal{S} \subset \scX \times
\mathbf{G}_\mathsf{k}$ is a definable family over $\scX$, then
$\widetilde{\mathcal{S}} \subset \scX \times \mathbf{G}_\mathcal{O}$
denotes the definable set of all pairs $(x,g)$ such that the reduction
of $g$ to $\mathbf{G}_\mathsf{k}$ lies in $\mathcal{S}_x$.

\begin{thm} \label{thm:partial.int} There are, for $n \in \N_ 0$,
  quantifier-free definable functions
  \[
  g_n,h_n \colon \scX \rightarrow \mathsf{\Gamma}, \qquad
  \mathcal{R}_n\colon \scX \rightarrow
  \Grass^\mathrm{nilp}(\mathfrak{g}_\mathsf{k}), \qquad \mathcal{L}_n\colon
  \scX \rightarrow \Grass(\mathfrak{g}_\mathsf{k}),
  \]
  constants $C_n\in\R$, and definable families $\mathcal{S}_n \subset
  \scX \times \mathbf{G}_\mathsf{k}$ of subgroups of
  $\mathbf{G}_\mathsf{k}$ such that the following hold.
  \begin{list}{}{\setlength{\leftmargin}{\myenumilistleftmargin}
      \setlength{\labelwidth}{20pt} \setlength{\itemsep}{0pt}
      \setlength{\parsep}{1pt}}
  \item[\textup{(1)}] $\mathcal{R}_0$ is the constant function
    $\{0\}$, $\mathcal{L}_0$ is the constant function
    $\mathfrak{g}_\mathsf{k}$, and $\mathcal{S}_0 = \scX \times
    \mathbf{G}_\mathsf{k}$.
  \item[\textup{(2)}] For $n \geq 1$, every finite extension $K
    \subset L$, almost all primes $\fq$ of $O_L$, and every $x \in
    \scX(L_\fq)$, the following hold:
    \begin{list}{$\circ$}{\setlength{\leftmargin}{\mylistleftmargin}
        \setlength{\labelwidth}{10pt} \setlength{\itemsep}{0pt}
        \setlength{\parsep}{1pt}}
    \item $(\widetilde{ \mathcal{S}_n})_x$ is the stabilizer of $\Xi_{
        \mathcal{R}_{n-1}}(x)$ in $(\widetilde{
        \mathcal{S}_{n-1}})_x$,
    \item $\mathcal{L}_n(x)$ is the Lie algebra of
      $(\mathcal{S}_n)_x$, and
    \item $ \mathcal{R}_n(x)$ is the Lie subalgebra of nilpotent
      matrices in the solvable radical of $ \mathcal{L}_n(x)$.
    \end{list}
  \item[\textup{(3)}] There exists $n_0 \in \N$ such that the
    sequences $g_n,h_n,\mathcal{R}_n,\mathcal{L}_n,\mathcal{S}_n$
    stabilize for $n \geq n_0$.
  \item[\textup{(4)}] For every $n$, every finite extension $K \subset
    L$ , and almost all primes $\fq$ of $O_L$,
    \[
    \zeta_{\mathbf{G}(O_{L,\fq})}(s) - \zeta_{\mathbf{G}(O_L/\fq)}(s)
    \sim_{C_n} \int_{\scX(L_\fq)} |O_L/\fq|^{g_n(x) - h_n(x)s}
    \zeta_{(\widetilde{\mathcal{S}_n})_x(L_\fq)|
      \Xi_{\mathcal{R}_n}(x)}(s) \, d \lambda(x).
    \]
  \end{list}
\end{thm}

\begin{proof} We first construct $\mathcal{R}_n, \mathcal{L}_n,
  \mathcal{S}_n$ using recursion on~$n \in \N_0$.  Suitable functions
  $g_n, h_n$ will be obtained in a second step.  The functions
  $\mathcal{R}_0, \mathcal{L}_0$ and the family $\mathcal{S}_0$ are
  prescribed by~(1).  Suppose $\mathcal{R}_n, \mathcal{L}_n,
  \mathcal{S}_n$ have been constructed.  The discussion at the
  beginning of Section~\ref{subsec:the.stabilizer} implies that there
  is a definable family of subgroups of $\mathbf{G}_\mathcal{O}$ over
  $\scX$ whose fiber at any $x \in \scX$ is the stabilizer of $\Xi_{
    \mathcal{R}_n}(x)$ in $( \widetilde{ \mathcal{S}_n} )_x$.  Take
  $\mathcal{S}_{n+1}$ to be the reduction of this family modulo the
  maximal ideal.  Similarly, we get $\mathcal{L}_{n+1}$, using
  Corollary~\ref{cor:q.f.Lie.algebras.of.stabilizers}, and
  $\mathcal{R}_{n+1}$, using Proposition~\ref{prop:grass.operations}.
  Thus (2) is taken care of.

  Next, we show that the sequences $\mathcal{R}_n, \mathcal{L}_n$, and
  $\mathcal{S}_n$, $n \in \N_0$, stabilize as required by~(3).  Note
  that the sequence $\dim \mathcal{R}_n$, $n \in \N_0$, is
  (pointwise) non-decreasing and the sequence $\dim \mathcal{L}_n$,
  $n \in \N_0$, is non-increasing.
  Proposition~\ref{prop:stab.conn.components} implies that, for
  every~$n \in \N_0$, there is an upper bound $D(n)$ for the number of
  connected components of each of the groups~$(\mathcal{S}_n)_x$.

  We claim that, if $n_1 \in \N_0$ is such that $\dim
  \mathcal{R}_{n_1}(x) = \dim \mathcal{R}_{n_1 + D(n_1)}(x)$ and $\dim
  \mathcal{L}_{n_1}(x) = \dim \mathcal{L}_{n_1 + D(n_1)}(x)$, then the
  sequences $\mathcal{R}_n(x),\mathcal{L}_n(x)$, and
  $\mathcal{S}_n(x)$, $n \in \N_0$, stabilize for $n > n_1 + D(n_1)$.
  Indeed, suppose that~$n \in \N_0$ with $n_1 \leq n < n_1 + D(n_1)$.
  If $\dim \mathcal{L}_n(x) = \dim \mathcal{L}_{n+1}(x)$, then
  $\mathcal{L}_n(x) = \mathcal{L}_{n+1}(x)$ and similarly for
  $\mathcal{R}_n(x)$.  Since $(\mathcal{S}_{n+1})_x$ is a subgroup of
  $(\mathcal{S}_n)_x$ and they have the same Lie algebra, either
  $\mathcal{S}_n(x) = \mathcal{S}_{n+1}(x)$ or $\mathcal{S}_{n+1}(x)$
  has fewer connected components than $\mathcal{S}_n(x)$.  It follows
  that there is $n \in \N_0$ with $n_1 \leq n < n_1 + D(n_1)$ such
  that $\mathcal{S}_n(x) = \mathcal{S}_{n+1}(x)$.  It now follows that
  the sequences of functions $\mathcal{R}_n, \mathcal{L}_n$, and
  $\mathcal{S}_n$ stabilize for sufficiently large~$n \in \N_0$.  Once
  $\mathcal{R}_n$ and $\mathcal{S}_n$ stabilize, we can keep also the
  functions $g_n$ and $h_n$ unchanged.  Thus (3) is satisfied.

  It remains to construct $g_n$ and $h_n$ for $n \in \N_0$ so that (4)
  holds.  We start with $n=0$.  Fix a finite extension $K \subset L$,
  and consider primes $\fq$ of $O_L$ that satisfy, in particular, the
  conditions of Lemma~\ref{lem:good.subgroup}.  For short we write $X
  = \scX(L_\fq)$, and we put
  \[
  X' = \bigsqcup \left\{ X_\theta \mid \theta \in
    \mathfrak{g}^{(1)}(O_{L,\fq})^\vee \smallsetminus \{ 1 \}
  \right\},
  \]
  where $X_\theta = \{x \in X \mid \Pi_{\{0\},\fq}(x) = \theta \}$ for
  every homomorphism $\theta$ from the additive group of the $1$st
  principal congruence Lie sublattice $\mathfrak{g}^{(1)}(O_{L,\fq })$
  to $\C^\times$.  Summing over $\theta \in
  \mathfrak{g}^{(1)}(O_{L,\fq})^\vee \smallsetminus \{ 1 \}$, applying
  the orbit method map $\Omega$, and using
  Lemma~\ref{lem:zeta.to.relative}, we obtain
  \begin{align*}
    \lefteqn{\zeta_{\mathbf{G}(O_{L,\fq})}(s) -
      \zeta_{\mathbf{G}(O_L/\fq)}(s)} \\
    & = \sum_{\theta \neq 1} \frac{1}{\lvert \Ad^*(\mathbf{G}(O_{L,
        \fq }))(\theta) \rvert} \, (\dim \Omega (\theta))^{-s} \,
    \zeta_{\mathbf{G}(O_{L,\fq }) \vert \Omega(\theta)}(s) \\
    & = \sum_{\theta \neq 1} \int_{X_\theta}
    \frac{1}{\lambda(X_\theta)} \, \frac{1}{\lvert
      \Ad^*(\mathbf{G}(O_{L, \fq })) (\Pi_{\{ 0 \}} (x)) \rvert} \,
    \dim( \Xi_{\{ 0 \}}(x))^{-s} \, \zeta_{\mathbf{G}(O_{L,\fq}) \vert
      \Xi_{\{ 0
        \}}(x)}(s) \, d \lambda(x) \\
    & = \int_{X'} \frac{1}{\lambda(\Pi_{\{ 0 \}}^{-1}(\Pi_{\{ 0
        \}}(x)))} \, \frac{1}{\lvert \Ad^*(\mathbf{G}(O_{L, \fq }))
      (\Pi_{\{ 0 \}} (x)) \rvert} \, \dim( \Xi _{\{ 0 \} }(x))^{-s} \,
    \zeta_{\mathbf{G}(O_{L,\fq }) \vert \Xi_{\{ 0 \}}(x)}(s) \, d
    \lambda(x).
  \end{align*}

  By Theorem~\ref{thm:Jaikin} and
  Corollary~\ref{cor:q.f.Lie.algebras.of.stabilizers}, there are
  quantifier-free definable functions $\phi_1,\phi_2,\phi_3
  \colon \scX \rightarrow \mathsf{\Gamma}$ such that, for almost all
  primes $\fq$ of $O_L$,
  \begin{list}{}{\setlength{\leftmargin}{\myenumilistleftmargin}
      \setlength{\labelwidth}{20pt} \setlength{\itemsep}{0pt}
      \setlength{\parsep}{1pt}}
  \item[\textup{(1)}] $\lambda(\Pi_{\{ 0 \}}^{-1} (\Pi_{\{ 0 \}}(x)))
    = \lvert O_L/\fq \rvert^{\phi_1(x)}$,
  \item[\textup{(2)}] $\dim \Xi_{\{ 0 \}}(x) = \lvert
    \Ad^*(\mathbf{G}^{(1)}(O_{L, \fq })) (\Pi_{\{ 0 \}} (x))
    \rvert^{1/2} = \lvert O_L/\fq \rvert^{\phi_2(x)}$,
  \item[\textup{(3)}] $\dim \mathcal{L}_1(x) = \lvert O_L/\fq
    \rvert^{\phi_3(x)}$.
  \end{list}
  
  By Lemma~\ref{lem:size.group}, there is a constant $C \in \R$ such
  that, for every $x \in X$,
  \begin{multline*}
    \lvert \Ad^*(\mathbf{G}(O_{L,\fq })) (\Pi_{\{ 0 \}}(x)) \rvert = \\
    \lvert \Ad^*(\mathbf{G}^{(1)}(O_{L,\fq })) (\Pi_{\{ 0 \}}(x))
    \rvert
    \cdot \lvert \mathbf{G}(O_{L}/\fq)/(\mathcal{S}_1)_x(L_\fq) \rvert
    \sim_C \lvert O_L/\fq \rvert^{2 \phi_2(x) + \dim \mathfrak{g} - \phi_3(x)}.
 \end{multline*}
 Therefore,
  \begin{multline*}
    \zeta_{\mathbf{G}(O_{L,\fq})}(s) - \zeta_{\mathbf{G}(O_L/\fq)}(s)
    \sim_C \\
    \int_{X'} \lvert O_L/\fq \rvert^{(-\phi_1(x) - 2 \phi_2(x) +
      \phi_3(x) - \dim \mathfrak{g}) - \phi_2(x)s} \,
    \zeta_{\mathbf{G}(O_{L,\fq}) \vert \Xi_{\{ 0 \}}(x)}(s) \, d
    \lambda (x),
  \end{multline*}
  and we set $g_0(x) = -\phi_1(x) - 2 \phi_2(x) + \phi_3(x) - \dim
  \mathfrak{g}$, $h_0(x) = \phi_2(x)$, and $C_0 = C$.

  Finally, we suppose that $g_n,h_n$ are given for some $n \in \N_0$
  and we explain how to construct $g_{n+1},h_{n+1}$.  Again, fix a
  finite extension $K \subset L$, and consider primes $\fq$ of $O_L$
  that are different from any of the finitely many primes omitted.

  Let $\mathfrak{a} \subset \mathfrak{g}(O_L/\fq)$ be a nilpotent Lie
  algebra, and recall that $\widetilde{\mathfrak{a}}$ denotes the
  pre-image of $\mathfrak{a}$ under the map $\mathfrak{g}(O_{L,\fq})
  \rightarrow \mathfrak{g}(O_L/\fq)$; it is a Lie ring, and, in
  particular, an additive group.  Let $\tau$ be a homomorphism from
  $\widetilde{\mathfrak{a}}$ to $\C^\times$, and consider the set
  \[
  X_{\mathfrak{a},\tau} = \Pi_{\mathcal{R}_n}^{-1}(\tau) = \{x \in
  \scX(L_\fq) \mid \mathcal{R}_n(x)=\mathfrak{a}, \Pi_{
    \mathcal{R}_n}(x) = \tau \}.
  \]
  Inductively, we see that on $X_{\mathfrak{a},\tau}$, the values of
  $\mathcal{R}_{0}(x), \mathcal{L}_{0}(x), \dots, \mathcal{R}_n(x),
  \mathcal{L}_n(x)$ and thus the values of $(\mathcal{S}_n)_x(L_\fq)$,
  $(\mathcal{S}_{n+1})_x(L_\fq)$, and $\mathcal{R}_{n+1}(x)$ are
  constant.  Denote the latter by $S_n,S_{n+1},$ and $\mathfrak{b}$
  respectively.  Writing $m = \lvert \mathfrak{b}:\mathfrak{a}
  \rvert$, let $\theta_1,\ldots,\theta_m$ denote the homomorphisms
  from $\widetilde{\mathfrak{b}}$ to $\C^\times$ that extend $\tau$.
  We define, for $i \in \{1,\ldots,m\}$,
  \[
  X_{\mathfrak{a},\tau}^i = \Pi_{\mathcal{R}_{n+1}}^{-1}(\theta_i) =
  \{x \in X_{\mathfrak{a},\tau} \mid \Pi_{\mathcal{R}_{n+1}}(x) =
  \theta_i\}.
  \]
  These sets form a partition of $X_{\mathfrak{a},\tau}$ into $m$
  parts of equal measure.  Using Lemma~\ref{lem:zeta.to.relative} and
  Clifford theory, we deduce that, for every $x \in
  X_{\mathfrak{a},\tau}$,
  \begin{multline*}
    \zeta_{(\widetilde{\mathcal{S}_n})_x(L_\fq) \vert
      \Xi_{\mathcal{R}_n}(x)}(s) = \zeta_{\widetilde{S_n} \vert \Omega
      (\tau)}(s) = \lvert \widetilde{S_n}:\widetilde{S_{n+1}}
    \rvert^{-s} \zeta_{\widetilde{S_{n+1}} \vert \Omega
      (\tau)}(s) \\=
    \lvert \widetilde{S_n}:\widetilde{S_{n+1}} \rvert^{-s}
    \sum_{i=1}^m \frac{\lvert \Ad^*(\exp
      \widetilde{\mathcal{R}_{n+1}}) (\tau) \rvert}{\lvert
      \Ad^*(\widetilde{S_{n+1}})(\theta_i) \rvert} \left( \frac{\dim
        \Omega (\theta_i)}{\dim \Omega (\tau)} \right)^{-s}
    \zeta_{\widetilde{S_{n+1}} \vert \Omega (\theta_i)}(s).
  \end{multline*}
  Therefore,
  \begin{multline*}
    \int_{X_{\mathfrak{a},\tau}}
    \zeta_{(\widetilde{\mathcal{S}_n})_x(L_\fq) \vert
      \Xi_{\mathcal{R}_n}(x)}(s) \, d \lambda(x) \\
    = \lvert \widetilde{S_n}:\widetilde{S_{n+1}} \rvert^{-s}
    \sum_{i=1}^m \int_{X_{\mathfrak{a},\tau}^i}
    \underbrace{\frac{\lambda(X_{\mathfrak{a},\tau})
      }{\lambda(X_{\mathfrak{a},\tau}^i)}}_{=m} \, \frac{\lvert
      \Ad^*(\exp \widetilde{\mathcal{R}_{n+1}}) (\tau) \rvert}{\lvert
      \Ad^*(\widetilde{S_{n+1}})(\theta_i) \rvert} \left( \frac{\dim
        \Omega (\theta_i)}{\dim \Omega (\tau)} \right) ^{-s}
    \zeta_{\widetilde{S_{n+1}} \vert \Omega (\theta_i)}(s) \, d
    \lambda (x) \hfill \\
    = \sum_{i=1}^m \int_{X_{\mathfrak{a},\tau}^i} \lvert
    (\mathcal{S}_n)_x(L_\fq):(\mathcal{S}_{n+1})_x(L_\fq) \rvert^{-s}
    \, \underbrace{\lvert
      \mathcal{R}_{n+1}(x):\mathcal{R}_n(x)\rvert}_{=m} \cdot
    \hfill \\
    \frac{\lvert \Ad^*(\exp \widetilde{\mathcal{R}_{n+1}})
      (\Pi_{\mathcal{R}_n}(x)) \rvert}{\lvert
      \Ad^*((\widetilde{\mathcal{S}_{n+1}})_x(L_\fq))(\Pi_{\mathcal{R}_{n+1}}(x))
      \rvert} \, \left( \frac{\dim \Xi_{\mathcal{R}_{n+1}}(x)}{\dim
        \Xi_{\mathcal{R}_n}(x)} \right) ^{-s}
    \zeta_{(\widetilde{\mathcal{S}_{n+1}})_x(L_\fq) \vert
      \Xi_{\mathcal{R}_{n+1}}(x)}(s) \, d \lambda (x).
  \end{multline*}

  Now there are quantifier-free definable functions
  $\psi_1,\psi_2,\psi_3,\psi_4 \colon \scX \rightarrow
  \mathsf{\Gamma}$ and a constant $C \in \R$ -- all independent of $L$
  and $\fq$ -- such that, for almost all primes $\fq$ of $O_L$, and
  all $x\in\scX(L_\fq)$, the following hold:
  \begin{list}{}{\setlength{\leftmargin}{\myenumilistleftmargin}
      \setlength{\labelwidth}{20pt} \setlength{\itemsep}{0pt}
      \setlength{\parsep}{1pt}}
  \item[\textup{(1)}] $\lvert (\mathcal{S}_n)_x(L_\fq):
    (\mathcal{S}_{n+1})_x(L_\fq) \rvert \sim_C \lvert O_L/\fq
    \rvert^{\psi_1(x)}$, by
    Corollary~\ref{cor:q.f.Lie.algebras.of.stabilizers} and
    Lemma~\ref{lem:size.group} (1), because the number of connected
    components of $(\mathcal{S}_n)_x$ is bounded by $D(n)$,
  \item[\textup{(2)}] $\lvert \mathcal{R}_{n+1}(x):\mathcal{R}_n(x)
    \rvert = \vert O_L/\fq \vert ^{\psi_2(x)}$, as $\mathcal{R}_n$ and
    $\mathcal{R}_{n+1}$ are quantifier-free definable,
  \item[\textup{(3)}] $\lvert \Ad^*(\exp
    \widetilde{\mathcal{R}_{n+1}}) (\Pi_{\mathcal{R}_n}(x)) \rvert /
    \lvert
    \Ad^*((\widetilde{\mathcal{S}_{n+1}})_x(L_\fq))(\Pi_{\mathcal{R}_{n+1}}(x))
    \rvert \sim_C \vert O_L/\fq \vert ^{\psi_3(x)}$, as
    \begin{multline*}
      \lvert
      \Ad^*((\widetilde{\mathcal{S}_{n+1}})_x(L_\fq))(\Pi_{\mathcal{R}_{n+1}}(x))
      \rvert = \\
      \lvert
      \Ad^*(\exp\widetilde{\mathcal{R}_{n+1}}(x))(\Pi_{\mathcal{R}_{n+1}}(x))
      \rvert \cdot \lvert (\mathcal{S}_{n})_x(L_\fq):
      (\mathcal{S}_{n+1})_x(L_\fq) \rvert,
    \end{multline*}
    by part (1), and by Proposition~\ref{prop:size.coadjoint.orbit},
  \item[\textup{(4)}] $\dim \Xi_{\mathcal{R}_{n+1}}(x) / \dim
    \Xi_{\mathcal{R}_n}(x) \sim_C \vert O_L/\fq \vert ^{\psi_4(x)}$
    because, for example, $\dim \Xi_{\mathcal{R}_n}(x) = \lvert
    \Ad^*(\exp\widetilde{\mathcal{R}_{n}}(x))
    (\Pi_{\mathcal{R}_{n}}(x)) \rvert ^{1/2}$ and by
    Proposition~\ref{prop:size.coadjoint.orbit}.
  \end{list}
  
  Writing $\alpha_n = \psi_2 + \psi_3$ and $\beta_n = \psi_1 +
  \psi_4$, we obtain
  \[
  \int_{X_{\mathfrak{a},\tau}} \zeta_{(\widetilde{S_n})_x(L_\fq)
    \vert \Xi_{\mathcal{R}_n}(x)}(s) \, d \lambda(x) \sim_{C^3}
  \int_{X_{\mathfrak{a},\tau}} \lvert O_L/\fq \rvert ^{\alpha_n(x) -
    \beta_n(x)s} \, \zeta_{(\widetilde{\mathcal{S}_{n+1}})_x(\fq)
    \vert \Xi_{\mathcal{R}_{n+1}}(x)}(s) \, d \lambda (x).
  \]

  Defining $g_{n+1} = g_n + \alpha_n$ and $h_{n+1} = h_n+ \beta_n$ we
  obtain, from the corresponding properties of $g_n,h_n$,
  \begin{align*}
    \zeta_{\mathbf{G}(O_{L,\fq})}(s) - \zeta_{\mathbf{G}(O_L/\fq)}(s)
    & \sim_{C_n} \int_{\scX(L_\fq)} \lvert O_L/ \fq \rvert ^{g_n(x) -
      h_n(x)s} \, \zeta_{(\widetilde{\mathcal{S}_n})_x(L_\fq) \vert
      \Xi_{\mathcal{R}_n}(x)}(s) \, d \lambda(x) \\
    & = \quad \sum_{\mathfrak{a},\tau}
    \int_{X_{\mathfrak{a},\tau}} \lvert O_L/ \fq \rvert^{g_n(x) -
      h_n(x)s} \, \zeta_{(\widetilde{\mathcal{S}_n})_x(L_\fq) \vert
      \Xi_{\mathcal{R}_n}(x)}(s)\, d \lambda(x) \\
    & \sim_{C^3} \sum_{\mathfrak{a},\tau}
    \int_{X_{\mathfrak{a},\tau}} \lvert O_L/ \fq \rvert^{g_{n+1}(x)
      - h_{n+1}(x) s} \,
    \zeta_{(\widetilde{\mathcal{S}_{n+1}})_x(L_\fq) \vert
      \Xi_{\mathcal{R}_{n+1}}(x)}(s) \, d \lambda(x) \\
    & = \quad \int_{\scX(L_\fq)} \vert O_L/ \fq \vert ^{g_{n+1}(x) -
      h_{n+1}(x)s} \, \zeta_{(\widetilde{\mathcal{S}_{n+1}})_x(L_\fq)
      \vert \Xi_{\mathcal{R}_{n+1}}(x)}(s) \, d \lambda(x).
  \end{align*}
  Here $\mathfrak{a}$ ranges over the finite set of nilpotent Lie
  subalgebras of $\mathfrak{g}(O_L/\fq)$ and $\tau$ ranges over the
  countable set of characters of $\widetilde{\mathfrak{a}}$.  The
  required properties of $g_{n+1}, h_{n+1}$ hold for $C_{n+1} = C_n
  C^3$.
\end{proof}

We are now ready to prove Theorem~\ref{thm:int.approx}.

\begin{proof}[Proof of Theorem~\ref{thm:int.approx}] We continue to
  use the notation set up in this section.  In particular, let
  $g_n,h_n,\mathcal{R}_n,\mathcal{L}_n,\mathcal{S}_n,C_n$ be the
  sequences constructed in Theorem~\ref{thm:partial.int}.  Suppose
  $n_0 \in \N$ is sufficiently large so that
  $g_n,h_n,\mathcal{R}_n,\mathcal{L}_n,\mathcal{S}_n$ are stable for
  $n \geq n_0$.  Fix $n \geq n_0$.  By Theorem~\ref{thm:int.approx},
  for all finite extensions $K \subset L$ and almost all
  $\mathfrak{q}$,
  \begin{equation} \label{eq:pf.int.approx}
    \zeta_{\mathbf{G}(O_{L,\fq})}(s)-\zeta_{\mathbf{G}(O_L/\fq)}(s)
    \sim_{C_n} \int_{\scX(L_\fq)} \vert O_L/\fq \vert ^{g_n(x) -
      h_n(x)s}\cdot \zeta_{(\widetilde{\mathcal{S}_n})_x(L_\fq) \vert
      \Xi_{\mathcal{R}_n}(x)}(s) d \lambda(x).
  \end{equation}

  Recall that $(\mathcal{S}_n)_x \subset \mathbf{G}_{\mathsf{k}}$ is
  an algebraic group. Let $(\mathcal{S}_n)_x^\circ$ denote the
  connected component of the identity, and let
  $\widetilde{(\mathcal{S}_n)_x^\circ} \subset \mathbf{G}_\mathcal{O}$
  be the pre-image of $(\mathcal{S}_n)_x^\circ$ under the reduction
  map modulo the maximal ideal. It was shown in
  Proposition~\ref{prop:stab.conn.components} that there is a constant
  $D_1 \in \R$ such that, for almost all~$\fq$ and all $x \in
  \scX(L_{\fq})$,
  \[
   \vert (\widetilde{\mathcal{S}_n})_x(L_\fq) :
  \widetilde{(\mathcal{S}_n)_x^\circ}(L_\fq) \vert  \leq
   \vert (\mathcal{S}_n)_x: (\mathcal{S}_n)_x^\circ \vert  \leq D_1.
  \]
  By Lemma \ref{lem:fin.index.BK} (2), we get for almost all~$\fq$ and
  all $x \in \scX(L_{\fq})$,
  \[
  \zeta_{(\widetilde{ \mathcal{S}_n})_x(L_\fq)  \vert  \Xi_{
      \mathcal{R}_n}(x)} \sim_{D_1} \zeta_{\widetilde{
      (\mathcal{S}_n)_x^\circ}(L_\fq)  \vert  \Xi_{ \mathcal{R}_n}(x)}.
  \]
  For every prime $\fq$ and every $x\in\scX(L_\fq)$, the group $Q_x =
  (\mathcal{S}_n)_x^\circ / \exp \mathcal{R}_n(x)$ is a reductive
  group over the field $O_L/\fq$ of dimension at most~$\dim
  \mathbf{G}$.  Omitting finitely many primes $\fq$ of $O_L$ as
  specified in Observation~\ref{obs:almost-all-q}, none of the Schur
  multipliers of the groups
  $\widetilde{(\mathcal{S}_n)_x^\circ}(L_\fq)/ \exp
  \widetilde{\mathcal{R}_n(x)}(L_\fq) = Q_x (O_L/ \fq)$, for $x\in
  \scX(L_\fq)$, contains elements of order $\charac(O_L/ \fq)$.  Using
  Lemmas~\ref{lem:p.power} and~\ref{lem:zeta.relative.quotient}, we
  thus obtain
  \[
  \zeta_{(\widetilde{ \mathcal{S}_n)_x^\circ}(L_\fq) \vert \Xi_{
      \mathcal{R}_n}(x)}=\zeta_{Q_x(O_L/ \fq)}.  \] The Lie algebra of
  $Q_x$ is $\mathcal{L}_n(x)/\mathcal{R}_n(x)$. By
  Proposition~\ref{prop:grass.operations}, there is a finite,
  quantifier-free partition of $\scX$ such that, on each part, the
  absolute root system $\Phi_x$ of $Q_x/Z(Q_x)$ and the dimension of
  the center of $Q_x$ -- which can be read off from the Lie algebra of
  $Q_x$, cf.\ Proposition~\ref{prop:grass.operations} -- are
  constant. By Corollary~\ref{cor:finite.reductive.zeta.approx}, there
  is a constant $D_2 \in \R$ such that
  \[
  \zeta_{Q_x(O_L / \fq)}(s) \sim_{D_2} \vert O_L / \fq \vert ^{\dim
    Z(Q_x)}(1+ \vert O_L/\fq \vert ^{\rk \Phi_x - \vert \Phi_x^+ \vert
    s}).
  \]
  Setting $\mathscr{Z}=\scX \times \left\{ 0,1 \right\}$,
  \[
  f_1(z)=\begin{cases} g_n(x)+\dim Z(Q_x) &\text{ for }z=(x,0)\in\scZ,\\
    g_n(x)+\dim Z(Q_x)+\rk \Phi_x &\text{ for }z=(x,1)\in\scZ,
  \end{cases}
  \]
  and
  \[
  f_2(z)= \begin{cases} - h_n(x) &\text{ for } z=(x,0)\in\scZ, \\ -
    h_n(x) + \vert \Phi_x^+ \vert &\text{ for } z=(x,1)\in\scZ,
  \end{cases}
  \]
  we obtain, with $C = C_nD_1D_2$,
  \begin{align*}
    \lefteqn{\zeta_{\mathbf{G}(O_{L,\fq})}(s) -
      \zeta_{\mathbf{G}(O_L/\fq)}(s)} \quad \quad \\
    & \sim_C \int_{\scX(L_\fq)} \vert O_L/\fq \vert ^{g_n(x)+\dim
      Z(Q_x)+h_n(x)s}(1+ \vert O_L/\fq \vert ^{\rk \Phi_x - \vert
      \Phi_x^+ \vert s})d
    \lambda(x)\\
    & = \int_{\scZ(L_\fq)} \vert O_L/\fq \vert ^{f_1(z) - f_2(z)s}d
    \lambda(z).
  \end{align*}
  This completes the proof of Theorem~\ref{thm:int.approx}.
\end{proof}

%%%%%

\section{Quantifier-Free Integrals} \label{sec:q.f.integrals}

In this section we complete the proof of
Theorem~\ref{thm:zeta.approx}. The section's main result is
Theorem~\ref{thm:mot.int}, expressing the dependence of an integral
such as the one in Theorem~\ref{thm:int.approx} on the local field the
integral is interpreted at.  We state the precise result in
Section~\ref{subsec:uniform.formulae} and prove it in
Section~\ref{subsec:mot.int}.

Throughout this section, we fix a number field $K$ and work within the
first order language of valued fields together with a constant, of the
value field sort, for every element of $K$. We will use the theory
$\Th_{\mathrm{Hen},K,0}$ of Henselian valued fields over $K$ of
characteristic~$0$; cf.\
Definition~\ref{def:theories}.%Section~\ref{sec:valued.fields}.

\subsection{Uniform Formulae for Quantifier-Free
  Integrals} \label{subsec:uniform.formulae}

We make no notational
distinction between an algebraic variety and the corresponding functor
of points.

\begin{defn}
  Let $X$ be a smooth algebraic variety of dimension $n$ over $K$ and
  let $\omega$ be a regular differential $n$-form on $X$. For any
  local field $F$ containing $K$, the set $X(F)$ has the structure of
  a $p$-adic analytic manifold; cf.\ \cite[Part~II,
  Chapter~III]{Ser}. We define a measure $| \omega |_F$ on $X(F)$ as
  follows: given a compact open set $U \subset X(F)$, an open compact
  subset $W \subset F^n$, and an analytic diffeomorphism $f \colon
  U\rightarrow W$, we write
  \[
  f^* \omega = gd x_1 \wedge \cdots \wedge d x_n,
  \]
  for some function $g \colon W \rightarrow F$, and define
  \[
  \lvert \omega \rvert_F(U) = \int_W \lvert g(x) \rvert_Fd \lambda(x),
  \]
  where $\lvert \cdot \rvert_F$ is the normalized absolute value of
  $F$ and $\lambda$ is the Haar measure on $F^n$ normalized so that
  $\lambda(O_F^n)=1$. The assignment $U \mapsto \lvert \omega
  \rvert_F(U)$ extends uniquely to a non-negative (possibly infinite)
  Radon measure on $X(F)$, which we also denote by $\lvert \omega
  \rvert_F$. See \cite[Section~3.1]{AA} for further details.
\end{defn}

We now state the section's main theorem.

\begin{thm} \label{thm:mot.int} Let $K$ be a number field with ring of
  integers $O_K$, and let $X \subset \mathbb{A} ^M$ be a smooth affine
  $K$-variety, and $\omega$ a regular differential top form on
  $X$. Suppose that $\scZ \subset X \cap \mathcal{O} ^M$ is a
  quantifier-free definable set and that $f_1,f_2 \colon \scZ
  \rightarrow \mathsf{\Gamma}$ are quantifier-free definable
  functions. There exist $N \in \N$, quasi-affine $O_K$-schemes
  $\mathbf{W}_i$ and integers $\alpha_i,\beta_i\in\N_0$ and
  $n_i\in\N$, for $i \in \{1,\ldots,N\}$, and $A_{ij},B_{ij}\in\Z$,
  for $i \in \{1,\ldots,N\}$ and $j \in \{1,\ldots,n_i\}$, such that
  the following holds: for every finite extension $K \subset L$ and
  almost all primes $\mathfrak{q}$ of $O_L$,
    \begin{equation} \label{eq:thm.mot.int}
      \int_{\scZ(L_\mathfrak{q})} |O_L / \mathfrak{q} |^{f_1(z) -
        f_2(z)s} | \omega |_{L_\mathfrak{q}}=\sum_{i=1}^{N}
      |O_L/\mathfrak{q}|^{\alpha_i - \beta_i s}|\mathbf{W}_i(O_L/
      \mathfrak{q})| \cdot \prod_{j=1}^{n_i}\frac{| O_L / \mathfrak{q}
        |^{A_{ij}-B_{ij}s}}{1- | O_L / \mathfrak{q}
        |^{A_{ij}-B_{ij}s}},
  \end{equation}
  for every $s\in \mathbb{C}$ for which the integral on the left
  converges.
\end{thm}

Theorem~\ref{thm:mot.int} may possibly be deduced from
\cite[Proposition~4.5 and Proposition~10.10]{HK} in a similar way that
\cite[Theorem~1.3]{HK} is. In the next section we give a direct proof,
in terms of resolutions of singularities.

\subsection{Proof of Theorem~\ref{thm:mot.int}} \label{subsec:mot.int}

In this section we decorate the reduction, angular component and
valuation maps with a superscript to indicate the arity of their
domain and write, e.g., $\red^{\times n}$, $\ac^{\times n}$ and
$\val^{\times n}$ (cf. Section \ref{sec:valued.fields}).

\begin{defn} A rational polyhedral cone in $\mathbb{Q} ^n$ is the
  intersection of finitely many open or closed linear half-spaces,
  i.e.\ subsets of the form $\left\{ {x}\in \mathbb{Q} ^n \mid \langle
  {x},{v} \rangle > 0 \right\}$ or $\left\{{x}\in \mathbb{Q} ^n \mid
  \langle {x},{v} \rangle \geq 0 \right\}$, where ${v}\in \mathbb{Q}
  ^n$ and $\langle \,, \rangle$ denotes the usual inner product
  on~$\R^n$.
\end{defn}

We will reduce Theorem~\ref{thm:mot.int} to the following special
case:

\begin{case}
  There exist an $O_K$-scheme $\mathbf{X} \subset \mathbb{A}_{O_K}^M$
  such that the structure map $\mathbf{X} \rightarrow \Spec(O_K)$ is
  smooth and its non-empty fibers are irreducible of dimension $n$, an
  \'etale map
  $\underline{\xi}=(\underline{\xi}_1,\dots,\underline{\xi}_n) \colon
  \mathbf{X} \rightarrow \mathbb{A}_{O_K} ^n$, a rational polyhedral
  cone $\mathcal{D} \subset \mathbb{Q}_{\geq 0}^n$, an $O_K$-scheme
  $\mathbf{M} \subset (\mathbb{G}_\text{m}^n)_{O_K} \times
  \mathbf{X}$, and integers $a_1,\ldots,a_n,b_1,\ldots,b_n$ such that
  \begin{list}{}{\setlength{\leftmargin}{\myenumilistleftmargin}
      \setlength{\labelwidth}{20pt} \setlength{\itemsep}{0pt}
      \setlength{\parsep}{1pt}}
  \item[\textup{(1)}] $X=\mathbf{X} \times \Spec(K)$,
    i.e.\ $\mathbf{X}$ is an $O_K$-model of $X$,
  \item[\textup{(2)}] $\omega = \underline{\xi} ^*( d x_1 \wedge
    \cdots \wedge d x_n)$, where $x_1,\dots,\cv{x_n}$ are coordinates
    on $\mathbb{A}^n_{O_K}$,
  \item[\textup{(3)}] $f_1(z)=\sum_{i=1}^n a_i
    \val(\underline{\xi}_i(z))$,
  \item[\textup{(4)}] $f_2(z)=\sum_{i=1}^n b_i
    \val(\underline{\xi}_i(z))$,
  \item[\textup{(5)}] the quantifier-free definable set $\scZ$ is
    defined by the formula
    \[
    z \in X \cap \mathcal{O} ^M \,\wedge\, {\val}^{\times
      n}(\underline{\xi}(z))\in \mathcal{D} \,\wedge\, \left(
      {\ac}^{\times n}(\underline{\xi}(z)),{\red}^{\times M}(z)
    \right) \in \mathbf{M}_{\mathsf{k}}.
    \]
  \end{list}
\end{case}

\subsubsection{Proof of Theorem~\ref{thm:mot.int} in the Special Case}
Let $\mathbf{Z} \subset \mathbf{X}$ be the image of $\mathbf{M}$ under
the projection to $\mathbf{X}$. After reordering the coordinates of
the map $\underline{\xi}$ and passing, if necessary, to one of the
parts of a finite, quantifier-free partition of~$\mathbf{M}$, we can
assume that, for some $0 \leq t \leq n$, the functions
$\underline{\xi}_1,\ldots,\underline{\xi}_t$ vanish on $\mathbf{Z}
\times \Spec(K)$ and
$\underline{\xi}_{t+1},\ldots,\underline{\xi}_{n}$ are invertible on
$\mathbf{Z} \times \Spec(K)$. This implies that, for every finite
extension $K \subset L$ and almost all primes $\mathfrak{q}$ of $O_L$,
the functions $\underline{\xi}_1,\ldots,\underline{\xi}_t$ vanish on
$\mathbf{Z} \times \Spec(O_L/\mathfrak{q})$ and
$\underline{\xi}_{t+1},\ldots,\underline{\xi}_n$ are invertible on
$\mathbf{Z} \times \Spec(O_L/\mathfrak{q})$.  We may thus assume,
without loss of generality, that $\mathcal{D}\subset \mathbb{Q}_{>0}^t
\times \left\{ 0 \right\} ^{n-t}$.
%If $\mathcal{D}$ is
%disjoint from $\mathbb{Q}_{>0}^t \times \left\{ 0 \right\} ^{n-t}$,
%then the integral is zero. Hence, we can assume that $\mathcal{D}
%\subset \mathbb{Q}_{>0}^t \times \left\{ 0 \right\} ^{n-t}$.

For each $\mathrm{p} \in \mathbf{X}(O_L / \mathfrak{q})$, the map
$\underline{\xi}$ induces a measure-preserving diffeomorphism between
$\left( \left( {\red}^{\times M}\right) ^{-1} (\mathrm{p})\cap
  \mathbf{X}(L_\mathfrak{q}),| \omega |_{L_\mathfrak{q}}\right)$ and
$\left( \left( {\red}^{\times n}\right) ^{-1}
  (\underline{\xi}(\mathrm{p})), |d x_1 \wedge \cdots \wedge d
  x_n|_{L_\mathfrak{q}}\right)$. The image of $\left( {\red}^{\times
    M} \right) ^{-1} (\mathrm{p}) \cap\scZ(L_\mathfrak{q})$ under
$\underline{\xi}$ is the set of all $x = (x_1,\ldots,x_n)\in
O_{L,\mathfrak{q}}^n$ that satisfy
\begin{list}{}{\setlength{\leftmargin}{\myenumilistleftmargin}
    \setlength{\labelwidth}{20pt} \setlength{\itemsep}{0pt}
    \setlength{\parsep}{1pt}}
\item[\textup{(1)}] ${\val}^{\times n}(x)\in \mathcal{D} \cap
  \mathbb{Z}^n$,
\item[\textup{(2)}] ${\red}^{\times n}(x)=\underline{\xi}(\mathrm{p})$,  and
\item[\textup{(3)}] $\left( {\ac}^{\times n}(x),\mathrm{p} \right) \in
  \mathbf{M}(O_L / \mathfrak{q})$,
\end{list}
or, equivalently,
\begin{list}{}{\setlength{\leftmargin}{\myenumilistleftmargin}
    \setlength{\labelwidth}{20pt} \setlength{\itemsep}{0pt}
    \setlength{\parsep}{1pt}}
\item[\textup{(1)}'] ${\val}^{\times n}(x)\in \mathcal{D} \cap
  \mathbb{Z} ^n$,
\item[\textup{(2)}'] $\left( \ac(x_{t+1}),\ldots,\ac(x_{n})\right) =
  \left(
    \underline{\xi}_{t+1}(\mathrm{p}),\ldots,\underline{\xi}_{n}(\mathrm{p})
  \right)$,
\item[\textup{(3)}'] $\left( \ac(x_1),\ldots,\ac(x_t) \right) \in
  \mathbf{M}^{(\mathrm{p})}(O_L/\mathfrak{q})$,
\end{list}
where
\begin{multline*}
  \mathbf{M}^{(\mathrm{p})}(O_L/\mathfrak{q}) = \\
  \left\{(\cv{y_1,\dots,y_t})\in((O_L/\mathfrak{q})^\times)^t \mid (\cv{y_1,
    \dots, y_t}, \underline{\xi}_{t+1}(\mathrm{p}), \dots,
    \underline{\xi}_n(\mathrm{p}),
    \mathrm{p})\in\mathbf{M}(O_L/\mathfrak{q})\right\}.
\end{multline*}
Setting $\mathcal{W}(\mathrm{p}) \subset O_{L,\mathfrak{q}}^n$ to be
the set defined by the conjunction of these three conditions, we get
\begin{align*}
  \lefteqn{\int_{\left( {\red}^{\times M} \right) ^{-1} (\mathrm{p})}
    1_{\scZ(L_\mathfrak{q})}(z) | O_L / \mathfrak{q} |^{f_1(z) -
      f_2(z)s} |
    \omega |_{L_\mathfrak{q}}}\\
  &=\int_{\left( {\red}^{\times n} \right) ^{-1}
    (\underline{\xi}(\mathrm{p}))} 1_{\mathcal{W}(\mathrm{p})}(x) |
  O_L / \mathfrak{q} |^{\sum_{i=1}^n (a_i - b_is)\val(x_i)}| d x_1
  \wedge \cdots \wedge d x_n
  |_{L_\mathfrak{q}}\\
  & =\left| \mathbf{M}^{(\mathrm{p})}(O_L / \mathfrak{q}) \right|
  \cdot \sum_{\gamma \in \mathcal{D} \cap \mathbb{Z} ^n} | O_L /
  \mathfrak{q} |^{\sum_{i=1}^n (a_i - b_is)\gamma_i}.
\end{align*}
Therefore,
\begin{equation}
\begin{split}
  \lefteqn{\int_{\scZ(L_\mathfrak{q})} | O_L / \mathfrak{q} |^{f_1(z) -
      f_2(z)s} | \omega |_{L_\mathfrak{q}}} \\
  & =\int_{\mathbf{X}(O_{L,\mathfrak{q}})} 1_{\scZ(L_\mathfrak{q})}
  (z)| O_L / \mathfrak{q} |^{f_1(z) - f_2(z)s} | \omega
  |_{L_\mathfrak{q}} \\
  & =\sum_{\mathrm{p}\in \mathbf{X}(O_L/\mathfrak{q})} \int_{\left(
      \red ^{\times M} \right) ^{-1} (\mathrm{p})}
  1_{\scZ(L_\mathfrak{q})}(z) | O_L / \mathfrak{q} |^{f_1(z) -
    f_2(z)s} | \omega |_{L_\mathfrak{q}} \\
  &=\sum_{\mathrm{p} \in \mathbf{X}(O_L/ \mathfrak{q} )}
  \left|\mathbf{M}^{(\mathrm{p})}(O_L / \mathfrak{q})\right| \cdot
  \sum_{\gamma \in \mathcal{D}\cap \mathbb{Z} ^n} | O_L / \mathfrak{q}
  |^{\sum_{i=1}^n (a_i - b_is)\gamma_i}.
\end{split}  \label{equ:sumX}
\end{equation}
Let $\mathbf{W}$ be the fiber product $\mathbf{M}
\times_{(\mathbb{G}_\text{m}^{n-t})_{O_K} \times \mathbf{X}}
\mathbf{X}$, where the map $\mathbf{M} \rightarrow
(\mathbb{G}_\text{m}^{n-t})_{O_K} \times \mathbf{X}$ is the projection
onto the last $n-t+M$ coordinates, and the map $\mathbf{X} \rightarrow
(\mathbb{G}_\text{m}^{n-t})_{O_K} \times \mathbf{X}$ is
$(\underline{\xi}_{t+1},\ldots,\underline{\xi}_n) \times
\Id_{\mathbf{X}}$. As
\[
| \mathbf{W}(O_L/\mathfrak{q}) | =
\sum_{\mathrm{p}\in\mathbf{X}(O_L/\mathfrak{q})} |
\mathbf{M}^{(\mathrm{p})}(O_L/\mathfrak{q})|,
\]
\eqref{equ:sumX} implies that
\begin{equation}\label{equ:sum}
  \int_{\scZ(L_\mathfrak{q})} | O_L / \mathfrak{q} |^{f_1(z) - f_2(z)s} |
  \omega |_{L_\mathfrak{q}} =\left| \mathbf{W}(O_L /
    \mathfrak{q})\right| \cdot \sum_{\gamma \in \mathcal{D} \cap
    \mathbb{Z} ^n} | O_L / \mathfrak{q} |^{\sum_{i=1}^n
    (a_i - b_is)\gamma_i}.
\end{equation}
The set $\mathcal{D} \cap \mathbb{Z} ^n$ may be decomposed into a
finite, disjoint union of cosets of free monoids. We may thus replace
$\mathcal{D} \cap \mathbb{Z} ^n$ by such a coset. Sums as the ones
in~\eqref{equ:sum} over free monoids are just finite products of
geometric progressions of the form $\frac{|O_L/\mathfrak{q}|^{A -
    Bs}}{1-|O_L/\mathfrak{q}|^{A - BS}}$, for suitable numerical data
$A,B\in\Z$. Translating the monoid amounts to multiplying the relevant
product by a factor of the form $|O_L/\mathfrak{q}|^{\alpha - \beta s}$,
for suitable~$\alpha,\beta\in\Z$. This proves
Theorem~\ref{thm:mot.int} in the Special Case.

\subsubsection{Reduction to the Special Case} \label{subsec:red.s.c}
Let $X, \omega,\scZ,f_1,f_2$ be as in Theorem~\ref{thm:mot.int}. As $X$ is
smooth, the integral \eqref{eq:thm.mot.int} in the theorem is the sum
of the respective integrals over the irreducible components of
$X$. Hence, we can assume that $X$ is irreducible.

\begin{lem} \label{lem:q.f.definable.functions.to.Gamma} If $h \colon
  \mathcal{O} ^M \rightarrow \mathsf{\Gamma}$ is a quantifier-free
  definable function, then there exist $N\in\N$, a finite
  quantifier-free definable partition
  \[
  \mathcal{O}^M = \Omega_1 \sqcup \dots \sqcup \Omega_N,
  \]
  and, for each $i \in \{1,\dots,N\}$, finitely many polynomials
  $P_{i1},\ldots,P_{in_i}$ over $K$ and rational numbers
  $r_{i1},\ldots,r_{i{n_i}}$ such that $h|_{\Omega_i}=\sum_{j=1}^{n_i}
  r_{ij} \val \circ P_{ij}$.
\end{lem}

\begin{proof}
  By assumption, the graph of $h$ is a Boolean combination of
  quantifier-free formulae in $M+1$ variables $x_1,\ldots,x_M,\gamma$,
  where $x_i$ are valued field sort and $\gamma$ is value group
  sort. The quantifier-free formulae, in turn, are Boolean
  combinations of formulae of the form
  \[
  P(x)=0, \; Q(\ac(R_1(x)),\dots,\ac(R_{N'}(x)))=0, \text{ and } n_0
  \gamma +\sum_{j\in J} n_j \val(P_j(x))=0,
  \]
  where $x=(x_1,\dots,x_M)$, $J$ is a finite index set, $P$,
  $R_1,\dots,R_{N'}$, and $P_j$, $j\in J$, are polynomials over~$K$,
  $Q$ is a polynomial over
  $\ac(K)$, and $n_0$ and $n_j$, $j\in J$, are integers. This implies
  the claim.
\end{proof}

By Lemma~\ref{lem:q.f.definable.functions.to.Gamma}, we can assume
that the functions $f_1$ and $f_2$ have the form
\begin{equation*}
  f_1 =\sum_{k\in J} r_k \val \circ F_k, \quad f_2 =\sum_{k\in J}
  r_k' \val \circ F_k,
\end{equation*}
where $J$ is a finite indexing set, $r_k,r_k' \in \mathbb{Q}$, and
$F_k$ are regular functions on $X$. By definition, the definable set
$\scZ$ is a disjoint union of sets that are defined using formulas of
the form
\begin{equation} \label{eq:q.f.set} 
z \in X \cap \mathcal{O}^M \,\wedge\, \varphi({\ac}^{\times M'}(H(z)))
\,\wedge\, \psi({\val}^{\times M'}(H(z))) \,\wedge\, H'(z)=0,
\end{equation}
where $H,H' \colon X\rightarrow \mathbb{A}_K^{M'}$ are regular maps
defined over $K$, $\varphi$ is a quantifier-free formula in the
language of fields (in variables of sort $\mathsf{k}$), and $\psi$ is
a quantifier-free formula in the language of ordered groups (in
variables of sort $\mathsf{\Gamma}$). Hence, it is enough to prove the
theorem assuming that $\scZ$ is defined by a formula of the
form~\eqref{eq:q.f.set}.

If $H' \neq 0$, then the integral in \eqref{eq:thm.mot.int} is
zero. Hence, we can assume that $H'=0$. The definable set defined by
$\varphi$ is equivalent to the disjoint union of finitely many
quasi-affine varieties. Partitioning $\scZ$ according to these
varieties, we can assume that $\varphi$ defines a single quasi-affine
$K$-variety $V$. Since we are only interested in evaluating the
integral~\eqref{eq:thm.mot.int} over local fields with large residue
field characteristic, we may replace $V$ by one of its $O_K$-models,
which we denote by~$\mathbf{V}$. We apply a similar argument for
$\psi$: after passing to one of the parts of a finite, quantifier-free
partition of~$\scZ$, we can assume that $\psi$ defines a translation
of a rational polyhedral cone in $\mathbb{Q}_{\geq 0}^{M'}$ by a
vector of the form ${\val}^{\times M'}(e)$, where~$e\in K^{M'}$. For
every finite extension $K \subset L$ and almost all $\mathfrak{q} \in
\Spec(O_L)$, we have ${\val}^{\times M'}(e)=0$. We may thus assume
that $\psi$ defines a rational polyhedral cone~$\mathcal{C}\subset
\Q_{\geq 0}^{M'}$. In conclusion, we can assume that $\scZ$ is defined
by the formula
\begin{equation} \label{eq:q.f.set.2} z\in X \cap \mathcal{O} ^M
  \,\wedge\, {\ac}^{\times M'}(H(z)) \in \mathbf{V}\,\wedge\,
             {\val}^{\times M'}(H(z))\in \mathcal{C}.
\end{equation}

\subsubsection{Resolution of Singularities}\label{subsubsec:res.sing}
We write $H=(H_1,\dots,H_{M'})$, let $P = \prod_{i=1}^{M'}H_i \cdot
\prod_{k\in J}F_k $, and consider the divisor $D = \mathrm{div}(P\cdot
\omega)$, i.e.\ the union of the vanishing loci of $\omega$
and~$P$. By Hironaka's theorem on strict resolution of singularities
(cf.\ \cite[Definition~B.5.1]{AA}), applied to the divisor $D$, there
exist $m\in\N$ and a smooth variety $Y \subset X \times \mathbb{P}^m$
defined over $K$ such that the projection $\pi \colon Y \rightarrow X$
is birational, is an isomorphism above the complement of~$D$, and the
pullback of $D$ under $\pi$ is a divisor with normal crossings.

Denote the dimension of $X$ by $n$. By the definition of divisor with
normal crossings, there is an open cover $Y = \bigcup_{i\in I} U_i$ by
affine $K$-varieties $U_i$, for some finite index set~$I$, and, for
each~$i\in I$, there is an \'etale map $\xi_i \colon U_i \rightarrow
\mathbb{A} ^n$ such that $\pi ^{-1} (\supp(D)) \cap U_i$ is contained
in the pre-image under $\xi_i$ of the coordinate hyperplanes in
$\mathbb{A}^n$. The divisor of $\pi ^* \omega$ is supported on $\pi
^{-1} (D)$, and $\xi_i^*(d x_1 \wedge \cdots \wedge d x_n)$ is an
invertible top differential form. Hence, the function $\frac{\pi ^*
  \omega}{\xi_i ^*(d x_1 \wedge \cdots \wedge d x_n)}$ is regular and
its divisor is supported on $\pi ^{-1} (D)$. Therefore, there is a
regular function $\theta_{i}:U_i \rightarrow \mathbb{G}_\text{m}$ and
$a_{ij}\in \mathbb{Z}$, $j\in\{1,\dots,n\}$, such that, for $y\in
U_i$,
\[
\frac{\pi ^* \omega}{\xi_i ^*(d x_1 \wedge \cdots \wedge d
  x_n)}(y)=\theta_{i}(y)\cdot \prod_{j=1}^n (\xi_i(y))_j^{a_{ij}},
\]
For each $i\in I$, fix an embedding $U_i \subset \mathbb{A} ^N$ for
some $N$. For every finite extension $K \subset L$ and almost all
primes $\mathfrak{q}$ of $O_L$, the function $\theta_{i}$ is the
restriction of a polynomial in $N$ variables with coefficients in
$O_{L,\mathfrak{q}}$. In particular, its restriction to
$O_{L,\mathfrak{q}}^N$ has non-negative valuation. The same is true
for $1/\theta_{i}$ and so, for almost all $\mathfrak{q}$, the
restriction of $\val(\theta_{i})$ to $U_i(L_\mathfrak{q})\cap O_{L,\mathfrak{q}}^N$
is~0. Consequently, we obtain, for $y\in U_i(L_\mathfrak{q}) \cap
O_{L,\mathfrak{q}}^N$,
\[
\val \left( \frac{\pi ^* \omega}{\xi_i ^*(d x_1 \wedge \cdots \wedge d
  x_n)} (y) \right)=\sum_{j=1}^n a_{ij}\val(\xi_i(y)_j).
\]
Similarly there are, for $j\in \{1,\dots,n\}$ and
$t\in\{1,\dots,M'\}$, integers $b_{ij},c_{ij},d_{itj}$ and functions $\eta_{it}:U_i
\rightarrow\mathbb{G}_\text{m}$, such that, for $y\in U_i(L_\mathfrak{q})\cap O_{L,\mathfrak{q}}^N$,
\begin{equation*}
f\circ \pi(y)=\sum_{j=1}^n b_{ij}\val(\xi_i(y)_j),\quad g\circ
\pi(y)=\sum_{j=1}^n c_{ij}\val(\xi_i(y)_j),
\end{equation*}
and, for $t\in\{1,\dots,M'\}$,
\[
\val(H_t\circ\pi(y))=\sum_{j=1}^n d_{itj}\val(\xi_i(y)_j),\quad
\ac(H_t\circ \pi(y))=\ac(\eta_{it}(y)) \cdot \prod_{j=1}^n
\ac(\xi_i(y)_j)^{d_{itj}}.
\]

\subsubsection{Reduction Modulo $\mathfrak{q}$}

Fix a total ordering $<$ on the finite index set $I$. For $i\in I$ set
$Z_i=U_i \smallsetminus \bigcup_{j<i}U_j \hookrightarrow \mathbb{A}
^N$, so that $Y = \bigcup_{i\in I}U_i = \bigsqcup_{i\in I}Z_i$, the
latter union being disjoint. Further choose $O_K$-models $\mathbf{X}
\subset \mathbb{A}_{O_K}^M$, $\mathbf{Y} \subset \mathbf{X} \times
\mathbb{P}_{O_K}^m$, and $\mathbf{Z}_i \subset \mathbb{A}_{O_K}^N$ for
the varieties $X$, $Y$, and $Z_i$, for~$i\in I$.  We also fix
$O_K$-models for the maps $\eta_{it}$, for $t\in\{1,\dots,M'\}$, and
$\xi_i$ which -- by slight abuse of notation -- we continue to denote
by these letters.

\begin{lem} \label{lem:X}
  For every finite extension $K \subset L$ and almost all primes
  $\mathfrak{q}$ of~$O_L$, $$\pi ^{-1}
  (\mathbf{X}(O_{L,\mathfrak{q}})) = \bigsqcup_{i\in I} \red ^{-1}
  (\mathbf{Z}_i(O_L / \mathfrak{q})).$$
\end{lem}

\begin{proof}
  We first contend that $\pi ^{-1}
  (\mathbf{X}(O_{L,\mathfrak{q}}))=\mathbf{Y}(O_{L,\mathfrak{q}})$ for
  almost all~$\mathfrak{q}$.  One containment follows from the
  projectivity of $\pi$, the other follows from the fact that, for
  almost all~$\mathfrak{q}$, the map $\pi$ is defined
  over~$O_{L,\mathfrak{q}}$.  Next, we claim that, for almost
  all~$\mathfrak{q}$, the sets $\mathbf{Z}_i(O_L / \mathfrak{q})$ form
  a partition of $\mathbf{Y}(O_L / \mathfrak{q})$. Indeed, we know
  that $\mathbf{Z}_i(\mathbb{C})$ form a partition of
  $\mathbf{Y}(\mathbb{C})$, so the $\mathbf{Z}_i((O_L /
  \mathfrak{q})^{\alg})$, $i\in I$, form a partition of
  $\mathbf{Y}((O_L / \mathfrak{q})^{\alg})$ for almost
  all~$\mathfrak{q}$, using Robinson's Principle. Since $\mathbf{Z}_i$
  and $\mathbf{Y}_i$ are quantifier-free, we get that
  $\mathbf{Z}_i(O_L/ \mathfrak{q})$ form a partition of
  $\mathbf{Y}(O_L / \mathfrak{q})$ for these primes~$\mathfrak{q}$.
  Altogether this yields, for almost all $\mathfrak{q}$,
  \begin{multline*}
    % \pi ^{-1} (\mathbf{X}(O_{L,\mathfrak{q}}))=
    \mathbf{Y}(O_{L,\mathfrak{q}}) = \bigsqcup_{\mathrm{p}\in
      \mathbf{Y}(O_L / \mathfrak{q})} \red ^{-1} (\mathrm{p}) =
    \bigsqcup_{i\in I} \bigsqcup_{\mathrm{p}\in
      \mathbf{Z}_i(O_L/\mathfrak{q})} \red ^{-1} (\mathrm{p}) =
    \bigsqcup_{i\in I} \red ^{-1} (\mathbf{Z}_i(O_L/\mathfrak{q})).
  \end{multline*}
\end{proof}
We deduce from Lemma~\ref{lem:X} that, for every finite extension $K
\subset L$ and almost all primes $\mathfrak{q}$ of~$O_L$,
\begin{align*}
  \lefteqn{\int_{\scZ(L_\mathfrak{q})} |O_L / \mathfrak{q}
    |^{f_1(z) - f_2(z)s} | \omega |_{L_\mathfrak{q}} =
    \int_{\mathbf{X}(O_{L,\mathfrak{q}})} 1_{\scZ(L_\mathfrak{q})}(x)
    |O_L / \mathfrak{q} |^{f_1(z) - f_2(z)s} | \omega |_{L_\mathfrak{q}}}\\
  &= \int_{\pi ^{-1} (\mathbf{X}(O_{L,\mathfrak{q}}))}
  1_{\scZ(L_\mathfrak{q})}(\pi (y)) |O_L / \mathfrak{q} |^{f_1(\pi
    (y)) - f_2(\pi (y))s} | \pi ^*\omega |_{L_\mathfrak{q}}\\ &=
  \sum_{i\in I} \int_{\red_\mathfrak{q} ^{-1} (\mathbf{Z}_i(O_L /
    \mathfrak{q}))} 1_{\pi ^{-1} (\scZ)(L_\mathfrak{q})}(y) |O_L /
  \mathfrak{q} |^{f_1(\pi (y)) - f_2(\pi (y))s} | \pi ^*\omega
  |_{L_\mathfrak{q}}.
\end{align*}
Recall from Section~\ref{subsec:red.s.c} the rational polyhedral cone
$\mathcal{C} \subset \mathbb{Q}_{\geq 0}^{M'}$ and the $O_K$-model
$\mathbf{V}$ of the quasi-affine $K$-variety $V$, featuring in the
definition~\eqref{eq:q.f.set.2} of the quantifier-free definable set
$\scZ$, as well as the various data defined in
Section~\ref{subsubsec:res.sing}. For each~$i\in I$, let
$\mathcal{D}_i \subset \mathbb{Q}^n_{\geq 0}$ be the rational
polyhedral cone defined by
\[
(\gamma_1,\ldots,\gamma_n)\in\mathcal{D}_i \iff \left( \sum_{j=1}^n
  d_{itj} \gamma_j \right)_{t=1}^{M'} \in \mathcal{C},
\]
let $\mathbf{M}_i \subset (\mathbb{G}_\text{m}^n)_{O_K}\times
\mathbf{Y}$ be the $O_K$-scheme defined by
\[
(x_1,\ldots,x_n,y)\in \mathbf{M}_i \iff \left(
  \eta_{it}(y)\prod_{j=1}^n x_j^{d_{itj}} \right)_{t=1}^{M'} \in
\mathbf{V},
\]
and let $\scZ_i$ be the definable set defined by
\[
y\in U_i \cap \mathcal{O} ^{M+m} \,\wedge\, \val^{\times n}(\xi_i(y))\in
\mathcal{D}_i \,\wedge\, \left( \ac^{\times n}(\xi_i(y)),\red^{\times (M+m)}(y) \right)\in
\left(\mathbf{M}_i\right)_{\mathsf{k}}.
\]
Then
\begin{multline*}
  \int_{\scZ(L_\mathfrak{q})} |O_L / \mathfrak{q} |^{f_1(z) - f_2(z)s}
  | \omega |_{L_\mathfrak{q}} =\\\sum_{i\in I}
  \int_{\scZ_i(L_\mathfrak{q})}| O_L / \mathfrak{q} |^{\sum_{j=1}^n
    (a_{ij}+b_{ij} - c_{ij}s)\val(\xi_i(y)_j)}| \xi_i ^* (d x_1 \wedge
  \cdots \wedge d x_n)|_{L_\mathfrak{q}},
\end{multline*}
and each summand on the right hand side is of the form covered by the
Special Case. This concludes the proof of Theorem~\ref{thm:mot.int}.

\subsection{Proof of
  Theorem~\ref{thm:zeta.approx}} \label{subsec:proof.zeta.approx}
%\begin{proof}[Proof of Theorem~\ref{thm:zeta.approx}]
Let $\Phi$ be the absolute root system of $\mathbf{G}$.  We show that
the assertions of the theorem hold for
\[
c(\mathbf{G}) = a(\Phi) \cup b(\mathbf{G}),
\]
where $a(\Phi) \in \mathcal{A}^+$ is the element constructed in
Theorem~\ref{thm:BC.finite.new} and $b(\mathbf{G})$ is obtained as
follows.  By Theorem~\ref{thm:int.approx}, there are a quantifier-free
definable set $\scZ$, quantifier-free definable functions $f_1,f_2
\colon \scZ \rightarrow \mathsf{\Gamma}$, and a constant $C_1 \in \R$
such that, for every finite extension $K \subset L$ and almost all
primes $\fq$ of $O_L$,
\[
\zeta_{\mathbf{G}(O_{L,\fq})}(s)-\zeta_{\mathbf{G}(O_L/ \fq)}(s)
\sim_{C_1} \int_{\scZ(L_\fq)}|O_L/ \fq |^{f_1(z) -
  f_2(z)s}d\lambda(z).
\]
Furthermore, Theorem~\ref{thm:mot.int} gives that, for almost all
primes~$\fq$,
\begin{equation} \label{eq:local.minus.finite.motivic}
  \zeta_{\mathbf{G}(O_{L,\fq})}(s)-\zeta_{\mathbf{G}(O_L/ \fq)}(s)
  \sim_{C_1} \sum_{i=1}^N
  |O_L/\mathfrak{q}|^{\alpha_i - \beta_i s}|\mathbf{W}_i(O_L/ \fq ) |
  \cdot \prod_{j=1}^{n_i} \frac{|O_L/ \fq |^{A_{ij}-B_{ij}s}}{1-|O_L/
    \fq|^{A_{ij}-B_{ij}s}},
\end{equation}
where the $\mathbf{W}_1, \ldots, \mathbf{W}_N$ are quasi-affine
$O_K$-schemes and $n_i$, $A_{ij}, B_{ij},\alpha_i,\beta_i$ are
integers specified in Theorem~\ref{thm:mot.int}.  We can assume that
the generic fiber of each $\mathbf{W}_i$ is non-empty and irreducible,
and we set
\[
b(\mathbf{G}) = \left\{ \left(\alpha_i+\dim \mathbf{W}_i +
\sum\nolimits_{j=1}^{n_i} A_{ij} , -\beta_i -
\sum\nolimits_{j=1}^{n_i} B_{ij} \right) \mid 1 \leq i \leq N \right\}
\in \mathcal{A}^+.
\]

Let $Q \subset \Spec(O_L)$ denote the set of all primes $\fq$ such
that \eqref{eq:local.minus.finite.motivic} holds.  By the Lang--Weil
estimates~\cite{LW}, there is a constant $C_2 \in \R$ such that, for
each $i \in \{1,\ldots,N\}$ and for almost all $\fq \in Q$, either
$\mathbf{W}_i(O_L/ \fq) = \varnothing$ or
\[
\frac{1}{2} \leq \frac{|\mathbf{W}_i(O_L/ \fq )|}{|O_L/ \fq |^{\dim
    \mathbf{W}_i}} \leq C_2.
\]

For each $\fq \in Q$, there exists $i \in \{1,\dots,N\}$ such that
$\mathbf{W}_i(O_L/\fq) \neq \varnothing$, and we define
\begin{align*} 
  \beta_\fq & = \max \left\{ A_{ij} / B_{ij} \mid 1 \leq i \leq N,\; 1
    \leq j \leq n_i ,\; \mathbf{W}_i(O_L/ \fq) \neq \varnothing,\;
    B_{ij}\neq 0 \right\} \in \mathbb{Q}_{>0}, \\
  b_\fq & = \left\{ \big( \alpha_i+\dim
      \mathbf{W}_i+\sum\nolimits_{j=1}^{n_i} A_{ij}, -\beta_i -
      \sum\nolimits_{j=1}^{n_i} B_{ij}\big) \mid 1 \leq i \leq N
    \text{, } \mathbf{W}_i(O_L/ \fq) \neq \varnothing \right\} \in
  \mathcal{A} ^+.
\end{align*}
We observe that, for each prime $\fq \in Q$, the abscissa of
convergence of $\zeta_{\mathbf{G}(O_{L,\fq})}$ is equal to $\beta_\fq$
by~\eqref{eq:local.minus.finite.motivic} and $b_\fq \subset
b(\mathbf{G})$.

Let $\epsilon \in \R_{>0}$. Since $\lvert O_L / \fq \rvert \geq 2$ for
all $\fq \in \Spec(O_L)$, there is a constant $\delta(\epsilon) \in
\R_{>0}$ such that, for each $i \in \{1,\ldots,N\}$, each $\fq \in Q$,
and all $\sigma \in \R$ with $\sigma > \beta_\fq + \epsilon$,
\[
\delta(\epsilon) < \prod\nolimits_j \left( 1-| O_L/\fq |^{A_{ij} -
    B_{ij}\sigma} \right) \leq 1.
\]
From \eqref{eq:local.minus.finite.motivic} it follows that there is a
constant $C_3(\epsilon) \in \R$ such that, for every $q \in Q$,
\begin{equation} \label{eq:full.minus.finite}
  \zeta_{\mathbf{G}(O_{L,\fq})}-\zeta_{\mathbf{G}(O_L/\fq)}
  \sim_{C_3(\epsilon)} \xi_{b_\fq,|O_L/ \fq |} \quad \text{for $\sigma
    > \beta_\fq +\epsilon$.}
\end{equation}
By Theorem~\ref{thm:BC.finite.new} and Remark~\ref{rem:BC.finite},
there is a constant $C_4 \in \R$ such that, for almost all primes $\fq
\in \Spec(O_L)$, there is $a_\fq \subset a(\Phi)$ such that
\begin{equation} \label{eq:finite.minus.1} \zeta_{\mathbf{G}(O_L/
    \fq)}-1 \sim_{C_4} \xi_{a_\fq,| O_L / \fq |}.
\end{equation}

Combining~\eqref{eq:full.minus.finite} and~\eqref{eq:finite.minus.1},
we get that, for almost all primes $\fq \in Q$, the following holds:
for every $\epsilon \in \R_{>0}$,
\[
\zeta_{\mathbf{G}(O_{L,\fq})}-1 \sim_{2 \max\{C_3(\epsilon),C_4\}}
\xi_{a_\fq \cup b_\fq,| O_L / \fq |} \quad \textrm{for $\sigma >
  \beta_\fq +\epsilon$.}
\]
These primes form a co-finite subset $T(L) \subset \Spec(O_L)$, and
this proves the assertion (1) of the theorem.

Assertion (2) of the theorem is derived from the argument above as
follows. The set
$R_1(L) =\{ \fq \in \Spec(O_L) \mid a_\fq = a(\Phi) \}$ is a
Chebotarev set by Corollary~\ref{cor:BC.finite.new.cbt}.  Moreover,
the Chebotarev Density Theorem implies that $\{\fq \in Q \mid \forall
i \in \{1,\ldots,N\}: \mathbf{W}_i(O_L/ \fq) \neq \varnothing \}$ is a
Chebotarev set.  Hence $R_2(L)= \{ \fq \in Q \mid b_\fq =
b(\mathbf{G}) \}$ is a Chebotarev set.  It follows that $R(L)=R_1(L)
\cap R_2(L)$ is a Chebotarev set, in particular of positive analytic
density. This concludes the proof
of Theorem~\ref{thm:zeta.approx}.

%\end{proof}

%%%%%


\begin{thebibliography}{WWW}

% \bibitem{AS} Adams, S.; Sarnak, P.; {\em Betti numbers of congruence
%     groups.} With an appendix by Ze'ev Rudnick. Israel J. Math. 88
%   (1994), 31--72.

\bibitem{AA} Aizenbud, A.; Avni, N.; {\em Representation growth and
    rational singularities of the moduli space of local systems.}
  \nir{ To appear in Inventiones Math.,} preprint:
  \url{http://arxiv.org/abs/1307.0371}.

\bibitem{A} Avni, N.; {\em Arithmetic groups have rational
    representation growth.} Ann. of Math. 174 (2011), 1009--1056.

\bibitem{AKOV} Avni, N.; Klopsch, B.; Onn, U.; Voll, C.; {\em
    Representation zeta functions of compact $p$-adic analytic groups
    and arithmetic groups}. Duke Math. J. 163 (2013), 111--197.

\bibitem{AKOV2} Avni, N.; Klopsch, B.; Onn, U.; Voll, C.; \uri{{\em
      Similarity classes of integral $\mathfrak{p}$-adic matrices and
      representation zeta functions of groups of type
      $\mathsf{A}_2$}. To appear in Proc. London Math. Soc., preprint:
    \url{http://arxiv.org/abs/1410.4533}.}

\bibitem{BoLuRa} Bosch, S.; L\"utkebohmert, W.; Raynaud, M.; {\em
  N\'eron models.} Ergebnisse der Mathematik und ihrer Grenzgebiete,
  21. Springer-Verlag, Berlin Heidelberg, 1990.

\bibitem{Ca} Carter, R. W.; {\em Finite groups of Lie type. Conjugacy
  classes and complex characters.} Pure and Applied Mathematics (New
  York). A Wiley-Interscience Publication. John Wiley \& Sons, Inc.,
  New York, 1985.

\bibitem{CK} Chang, C. C.; Keisler, H. J.; {\em Model theory. Third
    edition.} Studies in Logic and the Foundations of Mathematics,
  73. North-Holland Publishing Co., Amsterdam, 1990.

\bibitem{Cha} Chatzidakis, Z.; {\em Th\'eorie des mod\`eles des corps
    finis et pseudo-finis}. Available at:
  \url{http://www.logique.jussieu.fr/~zoe/index.html}.

\bibitem{CvdDM} Chatzidakis, Z.; van den Dries, L.; Macintyre, A.;
  {\em Definable sets over finite fields.} J. reine angew. Math. 427
  (1992), 107--135.

\bibitem{CL} Cluckers, R.; Loeser, F.; {\em Constructable motivic
    functions and motivic integration.} Invent. Math. 173 (2008),
  23--121.

\bibitem{Atlas} Conway, J. H.; Curtis, R. T.; Norton, S. P.; Parker,
  R. A.; Wilson, R. A.; {\em Atlas of finite groups. Maximal subgroups
    and ordinary characters for simple groups. With computational
    assistance from J. G. Thackray.} Oxford University Press, Eynsham,
  1985.

% \bibitem{CLS} Cox, D. A.; Little, J. B.; Schenck, H. K.; {\em Toric
%     varieties.} Graduate Studies in Mathematics, 124. American
%   Mathematical Society, 2011.

% \bibitem{De} Denef, J.; {\em On the degree of Igusa's local zeta
%     function.} Amer. J. Math. 109 (1987), 991--1008.

\bibitem{De81} Deriziotis, D. I.; {\em Centralizers of semisimple
    elements in a Chevalley group.} Comm. Algebra 9 (1981),
  1997--2014.

\bibitem{DM} Digne, F.; Michel, J.; {\em Representations of finite
    groups of Lie type.} London Mathematical Society Student Texts,
  21. Cambridge University Press, Cambridge, 1991.

% \bibitem{dSG} du Sautoy, M.; Grunewald, F.; {\em Analytic properties
%     of zeta functions and subgroup growth.} Ann. of Math. 152 (2000),
%   793--833.

\bibitem{EGA4_2} Grothendieck, A.; {\em El\'ements de g\'eom\'etrie
    alg\'ebrique. IV. \'Etude locale des sch\'emas et des morphismes
    de sch\'emas. II.} Inst. Hautes \'Etudes Sci. Publ. Math. No. 24,
  1965.

\bibitem{EGA4} Grothendieck, A.; {\em El\'ements de g\'eom\'etrie
    alg\'ebrique. IV. \'Etude locale des sch\'emas et des morphismes
    de sch\'emas IV.} Inst. Hautes \'Etudes Sci. Publ. Math. No. 32,
  1967.

\bibitem{G} Gonz\'alez-S\'anchez, J.; {\em Kirillov's orbit method for
    $p$-groups and pro-$p$ groups.} Comm. Algebra 37 (2009),
  4476--4488.

\bibitem{GK} Gonz\'alez-S\'anchez, J.; Klopsch, B.; {\em Analytic
    pro-$p$ groups of small dimensions.} J. Group Theory 12 (2009),
  711--734.

\bibitem{Gr} Greenberg, M.J.; {\em Schemata over local rings.}
  Ann. Math. 73 (1961), 624--648.

% \bibitem{Ho} Howe, R.E.; {\em Kirillov theory for compact $p$-adic
%     groups.} Pac. J. Math. 73 (1977), 365--381.

\bibitem{HK} Hrushovski, E.; Kazhdan. D.; {\em Integration in valued
    fields}. Algebraic geometry and number theory, volume 253 of
  Progr. Math., Birkh\"auser Boston, Boston, MA, 2006.

\bibitem{Hu} Humphreys, J. E.; {\em Conjugacy classes in semisimple
    algebraic groups}. Mathematical Surveys and Monographs 43,
  American Mathematical Society, Providence, RI, 1995.

\bibitem{HuBl} Huppert, B.; Blackburn, N.; {\em Finite
    groups. II}. Grundlehren der Mathematischen Wissenschaften 242,
  Springer-Verlag, Berlin-New York, 1982.

\bibitem{Is} Isaacs, I. M.; {\em Character theory of finite groups.}
  AMS Chelsea Publishing, Providence, RI, 2006.

\bibitem{Jac} Jacobson, N.; {\em Lie algebras.} Dover Publications,
  Inc., New York, 1979.

\bibitem{Jai} Jaikin--Zapirain A.; {\em Zeta function of
    representations of compact $p$-adic analytic groups.}
  J. Amer. Math. Soc. 19 (2006), 91--118.

\bibitem{Kaz_MI} Kazhdan, D.; {\em Lecture notes in motivic
  integration.} available at:
  \url{http://www.math.huji.ac.il/~kazhdan/Notes/motivic/b.pdf}.

\bibitem{K} Klopsch B.; {\em On the Lie theory of $p$-adic analytic
    groups.} Math. Z. 249 (2005), 713--730.

\bibitem{LW} Lang, S.; Weil, A.; {\em Number of points of varieties in
    finite fields}. Amer. J. Math. 76 (1954), 819--827.

\bibitem{LL} Larsen, M.; Lubotzky, A.; {\em Representation growth of
    linear groups}. J. Eur. Math. Soc. 10 (2008), 351--390.

  % \bibitem{LP} Larsen, M.; Pink, R.; {\em Finite subgroups of
  %     algebraic groups}. J. Amer. Math. Soc. 24 (2011), 1105--1158.

\bibitem{LS} Liebeck, M.W.; Shalev, A.; {\em Character degrees and
    random walks in finite groups of Lie type.} Proc. London
  Math. Soc. 90 (2005), 61--86.

\bibitem{LM} Lubotzky, A.; Martin, B.; {\em Polynomial representation
    growth and the congruence subgroup problem.} Israel J. Math. 144
  (2004), 293--316.

\bibitem{LuNi} Lubotzky, A.; Nikolov, N.; {\em Subgroup growth of
    lattices in semisimple Lie groups.} Acta. Math. 193 (2004),
  105--139.

\bibitem{Lu} Lusztig, G.; {\em On the representations of reductive
    groups with disconnected centre.} Asterisque 168 (1988), 157--166.

\bibitem{Margulis} Margulis, G. A.; {\em Discrete subgroups of
    semisimple Lie groups.}  Ergebnisse der Mathematik und ihrer
  Grenzgebiete (3), 17. Springer-Verlag, Berlin, 1991.

  % \bibitem{Ma} Marker, D.; {\em Model theory. An introduction.}
  %   Graduate Texts in Mathematics, 217. Springer-Verlag, New York,
  %   2002.

\bibitem{MT} Malle, G.; Testerman, D.; {\em Linear algebraic groups
    and finite groups of {L}ie type.} Cambridge Studies in Advanced
  Mathematics, 133. Cambridge University Press, Cambridge, 2011.

\bibitem{N} Nori, M.; {\em On subgroups of $\mathrm{GL}_n(\F_p)$.}
  Invent.\ Math.\ 88 (1987), 257--275.

\bibitem{Pas} Pas, J.; {\em Uniform $p$-adic cell decompositions and
    local zeta functions.}  J. reine angew. Math 399 (1989), 137--172.

\bibitem{PR} Platonov, V.; Rapinchuk, A.; {\em Algebraic groups and
    number theory.}  Pure and Applied Mathematics, 139. Academic
  Press, Inc., Boston, MA, 1994.

\bibitem{PrRa10} Prasad, G.; Rapinchuk, A.; \textit{Developments on
    the congruence subgroup problem after the work of Bass, Milnor and
    Serre}, in: Collected papers of John Milnor. V:~Algebra (ed.\
  H.~Bass and T.~Y.~Lam), Amer.\ Math.\ Soc., Providence, RI, 2010.

\bibitem{Ra} Raghunathan, M.S.; {\em The congruence subgroup problem.}
  Proc.\ Indian Acad.\ Sci.\ Math.\ 114 (2004), no. 4, 299--308.

\bibitem{Se} Serre, J.-P.; {\em Local fields.} Graduate Texts in
  Mathematics, 67. Springer-Verlag, New York, 1979.

\bibitem{Ser} Serre, J.-P.; {\em Lie algebras and Lie groups.}  1964
  lectures given at Harvard University. Corrected fifth printing of
  the second (1992) edition. Lecture Notes in Mathematics,
  1500. Springer-Verlag, Berlin, 2006.

\bibitem{Tits} Tits, J.; {\em Algebraic and abstract simple groups.}
  Ann.\ of Math.\ (2) 80 (1964), 313--329.

\bibitem{Wi} Witten, E.; {\em On quantum gauge theories in two
    dimensions.} Comm. Math. Phys. 141 (1991), no.~1, 153--209.
\end{thebibliography}
\end{document}